\documentclass[a4paper]{amsart}
\usepackage{amsmath,amssymb,amsthm}
\usepackage{hyperref}
\usepackage{color}
\usepackage{graphicx}
\usepackage{fullpage}
\usepackage{setspace}
\usepackage{tikz}
\usepackage{tikz-cd}
\usepackage[normalem]{ulem}
\usepackage{soul}

\newtheorem{thm}{Theorem}
\newtheorem{lem}{Lemma}

\newtheorem{prop}{Proposition}

\theoremstyle{definition}
\newtheorem{example}{Example}

\newcommand{\cQ}{\mathcal{Q}}

\newcommand{\fin}{\ensuremath{\mathrm{fin}}}
\newcommand{\cPfin}{\ensuremath{\mathcal{P}}^\fin}

\newcommand{\vrleq}[2]{\mathrm{VR}_\le(#1;#2)}
\newcommand{\vrm}[2]{\mathrm{VR}^m(#1;#2)}

\newcommand{\vrmleq}[2]{\mathrm{VR}^m_\leq(#1;#2)}
\newcommand{\vrmcircleq}{\mathrm{VR}^m_\leq(S^1)}

\newcommand{\diam}{\mathrm{diam}}

\newcommand{\supp}{\mathrm{supp}}

\newcommand{\arcs}{\mathrm{arcs}}
\newcommand{\cost}{\mathrm{cost}}

\title{Vietoris--Rips Metric Thickenings of the Circle}
\author{Michael Moy}
\date{}

\begin{document}

\maketitle

\begin{abstract}

Vietoris--Rips metric thickenings have previously been proposed as an alternate approach to understanding Vietoris--Rips simplicial complexes and their persistent homology.
Recent work has shown that for totally bounded metric spaces, Vietoris--Rips metric thickenings have persistent homology barcodes that agree with those of Vietoris--Rips simplicial complexes, ignoring whether endpoints of bars are open or closed.
Combining this result with the known homotopy types and barcodes of the Vietoris--Rips simplicial complexes of the circle, the barcodes of the Vietoris--Rips metric thickenings of the circle can be deduced up to endpoints, and conjectures have been made about their homotopy types.
We confirm these conjectures are correct, proving that the Vietoris--Rips metric thickenings of the circle are homotopy equivalent to odd-dimensional spheres at the expected scale parameters.
Our approach is to find quotients of the metric thickenings that preserve homotopy type and show that the quotient spaces can be described as CW complexes.
The quotient maps are also natural with respect to the scale parameter and thus provide a direct proof of the persistent homology of the metric thickenings.

\end{abstract}

\section{Introduction}\label{section:introduction}

Vietoris--Rips simplicial complexes have gained attention recently because of their use in applied topology, and in particular, they provide a practical method of computing persistent homology.
This method is justified theoretically by stability theorems that show Vietoris--Rips persistent homology handles perturbations to the underlying space well~(\cite{chazalPersistenceStabilityGeometric2014}).
Despite their use and these stability theorems, there is still a limited theoretical understanding of the topology of Vietoris--Rips complexes, even for simple underlying spaces.
One significant discovery in this area was the homotopy types of the Vietoris--Rips complexes of the circle, given in~\cite{AA_circle}.
This built on previous work in~\cite{Adamaszek_clique_complexes}, which gave the homotopy types of finite numbers of evenly spaced points on the circle; along with the stability of persistent homology, these finite cases suggested reasonable conjectures for the Vietoris--Rips complexes of the circle.
This understanding of the Vietoris--Rips complexes of the circle has led to an improved interpretation of the persistent homology of more general spaces: \cite{Virk_footprints} shows that certain loops in a space may be detected by persistent homology, as they contribute persistent homology bars similar to those of the circle.

Other work in this area has improved the understanding of the homotopy types of Vietoris--Rips complexes of general $n$-spheres at low scale parameters: two distinct approaches of \cite{lim_memoli_okutan} and \cite{zaremsky} apply in a range of scale parameters where the homotopy type of the $n$-sphere is recovered. 
Furthermore, \cite{lim_memoli_okutan} describes how their results in fact improve upon those implied by Hausmann's theorem (\cite{Hausmann1995}), which states that in the more general setting of Riemannian manifolds, the Vietoris--Rips complexes recover the homotopy type of the manifold at low scale parameters (in a range depending on the manifold).

Vietoris--Rips metric thickenings, along with more general simplicial metric thickenings, were introduced in~\cite{AAF} and provide an alternate approach.
Their topology agrees with Vietoris--Rips complexes in the case of finite underlying spaces but can differ in general.
However, results in~\cite{Moy_Masters_Thesis} and~\cite{vrp_arxiv} show that, under reasonable conditions, the persistent homology barcodes of Vietoris--Rips metric thickenings agree with those of Vietoris--Rips complexes, ignoring differences of open or closed endpoints of bars.
This suggests a close topological relationship between these two constructions.
In particular, the barcodes of the Vietoris--Rips metric thickenings of the circle agree with those of the Vietoris--Rips complexes of the circle, which provides some evidence that the homotopy types should agree.
The homotopy types of the metric thickenings were in fact already conjectured based on the known homotopy types of the simplicial complexes: see for example Conjecture~6 of~\cite{ABF}.
The homotopy types of the metric thickenings of the circle have previously been found for only low scale parameters~\cite{AAF, ABF, AM}.
More generally, Theorem~5.4 of~\cite{AAF} identifies the first new homotopy type, $S^n*\frac{\mathrm{SO(n+1)}}{A_{n+2}}$, that appears in the filtration of Vietoris--Rips metric thickenings of any $n$-sphere.
The proof applies at the single lowest scale parameter at which the metric thickening is no longer homotopy equivalent to the $n$-sphere.

We will use the intuition provided by these previous results to find all homotopy types of the Vietoris--Rips metric thickenings of the circle.
The main results are the following homotopy types, given in  Theorem~\ref{thm:vrm_homotopy_equivalent_to_spheres}:
\[
\vrmleq{S^1}{r} \simeq \begin{cases}
S^{2k+1} & \text{ if $r \in \big[\frac{2k \pi}{2k+1}, \frac{(2k+2)\pi}{2k+3}\big)$} \\
\{\ast\} & \text{ if $r \geq \pi$}.
\end{cases}
\]
Our technique will first show the metric thickenings are homotopy equivalent to CW complexes, which provide a clear understanding of why these homotopy types appear.
These CW complexes provide a much simpler view of the metric thickenings and will in fact be constructed as quotients of the metric thickenings.
The quotients will identify each measure with a measure supported on an odd number of regularly spaced points on the circle; a large part of our work will be dedicated to constructing the quotient maps and showing they are homotopy equivalences.

The CW complexes reveal the homotopy types of odd-dimensional spheres as follows.
For $k \geq 0$ and $r \in \big[\frac{2k \pi}{2k+1}, \frac{(2k+2)\pi}{2k+3}\big)$ as shown above, there is one $n$-cell for each dimension $0 \leq n \leq 2k+1$.
The $1$-cell is glued by its two endpoints to the $0$-cell to create a circle, which should be viewed as the underlying copy of $S^1$.
For $k \geq 1$, the metric thickening contains measures supported on three evenly spaced points around the circle, which can be viewed as points of a triangle, so we obtain a subspace of triangular measures on a set of triangles parameterized by a circle.
The single two cell is represented by a single distinguished equilateral triangle and is glued by its boundary to the circle to produce a $2$-disk.
The remaining triangles are parameterized by an interval, so adding in the remaining triangular measures amounts to gluing in a $3$-cell.  
The boundary of this $3$-cell is glued to the contractible $2$-disk, producing a space homotopy equivalent to $S^3$.
Higher dimensional spheres appear similarly.
For the final two steps, a single distinguished $2k$-cell, represented by measures supported on a chosen set of $2k+1$ evenly spaced points, is glued into the previous $(2k-1)$-sphere to produce a $2k$-disk, and then a $(2k+1)$-cell is glued in by its boundary, giving a space homotopy equivalent to $S^{2k+1}$.
\begin{figure}[h]
    \centering
    \fontsize{9pt}{11pt}
    \def\svgwidth{.9\textwidth}
\begingroup%
  \makeatletter%
  \providecommand\color[2][]{%
    \errmessage{(Inkscape) Color is used for the text in Inkscape, but the package 'color.sty' is not loaded}%
    \renewcommand\color[2][]{}%
  }%
  \providecommand\transparent[1]{%
    \errmessage{(Inkscape) Transparency is used (non-zero) for the text in Inkscape, but the package 'transparent.sty' is not loaded}%
    \renewcommand\transparent[1]{}%
  }%
  \providecommand\rotatebox[2]{#2}%
  \newcommand*\fsize{\dimexpr\f@size pt\relax}%
  \newcommand*\lineheight[1]{\fontsize{\fsize}{#1\fsize}\selectfont}%
  \ifx\svgwidth\undefined%
    \setlength{\unitlength}{651.96850394bp}%
    \ifx\svgscale\undefined%
      \relax%
    \else%
      \setlength{\unitlength}{\unitlength * \real{\svgscale}}%
    \fi%
  \else%
    \setlength{\unitlength}{\svgwidth}%
  \fi%
  \global\let\svgwidth\undefined%
  \global\let\svgscale\undefined%
  \makeatother%
  \begin{picture}(1,0.47391304)%
    \lineheight{1}%
    \setlength\tabcolsep{0pt}%
    \put(0,0){\includegraphics[width=\unitlength,page=1]{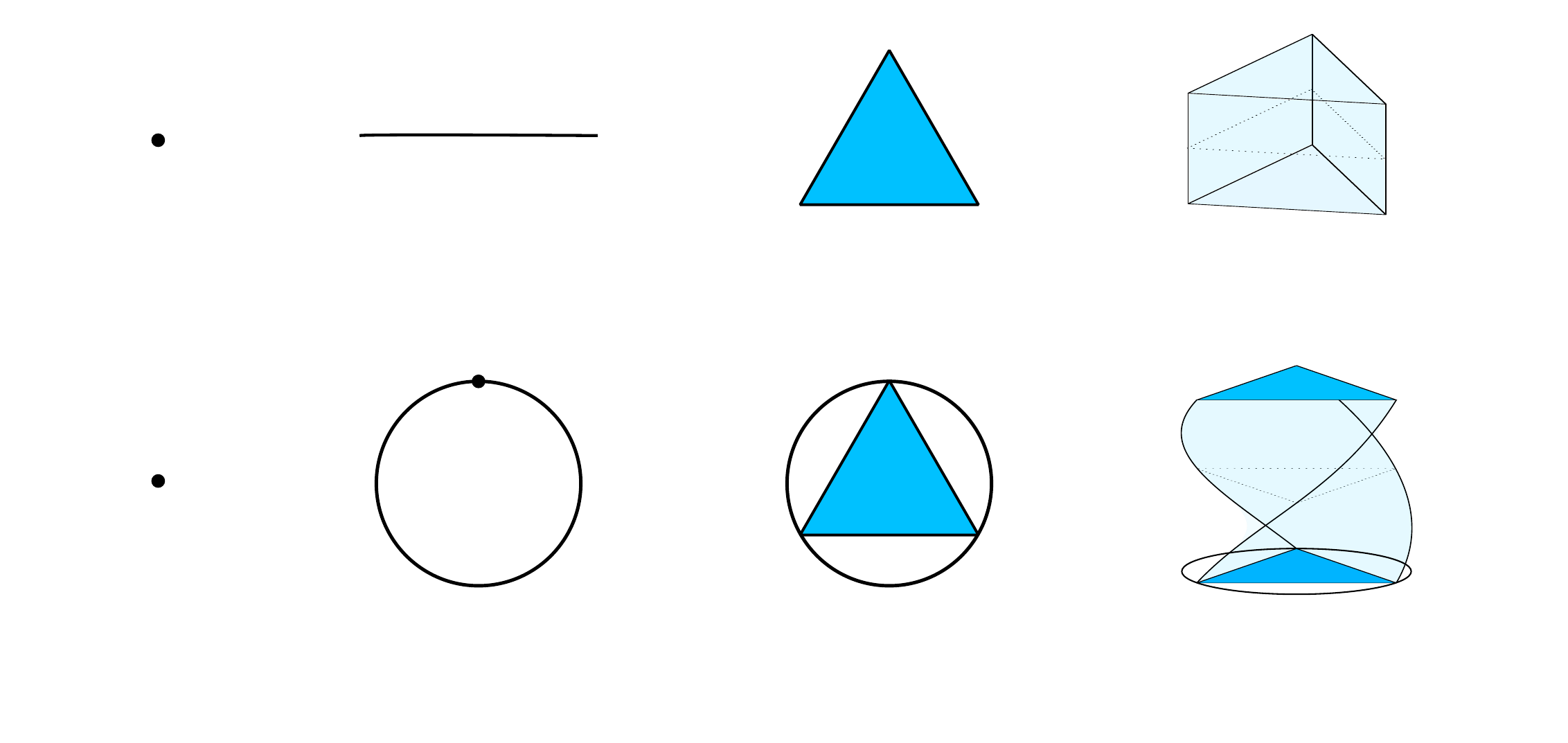}}%
    \put(0.08904177,0.02006408){\makebox(0,0)[lt]{\lineheight{1.25}\smash{\begin{tabular}[t]{l}$D^0$\end{tabular}}}}%
    \put(0.29599828,0.02006408){\makebox(0,0)[lt]{\lineheight{1.25}\smash{\begin{tabular}[t]{l}$S^1$\end{tabular}}}}%
    \put(0.55556352,0.02006408){\makebox(0,0)[lt]{\lineheight{1.25}\smash{\begin{tabular}[t]{l}$D^2$\end{tabular}}}}%
    \put(0.8173026,0.02006408){\makebox(0,0)[lt]{\lineheight{1.25}\smash{\begin{tabular}[t]{l}$S^3$\end{tabular}}}}%
  \end{picture}%
\endgroup%

    \caption{The CW complex giving the homotopy type of $S^3$ has one cell in dimensions 0, 1, 2, and 3.
    The 2-cell is a triangle and is glued by its boundary to the circle.
    The 3-cell is a triangular prism with cross sectional triangles corresponding to all equilateral triangles on the circle.
    Both of its triangular faces are glued to the 2-cell, with the top face rotated by $\frac{2 \pi }{3}$.
    The rectangular faces are collapsed to the circle.}
    \label{figure:cell_structure}
\end{figure}
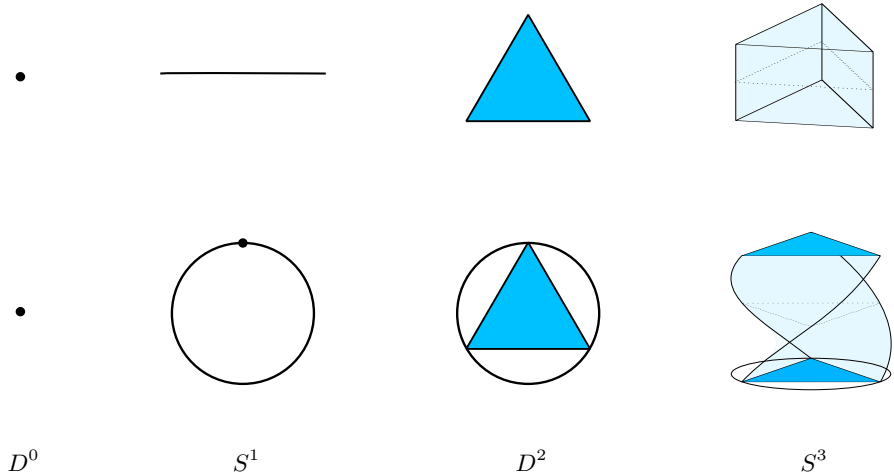

This approach of reducing to a CW complex is reminiscent of Morse theory and will hopefully contribute to the development of a more general Morse-like theory for simplicial metric thickenings.
Morse-theoretic ideas have previously been applied to related problems.
One such idea is described in~\cite{katz1991neighborhoods}, and this work was later found to be closely related to Vietoris--Rips complexes; see~\cite{lim_memoli_okutan}.
The main result of~\cite{katz1991neighborhoods} in fact implies the homotopy type of $S^3$ in the case of Vietoris--Rips complexes of the circle.
Discrete Morse theory has also been applied to the study of the Vietoris--Rips complexes of spheres: \cite{zaremsky} uses a version of discrete Morse theory to show the complexes recover the homotopy types of $n$-spheres at low scale parameters.

This paper is organized as follows.
In Section~\ref{section:notation}, we provide the relevant background information on Vietoris--Rips metric thickenings and give a useful technique for constructing homotopies in metric thickenings.
Section~\ref{section:arcs} describes properties of the Vietoris--Rips metric thickening of the circle that will be used throughout, many of which suggest the methods of the later sections.
In Section~\ref{section:homotopies_quotients_HEP}, we give background information and basic results related to quotients and the homotopy extension property, and we show that certain pairs of subspaces of the metric thickenings of the circle have the homotopy extension property.
Section~\ref{section:collapse} describes homotopies that collapse certain subspaces of the metric thickenings, and in Section~\ref{section:sequence_of_quotients}, we piece these homotopies together, defining a quotient map that is a homotopy equivalence.
In Section~\ref{section:cw_complex_and_homotopy_types}, we show this quotient has the CW complex structure described above and use the CW complex to find the homotopy types of the metric thickenings.
As a final result, in Section~\ref{section:persistent_homology}, we find the persistent homology of the metric thickenings directly, that is, without relying on any previous knowledge about the simplicial complexes.

\vspace{.15cm}
\section{Background, Notation, and Conventions}\label{section:notation}

We begin with an overview of Vietoris--Rips metric thickenings and the concepts from optimal transport used to define them.
Details and further properties sufficient for our purposes can be found in~\cite{AAF}, and more general constructions and technical discussions can be found in~\cite{vrp_arxiv}.

\subsection{Spaces of Probability Measures and the Wasserstein Distance}\label{subsection: spaces of probability measures and Wasserstein distance}

For any metric space $(X,d)$, we let $\cPfin(X)$ be the space of finitely supported probability measures on $X$ with the 1-Wasserstein distance.  
We write the 1-Wasserstein distance as $d_W$ for any space (and from Section~\ref{section:arcs} on, it will always be the 1-Wasserstein distance on subspaces of $\cPfin(S^1)$).
Each finitely supported probability measure $\mu \in \cPfin(X)$ can be written as a convex combination of delta measures: $\mu = \sum_{i=1}^n a_i \delta_{x_i}$ where $a_i \geq 0$ for each $i$, $\sum_{i=1}^n a_i = 1$, each $\delta_{x_i}$ is the delta measure at $x_i$, and $x_1, \dots, x_n \in X$. 
We can think of each $a_i$ as the amount of mass located at the point $x_i$.
We write the support of a measure $\mu \in \cPfin(X)$ as $\supp(\mu)$, so if $\mu = \sum_{i=1}^n a_i \delta_{x_i}$, we have $\supp(\mu) = \{ x_i \mid a_i > 0 \}$.

The 1-Wasserstein distance can be thought of in the language of optimal transport: we will imagine transporting the mass of one measure to the mass of another, where moving a mass of $m$ from $x$ to $x'$ costs $m \cdot d(x,x')$.
In our setting of finitely supported probability measures, the 1-Wasserstein distance has a simple description.
A transport plan between two measures $\mu = \sum_{i=1}^n a_i \delta_{x_i}$ and $\mu'= \sum_{j=1}^{n'} a'_j \delta_{x'_j}$ can be described by a \textit{matching}, an indexed set $\kappa = \{ \kappa_{i,j} \mid 1 \leq i \leq n, 1 \leq j \leq n' \}$ of nonnegative real numbers such that $\sum_{i=1}^n \kappa_{i,j} = a'_j$ for each $j$ and $\sum_{j=1}^{n'} \kappa_{i,j} = a_i$ for each $i$.
Each $\kappa_{i,j}$ represents the mass transported from $x_i$ to $x'_j$, and we will sometimes succinctly describe matchings by describing how mass is transported.
The cost of this matching is defined by $\cost(\kappa) = \sum_{i=1}^n \sum_{j=1}^{n'} \kappa_{i,j} d(x_i, x'_j)$.
The 1-Wasserstein distance between $\mu$ and $\mu'$ is then given by 
\[
d_W(\mu, \mu') = \inf_{\kappa}(\cost(\kappa)) = \inf_{\kappa} \left( \sum_{i=1}^n \sum_{j=1}^{n'} \kappa_{i,j} d(x_i, x'_j)\right),
\]
where the infimum is taken over all matchings $\kappa$ from $\mu$ to $\mu'$.
In this setting of finitely supported measures, the infimum is always attained (the space of matchings is a compact subset of $\mathbb{R}^{n \cdot n'}$), and a matching $\kappa$ from $\mu$ to $\mu'$ such that $\cost(\kappa) = d_W(\mu, \mu')$ will be called an \textit{optimal matching}.

A convenient way of working with topology of $\cPfin(X)$ comes from the relationship between the Wasserstein distance and weak convergence.
A sequence of measures $\{ \mu_n \}_{n \geq 0}$ in $\cPfin(X)$ is said to \textit{converge weakly} to $\mu \in \cPfin(X)$ if $\lim_{n \to \infty} \int_X f \,d \mu_n = \int_X f \,d \mu$ for all bounded and continuous $f \colon X \to \mathbb{R}$.
We record the following result in language that applies to our work: for a more general statement, see~\cite{villani2003topics}.

\begin{lem}\label{lemma:weak_convergence_and_Wasserstein_convergence}
Let $X$ be a Polish bounded metric space and suppose $\{ \mu_n \}_{n \geq 0}$ is a sequence of measures in $\cPfin(X)$.
Then $\{ \mu_n \}_{n \geq 0}$ converges weakly to $\mu \in \cPfin(X)$ if and only if $\lim_{n \to \infty} d_W(\mu_n, \mu) = 0$.
\end{lem}

\begin{proof}
This is a special case of Theorem~7.12 of~\cite{villani2003topics} (the case of bounded metric spaces is mentioned in Remark~7.13(iii) following the theorem).
\end{proof}

We will be interested in the case where $X = S^1$.
The lemma applies since $S^1$ is a Polish space: it is separable and is complete with respect to either the usual Euclidean metric or the geodesic metric, which we describe in Section~\ref{subsection: coordinates on the circle}.

\subsection{Vietoris--Rips Metric Thickenings and Support Homotopies}\label{subsection: vrm and support homotopies}

For $r \in \mathbb{R}$, the traditional Vietoris--Rips simplicial complex $\vrleq{X}{r}$ has all finite subsets of $X$ of diameter at most $r$ as simplices\footnote{We use the $\leq$ convention, which includes simplices that have a diameter of exactly $r$.
Also note that we use the same definition for all $r \in \mathbb{R}$, so that for $r<0$, $\vrleq{X}{r}$ is empty.}.
To define the Vietoris--Rips metric thickening, we replace the points of the simplices in this definition with the corresponding finitely supported probability measures.
That is, we define the \textit{Vietoris--Rips metric thickening} $\vrmleq{X}{r}$ to be the subspace of $\cPfin(X)$ consisting of all measures $\mu$ such that $\supp(\mu)$ has diameter at most $r$.
The topology of $\vrmleq{X}{r}$ is thus the metric topology given by the 1-Wasserstein distance\footnote{If $X$ is a bounded metric space, then any $p$-Wasserstein distance with $p \in [1, \infty)$ gives $\vrmleq{X}{r}$ the same topology: this is discussed (in more generality) in~\cite{vrp_arxiv}.}.
For all $r \geq 0$, the original metric space $X$ is isometrically embedded in $\vrmleq{X}{r}$ by the map $x \mapsto \delta_{x}$.
The corresponding inclusion from $X$ into the simplicial complex $\vrleq{X}{r}$ is not even necessarily continuous, as the vertices of a simplicial complex form a discrete subspace; this is an important difference between the topologies of metric thickenings and simplicial complexes.
We can extend the map $x \mapsto \delta_{x}$ above by taking convex combinations to obtain a bijection $\vrleq{X}{r} \to \vrmleq{X}{r}$, which is in fact continuous.
The inverse $\vrmleq{X}{r} \to \vrleq{X}{r}$, however, is not necessarily continuous ($X$ can be naturally embedded in $\vrmleq{X}{r}$ as above, but this is not necessarily true for $\vrleq{X}{r}$).

We now show that the metric thickening topology allows for a convenient way of constructing homotopies in $\vrmleq{X}{r}$.
In fact, we consider the more general setting of subspaces of $\cPfin(X)$.
The following lemma shows that if $X$ is bounded, then continuously deforming the supports of measures in a subset $U \subseteq \cPfin(X)$ results in a homotopy $U \times I \to X$, as long as all measures remain in $U$ as their supports are deformed.
We will allow the motion of a mass at a point $x$ in $\supp(\mu)$ to depend on both $x$ and $\mu$, so we begin with a homotopy on the subspace $\cQ(X,U) = \{ (x,\mu) \in X \times U \mid x \in \text{supp}(\mu)\} \subseteq X \times U$.  
We define a \textit{support homotopy} in $U$ to be a homotopy $H \colon \cQ(X,U) \times I \to X$ such that for any $t\in I$ and any $\mu = \sum_{i=1}^n a_i \delta_{x_i} \in U$ with $a_i>0$ for each $i$, we have $\sum_{i=1}^n a_i \delta_{H(x_i,\mu,t)} \in U$ (here we require $\mu$ to be written with each $a_i$ positive so that $x_i \in \supp(\mu)$, making $(x_i,\mu,t)$ in the domain of $H$).
The following lemma shows that a support homotopy in $U$ induces a homotopy $\widetilde{H} \colon U \times I \to U$ if $X$ is bounded.
The proof uses ideas similar to the proof of Lemma 5.2 of~\cite{AAF}.

\begin{lem}\label{lemma_support_homotopy}
Let $(X,d)$ be a bounded metric space and let $U \subseteq \cPfin(X)$.  
If $H \colon \cQ(X,U) \times I \to X$ is a support homotopy in $U$, then  $\widetilde{H} \colon U \times I \to U$ given by $\widetilde{H}(\mu,t) = \sum_{i=1}^n a_i \delta_{H(x_i,\mu,t)}$ for $\mu = \sum_{i=1}^n a_i \delta_{x_i}$ with $a_i > 0$ for each $i$, is well-defined and continuous.
\end{lem}

\begin{proof}
Up to reordering, there is a unique way to write a measure $\mu \in U \subseteq \cPfin(X)$ as $\mu = \sum_{i=1}^n a_i \delta_{x_i}$ with $a_i > 0$ for each $i$ and with $x_1, \dots, x_n$ distinct.
Note that if $x_1, \dots, x_n$ are not distinct, this does not affect the value of $\widetilde{H}(\mu,t)$, so $\widetilde{H}(\mu,t)$ is uniquely determined for each $(\mu,t) \in U \times I$.
Furthermore, by definition of a support homotopy, $\widetilde{H}$ does in fact send elements of $U \times I$ into $U$, so it is a well-defined function; we must show it is continuous.
Since $X$ is bounded, let $C > 0$ be such that $d(x,y)<C$ for any $x,y \in X$.  
Fix $t\in I$ and $\mu = \sum_{i=1}^n a_i \delta_{x_i} \in U$ with each $a_i > 0$.  
To show continuity of $\widetilde{H}$ at $(\mu,t)$, let $\varepsilon>0$.  
Using continuity of $H$ at the finitely many points $(x_1,\mu,t), \dots, (x_n,\mu,t)$, there exist $\eta_1, \eta_2, \eta_3 > 0$ such that for each $i$, if $(y,\mu',t') \in \cQ(X,U) \times I$ satisfies $d(x_i,y)<\eta_1$, $d_W(\mu,\mu')<\eta_2$, and $|t-t'|<\eta_3$, then $d(H(x_i,\mu,t),H(y,\mu',t'))< \frac{\varepsilon}{2}$.  
We will reduce $\eta_2$ if necessary so that $0<\eta_2<\frac{\varepsilon \eta_1}{2C}$.  

Suppose $(\mu',t') \in U \times I$ satisfies $d_W(\mu,\mu')<\eta_2$ and $|t-t'|<\eta_3$, where $\mu' = \sum_{j=1}^{n'} a'_j \delta_{x'_j}$.  
Then there exists a matching $\{ \kappa_{i,j} \}$ from $\mu$ to $\mu'$ such that $\sum_{i,j} \kappa_{i,j} d(x_i,x'_j) < \eta_2$.  
Let $A = \{ (i,j) \mid d(x_i,x'_j) \geq \eta_1 \}$ and $B = \{ (i,j) \mid d(x_i,x'_j) < \eta_1 \}$.  Then we have 
\[  \sum_{(i,j) \in A} \kappa_{i,j} \leq 
\sum_{(i,j) \in A} \kappa_{i,j} \frac{d(x_i,x'_j)}{\eta_1}
<\frac{\eta_2}{\eta_1}
<\frac{\varepsilon}{2C}.\]
We can use the same $\{ \kappa_{i,j} \}$ to define a matching between the measures $\widetilde{H}(\mu,t) = \sum_{i=1}^n a_i \delta_{H(x_i,\mu,t)}$ and $\widetilde{H}(\mu',t') = \sum_{j=1}^{n'} a'_j \delta_{H(x'_j,\mu',t')}$, and by our choice of $\eta_1$,$\eta_2$, and $\eta_3$, we have 
\begin{align*}
& d_W(\widetilde{H}(\mu,t), \widetilde{H}(\mu',t')) \\
\leq& \sum_{i,j}\kappa_{i,j} d(H(x_i,\mu,t),H(x'_j,\mu',t')) \\
=& \sum_{(i,j) \in A} \kappa_{i,j} d(H(x_i,\mu,t),H(x'_j,\mu',t')) + \sum_{(i,j) \in B} \kappa_{i,j} d(H(x_i,\mu,t),H(x'_j,\mu',t')) \\
<& \sum_{(i,j) \in A} \kappa_{i,j} C + \sum_{(i,j) \in B} \kappa_{i,j} \frac{\varepsilon}{2} \\
<& \frac{\varepsilon}{2C}C + \frac{\varepsilon}{2} \\
=& \varepsilon.
\end{align*}
Therefore $\widetilde{H}$ is continuous at $(\mu,t)$.
\end{proof}

\subsection{Coordinates on the Circle}\label{subsection: coordinates on the circle}

We now describe some conventions for our work with the circle.
The straightforward techniques here will be used in detail later.
We give the circle $S^1$ the geodesic metric, written $d_{S^1}$, which assigns to two points the arc length of the shorter arc between them.
We will typically use an angle enclosed in square brackets to indicate a point on the circle: that is, $[\theta] = (\cos(\theta), \sin(\theta))$.
The square brackets can be thought of as denoting equivalence classes of points identified by the map $\mathbb{R} \to S^1$ given by $\theta \mapsto (\cos(\theta), \sin(\theta))$.
Thus, $[\theta_1] = [\theta_2]$ if and only if $\theta_1-\theta_2$ is an integer multiple of $2 \pi$.
We can then describe the distance between two points easily: without loss of generality, two points can be written as $[\theta_1]$ and $[\theta_2]$ with $\theta_1 \leq \theta_2 \leq \theta_1+\pi$, and the distance between them is given by $d_{S^1}([\theta_1],[\theta_2]) = \theta_2-\theta_1$.

It will be convenient to identify the circle minus a point with an open interval of length $2 \pi$ on the real line in a way that preserves distances locally.
For any angle $\theta_0 \in \mathbb{R}$ and any chosen point $y_0 \in \mathbb{R}$, we can make this identification by a coordinate system $x \colon S^1 - \{ [\theta_0] \} \to \mathbb{R}$ defined by 
\[
x([\theta]) = y_0 + \theta - \theta_0,
\]
where the representative $\theta$ for $[\theta]$ is taken in the interval $(\theta_0, \theta_0+2\pi)$.
The image of $x$ is thus $(y_0, y_0 + 2\pi)$.
We will describe such an $x$ as a \textit{coordinate system that excludes $[\theta_0]$}.
The $2\pi$-periodic function $\tau \colon \mathbb{R} \to S^1$ given by
\[
\tau(z) = (\cos(z+\theta_0 -y_0),\, \sin(z+\theta_0 -y_0))
\]
is a left inverse for $x$, that is, $\tau \circ x ([\theta]) = [\theta]$ for $[\theta]\neq[\theta_0]$.
Composing in the other direction, we have $x \circ \tau(z) = z$ for $z \in (y_0, y_0 + 2\pi)$.
We note that $\tau$ preserves distances that are at most $\pi$.
In general, we have $d_{S^1}(\tau(z_1), \tau(z_2)) \leq d_{\mathbb{R}}(z_1,z_2)$, that is, $\tau$ is 1-Lipschitz.
We will only use coordinate systems defined as $x$ is above, and we will also call $(x,\tau)$ a coordinate system to indicate that $\tau$ is the $2 \pi$-periodic left inverse for $x$.

We can convert between the coordinate system $x$ and another, $x' \colon S^1 - \{ [\theta'_0] \} \to \mathbb{R}$, defined by 
\[
x'([\theta]) = y'_0 + \theta - \theta'_0,
\]
where the representative $\theta$ for $[\theta]$ is taken in the interval $(\theta'_0, \theta'_0+2\pi)$.
We will assume without loss of generality that $\theta_0 \leq \theta'_0 < \theta_0 + 2\pi$.
These coordinate systems are related as follows:
\[
x([\theta]) - x'([\theta]) = \begin{cases}
(y_0-y_0')+(\theta'_0 - \theta_0) - 2\pi & \text{ if $\theta_0 < \theta < \theta'_0$} \\
(y_0-y_0')+ (\theta'_0 - \theta_0) & \text{ if $\theta'_0 < \theta < \theta_0 + 2\pi$}, \\
\end{cases}
\]
where here we choose the representative $\theta$ for $[\theta]$ from the intervals $(\theta_0, \theta'_0)$ or $(\theta'_0, \theta_0 + 2\pi)$.
If $\tau$ and $\tau'$ are the periodic left inverses, then we must have
\[
\tau(z) = \tau'(z-(y_0-y_0')-(\theta'_0-\theta_0)).
\]

\subsection{The Vietoris--Rips Metric Thickenings of the Circle}
Our main goal is to find the homotopy type of $\vrmleq{S^1}{r}$ for each $r$.
The homotopy types and the persistent homology barcodes of the Vietoris--Rips complexes of the circle were found in~\cite{AA_circle}.
Results in~\cite{Moy_Masters_Thesis} and~\cite{vrp_arxiv} show that the barcodes of the Vietoris--Rips metric thickenings of the circle agree, ignoring whether endpoints of bars are open or closed.
Combining these results shows that Vietoris--Rips metric thickenings have one persistent homology bar in each odd dimension $2k+1$, with endpoints $\frac{2k \pi}{2k+1}$ and $\frac{(2k+2)\pi}{2k+3}$.

We will show that the Vietoris--Rips metric thickenings of the circle have homotopy types of odd dimensional spheres in these intervals, as the barcodes suggest.
Lemma~\ref{lemma_support_homotopy} will be a key tool: it will allow us to define homotopies in $\vrmleq{S^1}{r}$ by sliding support points along the circle.
We use this technique to collapse the metric thickening down to the CW complex described in the introduction, which we show has the same homotopy type.
The construction of this collapse will depend on properties specific to the Vietoris--Rips metric thickenings of the circle, explored in the next section.

\vspace{.15cm}
\section{Odd numbers of arcs on the circle}\label{section:arcs}

A surprising amount of the structure of Vietoris--Rips metric thickenings of the circle will depend on the following simple observation.
Consider $n$ distinct pairs of antipodal points on the circle, with one of each pair colored blue and the other red.
Given such a set of red and blue points on the circle, consider the maximal length open arcs on the circle containing at least one blue point and no red points.
We can show by induction that there will always be an odd number of these maximal blue arcs.
There is one arc for $n=1$ or $n=2$, and adding a new pair changes the number of arcs only if the blue point is placed between two consecutive red points.
This introduces a new blue arc, but it also splits a previous blue arc, since the antipodal red point was placed between two consecutive blue points.
Therefore, each new antipodal pair introduced increases the number of maximal blue arcs by either 0 or 2.

We find similar behavior in finite subsets of $S^1$ with constrained diameter.
Let $r \in [0, \pi)$, and consider a nonempty set $\Theta = \{ [\theta_0], \dots , [\theta_n] \} \subset S^1$ with $\diam(\Theta) \leq r$.
Since $\Theta$ cannot contain a pair of antipodal points, we may color the points in $\Theta$ blue and the points opposite them red, obtaining the situation described above.
Furthermore, for any $[\theta_i] \in \Theta$, the open interval of length $2(\pi -r)$ opposite $[\theta_i]$ does not contain any other point in $\Theta$; we call a point in any such interval \textit{excluded by $\Theta$}.
Let a \textit{$(\Theta;r)$-arc} be a closed arc of maximal length such that there is at least one point of $\Theta$ contained in the arc and no point excluded by $\Theta$ is contained in the arc. 
We allow the case where a $(\Theta;r)$-arc consists of an individual point.
This simply shrinks the blue arcs described in the case of antipodal pairs, so the number of $(\Theta;r)$-arcs is still odd.
Let $\arcs_r(\Theta)$ be the number of $(\Theta;r)$-arcs.

If $\mu \in \vrmleq{S^1}{r}$, then by definition $\diam(\supp(\mu)) \leq r$, so the definitions above may be applied with $\Theta = \supp(\mu)$.
In this case we will call a $(\supp(\mu),r)$-arc a \textit{$(\mu,r)$-arc}, and will write $\arcs_r(\mu)$ for $\arcs_r(\supp(\mu))$.
For any $[\theta]$ in $\supp(\mu)$, we call any point in the open interval of length $2(\pi - r)$ opposite $[\theta]$ a point \textit{excluded by $\mu$} (note that this definition depends on the parameter $r$ -- we will use this term when $r$ is understood).
The set of all points excluded by $\mu$ may be called the \textit{excluded region} of $\mu$; this is the set of points that are at distance greater than $r$ from some point in $\supp(\mu)$.
Thus, a $(\mu,r)$-arc is a closed arc of maximal length such that there is at least one point in $\supp(\mu)$ contained in the arc and no point excluded by $\mu$ is contained in the arc. 
Each point in $\supp(\mu)$ is contained in exactly one $(\mu,r)$-arc, and as above, the number of $(\mu,r)$-arcs is odd.
Note that $(\mu,r)$-arcs are defined entirely in terms of $\supp(\mu)$, so if $\mu$ and $\mu'$ have the same support, then the $(\mu,r)$-arcs agree with the $(\mu',r)$-arcs.

\begin{figure}
    \centering
    \includegraphics[width = .4\textwidth]{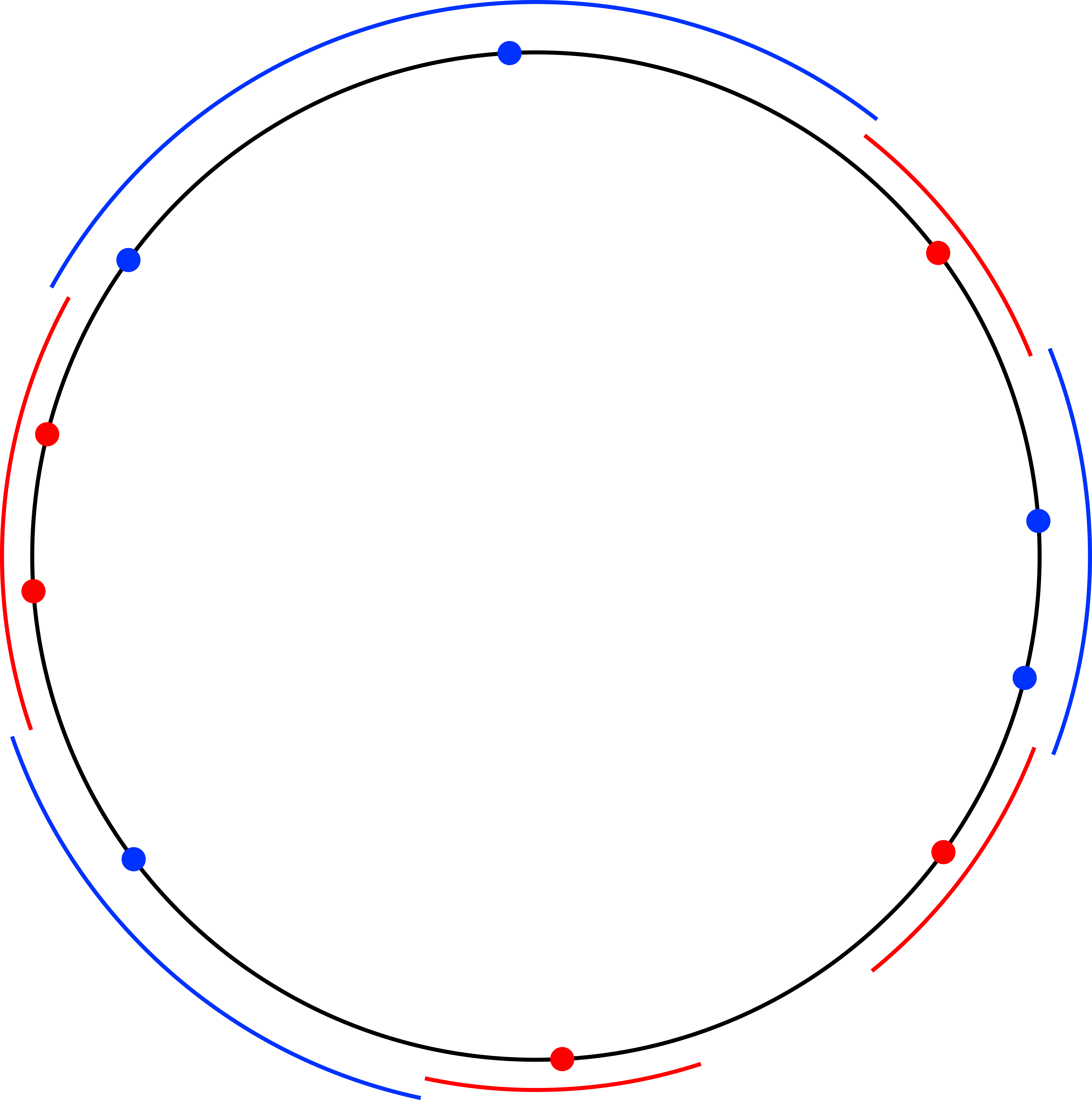}
    \caption{Visualization of $\Theta$-arcs and points excluded by $\Theta$.
    The blue points are points of $\Theta$ and the red points are the points opposite them.
    The blue arcs show the $\Theta$-arcs and the red arcs are excluded by $\Theta$.}
    \label{figure:red_and_blue_points}
\end{figure}

For any $k \geq 0$, let $V_{2k+1}(r)$ be the set of all measures $\mu \in \vrmleq{S^1}{r}$ that have exactly $2k+1$ $(\mu,r)$-arcs, and let $W_{2k+1}(r) = \bigcup_{l=0}^k V_{2l+1}(r)$ be the set of measures $\mu$ with at most $2k+1$ $(\mu,r)$-arcs.
For convenience, we will let $V_{-1}(r)$ and $W_{-1}(r)$ be empty.
By definition, the $V_{2k+1}(r)$ are disjoint, and their union over all $k \geq 0$ is $\vrmleq{S^1}{r}$.
We will mostly work with a fixed parameter $r$ and will often suppress the $r$ from the notation. 
In particular, we will often write the sets of measures above as $\vrmcircleq$, $V_{2k+1}$, and $W_{2k+1}$, and we will use the terms \textit{$\Theta$-arc} and \textit{$\mu$-arc} when $r$ is fixed or understood from context.

For $r \in [0,\pi)$, the region excluded by a point in the support of a measure has length $2(\pi-r)>0$, so there is a maximum number of arcs a measure in $\vrmleq{S^1}{r}$ can have.
Thus, $V_{2k+1}(r)$ is empty for all sufficiently large $k$. 
From here on, we let $K = K(r)$ be the largest value of $k$ such that $V_{2k+1}(r)$ is nonempty; then $\vrmleq{S^1}{r} = V_1(r) \cup \dots \cup V_{2K+1}(r) = W_{2K+1}(r)$.
To find $K$, note that in order for a measure $\mu$ to have $2k+1$ arcs, the set of points excluded by $\mu$ must be split into $2k+1$ connected components.
Since the open arc of length $2(\pi-r)$ opposite any point of $\supp(\mu)$ is excluded, this can only happen if $2(\pi-r)(2k+1) \leq 2\pi$, or equivalently $r \geq \frac{2k \pi}{2k+1}$.
Conversely, if $r \geq \frac{2k \pi}{2k+1}$, then any measure with support equal to a set of $2k+1$ evenly spaced points has diameter $\frac{2k \pi}{2k+1}$ and is thus in $V_{2k+1}(r)$.
This shows $V_{2k+1}(r)$ is nonempty if and only if $r \geq \frac{2k \pi}{2k+1}$.

We summarize the properties described above in the following proposition.

\begin{prop}\label{prop:facts_about_arcs_and_V_2k+1}
Let $r \in [0, \pi)$.
For any $\mu \in \vrmleq{S^1}{r}$, there are an odd number of $(\mu,r)$-arcs.
$V_{2k+1}(r)$ is nonempty if and only if $r \geq \frac{2k \pi}{2k+1}$, so $K = K(r)$ is the unique integer such that $\frac{2K \pi}{2K+1} \leq r < \frac{(2K+2)\pi}{2K+3}$. 
Thus, $V_1(r), V_3(r), \dots, V_{2K+1}(r)$ partition $\vrmleq{S^1}{r}$, and $\vrmleq{S^1}{r} = W_{2K+1}(r)$.
\end{prop}

The following lemma will allow us to determine when a measure $\mu$ is in $\vrmcircleq - W_{2k-1}$ (that is, when $\mu$ has at least $2k+1$ arcs).
We will write it in the general setting of finite subsets of $S^1$.

\begin{lem}\label{lemma:red_blue_points_and_mu_arcs}

Let $r \in [0, \pi)$ and let $\Theta$ be a nonempty finite set of points in $S^1$ with $\diam(\Theta) \leq r$.
There are at least $2k+1$ distinct $(\Theta,r)$-arcs if there exist distinct points $[\theta_0], \dots, [\theta_{2k}] \in \Theta$ such that if the $[\theta_i]$ are colored blue and their antipodal points are colored red, the red and blue points alternate around the circle.
Conversely, if there are exactly $2k+1$ $(\Theta,r)$-arcs, then if $[\theta_0], \dots, [\theta_{2k}]$ are points in $\Theta$, each contained in a distinct $(\Theta,r)$-arc, then coloring the $[\theta_i]$ blue and their antipodal points red results in red and blue points that alternate around the circle.
\end{lem}

\begin{proof}
Both statements are trivially true if $k=0$, so we suppose $k\geq1$.
Suppose first that there exist distinct points $[\theta_0], \dots, [\theta_{2k}] \in \Theta$ such that if the $[\theta_i]$ are colored blue and their antipodal points are colored red, the red and blue points alternate.
Each blue point belongs to a unique $\Theta$-arc, and each red point is excluded by $\Theta$. 
Since the red and blue points alternate, for any pair of blue points, there is a red point on each of the two arcs between the blue points.
Therefore the blue points must be contained in distinct $\Theta$-arcs, so there are at least $2k+1$ distinct $\Theta$-arcs.

To prove the second statement, suppose there are exactly $2k+1$ $(\Theta,r)$-arcs (note that if $k \geq 1$ and there are $2k+1$ $(\Theta,r)$-arcs, we must have $r \geq \frac{2k \pi}{2k+1} \geq \frac{2 \pi}{3}$).
Let $[\theta_0], \dots, [\theta_{2k}] \in \Theta$ with one $[\theta_i]$ in each $\Theta$-arc.
Color the $[\theta_i]$ blue and the points opposite them red.
We show there must be a red point on any arc between any distinct blue points $[\theta_i]$ and $[\theta_j]$; without loss of generality, we assume $\theta_i<\theta_j<\theta_i+2\pi$ and show there is a $\theta \in (\theta_i,\theta_j)$ such that $[\theta]$ is a red point.
If $\theta_i + \pi < \theta_j$, then $[\theta_i+\pi]$ is a red point and $\theta_i < \theta_i+\pi < \theta_j$, so now consider the case where $\theta_i < \theta_j < \theta_i + \pi$.
Since $[\theta_i]$ and $[\theta_j]$ belong to different $\Theta$-arcs, there is a point excluded by $\Theta$ that can be represented by an angle in $(\theta_i,\theta_j)$, so $\theta_j \geq \theta_i + 2(\pi-r)$ and there is a $\theta' \in \big[\theta_i+\pi+(\pi-r), \theta_j+\pi-(\pi-r)\big]$ such that $[\theta'] \in \Theta$.
Since $[\theta']$ belongs to some $\Theta$-arc, there must be a blue point $[\theta_k]$ in this $\Theta$-arc, and without loss of generality, we can choose the representative $\theta_k$ such that $\theta_k \in \big[\theta_i+\pi+(\pi-r), \theta_j+\pi-(\pi-r)\big]$.
Then $[\theta_k - \pi]$ is a red point and $\theta_i < \theta_k - \pi < \theta_j$.
Therefore there must be a red point on any arc between distinct blue points, so the red and blue points alternate around the circle.
\end{proof}

The somewhat combinatorial nature of the definition of $V_{2k+1}$ and $W_{2k+1}$ leads to interesting topological properties.
In general, the $V_{2k+1}$ are neither open nor closed, as shown in the following example.
However, we will see soon that each $W_{2k+1}$ is closed, and we will provide a description of the closure of $V_{2k+1}$ in $\vrmcircleq$.

\vspace{.2cm}
\begin{example}\label{example:V2k+1_not_necessarily_open_or_closed}
For any $k \geq 1$, suppose $r \in [0,\pi)$ is large enough so that $V_{2k+1}(r)$ is nonempty.
Then $V_{2k+1}$ contains measures supported on $2k+1$ evenly spaced points, and we can define a sequence with a predictable limit by varying the mass placed at these points: define the sequence $\{ \mu_n \}_{n \geq 1}$ by
\[
\mu_n = \tfrac{1}{n} \delta_{[0]} + \sum_{j=1}^{2k}\left( \tfrac{1}{2k} - \tfrac{1}{2kn} \right) \delta_{[\frac{2j \pi}{2k+1}]}.
\]
For each $n \geq 2$, $\mu_n$ is in $V_{2k+1}$, since it has nonzero mass at each of the $2k+1$ regularly spaced points.
On the other hand, the sequence converges to $\mu = \sum_{j=1}^{2k} \tfrac{1}{2k} \delta_{[\frac{2j \pi }{2k+1}]}$, which is in $V_{2k-1}$ because $[\frac{2k \pi}{2k+1}]$ and $[\frac{2(k+1) \pi}{2k+1}]$ are in the same $\mu$-arc (since $0$ is not in $\supp(\mu)$).
Thus, $\{\mu_n\}_{n \geq 2}$ is a sequence in $V_{2k+1}$ that converges to a measure in $V_{2k-1}$.
This shows $V_{2k-1}$ is not open in $\vrmcircleq$ and $V_{2k+1}$ is not closed in $\vrmcircleq$.
Since this example applies whenever $1 \leq k \leq K$, we see $V_{2k+1}$ is neither open nor closed if $1 \leq k \leq K-1$.
If $K > 0$, then $V_1$ is not open and $V_{2K+1}$ is not closed.
We will see soon that $V_1$ is always closed and $V_{2K+1}$ is always open.

\end{example}
\vspace{.3cm}

The following lemma shows that all measures sufficiently close to a fixed measure $\mu$ have certain properties determined by $\mu$.

\begin{lem}\label{lemma:close_measures_implies_close_masses}

Let $k \geq 0$ and $r \in [0, \pi)$, and suppose $\mu \in V_{2k+1}(r)$.
For all $\varepsilon > 0$, there exists an $\eta > 0$ such that the following statements hold for all $\nu \in \vrmleq{S^1}{r}$ satisfying $d_W(\mu,\nu)<\eta$.
\begin{enumerate}

    \item\label{lemma_item_support_points} For any $[\theta] \in \supp(\mu)$, there is a $[\theta'] \in \supp(\nu)$ such that $d_{S^1}([\theta],[\theta']) < \varepsilon$.
    
    \item\label{lemma_item_number_of_arcs} $\nu \in \vrmleq{S^1}{r} - W_{2k-1}(r)$, that is, $\nu$ has at least $2k+1$ arcs.
    
    \item\label{lemma_item_close_arc_masses} If $A_0, A_1, \dots, A_{2l}$ are all the $\nu$-arcs, then for each $i$, define the closed arc $A'_i$ by expanding $A_i$ by $\frac{\pi-r}{2}$ on both sides.
    Then $\nu(A'_i) = \nu(A_i)$ for each $i$, the arcs $A'_0, A'_1, \dots, A'_{2l}$ are disjoint, $\supp(\mu) \subseteq \bigcup_i A'_i$, and $|\mu(A'_i) - \nu(A'_i)| < \varepsilon$ for all $i$.
    
\end{enumerate}
\end{lem}

The length of $\frac{\pi-r}{2}$ in (\ref{lemma_item_close_arc_masses}) is just used for convenience, and it could be replaced with an arbitrarily small positive number.
This length will also be used later when we need to expand arcs by a small amount.

\begin{proof}
For (\ref{lemma_item_support_points}), let $m = \min \{ \mu([\theta]) : [\theta] \in \supp(\mu) \}$, noting that $m >0$ because $\mu$ is finitely supported.
Then moving any point mass of $\mu$ a distance of $\varepsilon$ costs at least $m\varepsilon$. 
Choosing $\eta \in (0,m\varepsilon)$, we find that for any $\nu \in \vrmcircleq$ satisfying $d_W(\mu,\nu) < \eta$, there is a transport plan between $\mu$ and $\nu$ with a cost of less than $m\varepsilon$, so each point in $\supp(\mu)$ must be at a distance less than $\varepsilon$ from some point in $\supp(\nu)$.

To show (\ref{lemma_item_number_of_arcs}), choose one point in each $\mu$-arc that is in $\supp(\mu)$ to color blue, and color the points opposite them red.
We apply (\ref{lemma_item_support_points}) to choose $\eta$ such that each point in $\supp(\mu)$ is at a distance less that $\frac{\pi-r}{2}$ from some point in $\supp(\nu)$.
Then for each blue point, choose a point in $\supp(\nu)$ at distance less than $\frac{\pi-r}{2}$ to color green, and color the points opposite the green points orange (here a point may be colored both blue and green or may be colored both red and orange).
Since the blue points are in distinct $\mu$-arcs, the distance between any two of them is at least $2(\pi-r)$, and thus the green points are distinct; so we have $2k+1$ points of each color.
Since the red and blue points alternate by Lemma~\ref{lemma:red_blue_points_and_mu_arcs} and each red point is at a distance of at least $\pi-r$ from each blue point, the green and orange points must alternate as well.  
So again by Lemma~\ref{lemma:red_blue_points_and_mu_arcs}, $\nu$ has at least $2k+1$ arcs.

Finally, we prove (\ref{lemma_item_close_arc_masses}).
By (\ref{lemma_item_support_points}), we may choose $\eta$ so that each point in $\supp(\mu)$ is within a distance of $\frac{\pi-r}{2}$ from some point in $\supp(\nu)$, and this will imply $\supp(\mu) \subseteq \bigcup_i A'_i$.
Since $A_0, A_1, \dots, A_{2l}$ are distinct $\nu$-arcs, each is at a distance of at least $2(\pi-r)$ from all others, so each $A'_i$ is at a distance of at least $\pi-r$ from all others.
This shows that the $A'_i$ are disjoint.
For each $i$, the only points of $\supp(\nu)$ in $A'_i$ are those that are in $A_i$, so $\nu(A'_i) = \nu(A_i)$.
In any transport plan between $\mu$ and $\nu$, for each $i$, at least a mass of $|\mu(A'_i) - \nu(A'_i)|$ must be transported from $A'_i$ to outside of $A'_i$.
Since all other $A'_j$ are at a distance of at least $\pi-r$ from $A'_i$, we must have $|\mu(A'_i) - \nu(A'_i)| (\pi-r) \leq d_W(\mu,\nu)$.
Therefore, if we require $\eta < (\pi-r)\varepsilon$, we find that if $d_W(\mu,\nu) < \eta$, then $|\mu(A'_i) - \nu(A'_i)| < \varepsilon$.
\end{proof}

For any $k \geq 0$ and any $\mu \in \vrmcircleq - W_{2k+1}$,  Lemma~\ref{lemma:close_measures_implies_close_masses}(\ref{lemma_item_number_of_arcs}) above implies there is an open neighborhood of $\mu$ in which all measures have at least as many arcs as $\mu$.
This neighborhood is therefore contained in $\vrmcircleq - W_{2k+1}$. 
This gives us the following lemma.

\begin{lem}\label{lemma:V_1_to_V_2k+1_closed}

For any $r \in [0,\pi)$ and any $k \geq 0$, $W_{2k+1}(r)$ is closed in $\vrmleq{S^1}{r}$.
\end{lem}

Note that in the case $k=0$, we have $W_1=V_1$, so this lemma shows $V_1$ is closed in $\vrmcircleq$.
This lemma also implies $V_{2k+1} = W_{2k+1} - W_{2k-1}$ is open in $W_{2k+1}$ for each $k$.
We now give an explicit description of the closure of each $V_{2k+1}$: we will write the closure of $V_{2k+1}(r)$ in $\vrmleq{S^1}{r}$ as $\overline{V_{2k+1}(r)}$.
Note that Lemma~\ref{lemma:V_1_to_V_2k+1_closed} already implies $\overline{V_{2k+1}} \subseteq W_{2k+1}$.
The following lemma shows that situations like that in Example~\ref{example:V2k+1_not_necessarily_open_or_closed}, in which a sequence in $V_{2k+1}$ converges to a point in $\overline{V_{2k+1}}$ by altering the masses on a fixed support, in fact account for all measures in $\overline{V_{2k+1}}$.
While this result is not unexpected, the proof is long and we give it in the appendix.

\begin{lem}
\label{lemma:closure_of_V_2k+1}

For all $k \geq 0$ and all $r \in [0,\pi)$, $\mu \in \overline{V_{2k+1}(r)}$ if and only if $\supp(\mu)$ is contained in a finite set $T \subset S^1$ such that $\diam(T) \leq r$ and $\arcs_r(T) = 2k+1$.
\end{lem}

Lemma~\ref{lemma:closure_of_V_2k+1} will allow us to describe measures of $\vrmcircleq$ in a concise and useful form.
First, by Lemma~\ref{lemma:V_1_to_V_2k+1_closed}, the closure of $V_{2k+1}$ in $W_{2k+1}$ is $\overline{V_{2k+1}}$, and the interior of $V_{2k+1}$ in $W_{2k+1}$ is $V_{2k+1}$.
Therefore, the boundary of $V_{2k+1}$ in $W_{2k+1}$ is $\overline{V_{2k+1}} - V_{2k+1}$.
From here on, we write $\partial V_{2k+1} = \overline{V_{2k+1}} - V_{2k+1}$ for the boundary of $V_{2k+1}$ in $W_{2k+1}$.
Note that this is not necessarily the boundary of $V_{2k+1}$ in $\vrmcircleq$, as there may be points in $V_{2k+1}$ that are not in the interior of $V_{2k+1}$ in $\vrmcircleq$ (see Example~\ref{example:V2k+1_not_necessarily_open_or_closed}).
By Lemma~\ref{lemma:closure_of_V_2k+1}, we may write any measure $\mu \in \overline{V_{2k+1}}$ as $\mu = \sum_{i=0}^{2k} a_i \mu_i$ where $\bigcup_i \supp(\mu_i)$ has $2k+1$ arcs, each $\mu_i$ is a probability measure supported on a distinct $(\bigcup_i \supp(\mu_i))$-arc, $a_i \geq 0$ for each $i$, and $\sum_i a_i = 1$.
Furthermore, $\mu \in \partial V_{2k+1}$ if and only if $a_i=0$ for some $i$. 
If $\mu \in V_{2k+1}$, then each $a_i$ is the amount of mass in an individual $\mu$-arc, so in this case we will refer to the $a_i$ as the \textit{arc masses} of $\mu$.
When we write $\mu \in \overline{V_{2k+1}}$ as $\mu = \sum_{i=0}^{2k} a_i \mu_i$ meeting the description above, we will say $\mu$ is written in \textit{(2k+1)-arc mass form}, or simply \textit{arc mass form} when $k$ is understood.
The value of $k$ is relevant, as measures may be in $\overline{V_{2k+1}}$ for multiple values of $k$ (in general, the closures $\overline{V_{2k+1}}$ are not disjoint, even though the $V_{2k+1}$ are disjoint: again, see Example~\ref{example:V2k+1_not_necessarily_open_or_closed}).
If $\mu \in V_{2k+1}$, both the set of $\mu_i$ and their corresponding $a_i$ are completely determined by $\mu$, so the arc mass form of $\mu$ is unique up to reordering the sum.
In general, it is not unique if $\mu \in \partial V_{2k+1}$, since if $a_i = 0$, there are many choices for $\mu_i$.
We now expand on the ideas of Lemma~\ref{lemma:close_measures_implies_close_masses}: the following lemma essentially shows that close measures have close arc masses.

\begin{lem}\label{lemma:close_arc_masses}

Let $r \in [0, \pi)$, and let each sum below express measures in arc mass form.

\begin{enumerate}
    \item\label{arc_mass_lemma_item_boundary} Let $k \geq 1$, and let $\mu \in \partial V_{2k+1}(r)$. 
    For any $\varepsilon > 0$, there exists an $\eta > 0$ such that if $\nu = \sum_{i=0}^{2k} b_i \nu_i \in V_{2k+1}(r)$ satisfies $d_W(\mu,\nu) < \eta$, then $b_i < \varepsilon$ for some $i$.
    
    \item\label{arc_mass_lemma_item_close_arc_masees} Let $k \geq 0$, and let $\mu = \sum_{i=0}^{2k} a_i \mu_i \in V_{2k+1}(r)$. 
    For any $\varepsilon > 0$, there exists an $\eta > 0$ such that if $\nu = \sum_{i=0}^{2k} b_i \nu_i \in V_{2k+1}(r)$ satisfies $d_W(\mu, \nu) < \eta$ and $A_0, \dots, A_{2k}$ are closed arcs obtained by expanding the $\nu$-arcs by $\frac{\pi - r}{2}$ on both sides, then possibly after reordering, we have $\supp(\mu_i) \subseteq A_i$, $a_i = \mu(A_i)$, $b_i = \nu(A_i)$, and $|a_i - b_i| < \varepsilon$ for each $i$.
\end{enumerate}

\end{lem}

\begin{proof}
To prove (\ref{arc_mass_lemma_item_boundary}), suppose $\nu = \sum_{i=0}^{2k} b_i \nu_i \in V_{2k+1}(r)$ is written in arc mass form and satisfies $d_W(\mu,\nu) < \eta$.
For each $i$, $\supp(\mu)$ must have a point within $\frac{\eta}{b_i}$ of some point in $\supp(\nu_i)$, otherwise the mass of $b_i \nu_i$ could not be transported for a cost of less than $\eta$.
We now suppose for a contradiction that $\frac{\eta}{\min_i \{ b_i \}} < \frac{\pi-r}{2}$.
For each $i$, we may choose a point in $\supp(\nu_i)$ to color green and a point in $\supp(\mu)$ within a distance of $\frac{\eta}{\min_i \{ b_i \}}$ from this green point to color blue (here we allow a point to be colored both green and blue).
The green points are in separate $\nu$-arcs, so they are at distance at least $2(\pi-r)$ from each other.
Since $\frac{\eta}{\min_i \{ b_i \}} < \frac{\pi-r}{2}$, the blue points must be distinct, so we have $2k+1$ points of each color.
Color the points opposite the blue points red and the points opposite the green points orange.
By Lemma~\ref{lemma:red_blue_points_and_mu_arcs}, the green and orange points alternate around the circle, and each green point is at a distance of at least $\pi-r$ from each orange point since $\diam(\supp(\nu)) \leq r$.
Since $\frac{\eta}{\min_i \{ b_i \}} < \frac{\pi-r}{2}$, this implies the red and blue points alternate as well.
But by Lemma~\ref{lemma:red_blue_points_and_mu_arcs}, this implies $\mu$ has at least $2k+1$ arcs, contradicting the assumption that $\mu \in \partial V_{2k+1}$.
Therefore we can conclude that $\frac{\eta}{\min_i \{ b_i \}} \geq \frac{\pi-r}{2}$, so $\min_i \{ b_i \} \leq \frac{2 \eta}{\pi - r}$.
So given any $\varepsilon > 0$, setting $\eta = \frac{\pi-r}{2}\varepsilon$ gives the desired result.

To prove (\ref{arc_mass_lemma_item_close_arc_masees}), let $\varepsilon > 0$.
Applying parts (\ref{lemma_item_support_points}) and (\ref{lemma_item_close_arc_masses}) of Lemma~\ref{lemma:close_measures_implies_close_masses}, we can choose an $\eta > 0$ such that for any $\nu = \sum_{i=0}^{2k} b_i \nu_i \in V_{2k+1}(r)$ written in arc mass form and satisfying $d_W(\mu, \nu) < \eta$, the following hold: each point in $\supp(\mu)$ is within $\frac{\pi-r}{2}$ of some point in $\supp(\nu)$, and letting $A_0, \dots, A_{2k}$ be the disjoint closed arcs obtained by expanding the $\nu$-arcs by $\frac{\pi -r}{2}$ on both sides, $\supp(\mu) \subseteq \bigcup_i A_i$ and $|\mu(A_i) - \nu(A_i)| < \varepsilon$ for each $i$.
If $k=0$, we are done, since $A_0$ contains all points of $\supp(\mu)$ and $\supp(\nu)$, so we can suppose $k \geq 1$.
Reordering if necessary, we have $b_i = \nu(A_i)$ by the definition of arc mass form.
We will show that for any $i$, the points of $\supp(\mu)$ contained in $A_i$ all belong to the same $\mu$-arc; since $\supp(\mu) \subseteq \bigcup_i A_i$, this will imply that each $A_i$ contains some points of $\supp(\mu)$ and thus contains exactly the points of $\supp(\mu)$ belonging to a particular $\mu$-arc.
After reordering if necessary, this will show $a_i = \mu(A_i)$ for each $i$.
Suppose for a contradiction that $[\theta_1], [\theta_2] \in \supp(\mu)$ are in distinct $\mu$-arcs and that $[\theta_1], [\theta_2] \in A_i$.
Color one point of $\supp(\mu)$ in each $\mu$-arc blue, choosing $[\theta_1]$ and $[\theta_2]$ for their $\mu$-arcs, and color the points opposite the blue points red.
By Lemma~\ref{lemma:red_blue_points_and_mu_arcs}, the red and blue points alternate, so there is a red point $[\theta']$ contained between $[\theta_1]$ and $[\theta_2]$ in $A_i$, and since $[\theta']$ is at distance at least $\pi-r$ from both $[\theta_1]$ and $[\theta_2]$, $[\theta']$ must in fact be contained in the $\nu$-arc contained in $A_i$.
Since there must be a point of $\supp(\nu)$ within $\frac{\pi-r}{2}$ of the blue point $[\theta' + \pi]$, the point $[\theta']$ is excluded by $\nu$, contradicting the fact that it is in a $\nu$-arc.
Therefore if $[\theta_1], [\theta_2] \in \supp(\mu) \cap A_i$, they must belong to the same $\mu$-arc, as required.
\end{proof}

We can now define certain subspaces of $\vrmcircleq$ of interest and describe how these subspaces will eventually be used to define a CW complex.
As with our previous definitions, we will often omit the parameter $r$.
For $0 \leq k \leq K$, let $P_{2k+1} = P_{2k+1}(r) \subseteq \vrmleq{S^1}{r}$ be the set of measures whose support is $2k+1$ evenly spaced points; we refer to these as \textit{regular polygonal measures}.
Each $P_{2k+1}$ is nonempty if and only if $V_{2k+1}$ is nonempty, and by Proposition~\ref{prop:facts_about_arcs_and_V_2k+1}, this holds if and only if $r \geq \frac{2k \pi}{2k+1}$.
The closure $\overline{P_{2k+1}}$ of $P_{2k+1}$ in $\vrmcircleq$ consists of measures whose support is contained in a set of $2k+1$ evenly spaced points (where not all of these points are required to be in the support).
The set of measures with support contained in a fixed individual $(2k+1)$-gon is homeomorphic to a $2k$-simplex (and thus homeomorphic to the disk $D_{2k}$), where a homeomorphism can be defined by taking linear combinations of delta measures to the corresponding linear combinations of vertices of a $2k$-simplex.
This homeomorphism sends a measure with all $2k+1$ points in its support to the interior of the simplex.
Furthermore, $\overline{P_{2k+1}} \cong D^{2k} \times S^1$, where the $S^1$ parameterizes the set of regular $(2k+1)$-gons on the circle (this $S^1$ may be better thought of as the quotient of $S^1$ by the action of $\frac{\mathbb{Z}}{(2k+1)\mathbb{Z}}$).
These homeomorphisms can be checked using Proposition 5.2 of~\cite{AAF}, for instance.
We define $\partial P_{2k+1} = \overline{P_{2k+1}} - P_{2k+1} = \overline{P_{2k+1}} \cap \partial V_{2k+1}$ since for $k\geq 1$, we have the homeomorphism $\partial P_{2k+1} \cong S^{2k-1} \times S^1$; however, $\partial P_{2k+1}$ is in general not the boundary of $P_{2k+1}$ in $\vrmcircleq$.
Note that $P_{1}(r)$ consists of all delta measures, is equal to its own closure, and is homeomorphic to $S^1$ by the canonical embedding of $S^1$ into $\vrmcircleq$.
We also let $R_{2k} = R_{2k}(r) \subseteq \vrmleq{S^1}{r}$ be the set of measures with support equal to the specific regular $(2k+1)$-gon $\{ [0], [\frac{1 \cdot 2 \pi}{2k+1}], \dots, [\frac{2k \cdot 2\pi}{2k+1} ] \}$ (the choice of the polygon containing $[0]$ is just for convenience; any fixed individual polygon could also be used).
Then $R_{2k} \subseteq V_{2k+1}$ and $R_{2k}$ is homeomorphic to the interior of a $2k$-simplex.
The closure $\overline{R_{2k}}$ of $R_{2k}$ in $\vrmcircleq$ is the set of measures with support contained in $\{ [0], [\frac{1 \cdot 2 \pi}{2k+1}], \dots, [\frac{2k \cdot 2\pi}{2k+1} ] \}$, and we will write $\partial R_{2k} = \overline{R_{2k}} - R_{2k}$.
Thus, for $k \geq 1$, the pair $(\overline{R_{2k}}, \partial R_{2k})$ is homeomorphic to $(D^{2k}, S^{2k-1})$.
Note that $\partial R_{2k}$ is not necessarily the boundary of $R_{2k}$ in $\vrmcircleq$.
For $k=0$, we let $D^0$ be a space with one point, so $R_0 \cong D^0$.

Our strategy for finding the homotopy type of $\vrmcircleq$ will be to define (in Section~\ref{section:sequence_of_quotients}) a quotient map $q \colon \vrmcircleq \to \vrmcircleq / \sim$ that is a homotopy equivalence.
Under the equivalence relation $\sim$, each measure of $\vrmcircleq$ will be equivalent to exactly one regular polygonal measure, that is, one measure in some $P_{2k+1}$.
Thus, $\vrmcircleq / \sim$ can be described by specifying how the closures $\overline{P_{2k+1}}$ are glued together by their boundaries.
We further split each $P_{2k+1}$ into $R_{2k}$ and $P_{2k+1} - R_{2k}$, which are homeomorphic to an open $2k$-disk and an open $(2k+1)$-disk respectively.
We will show that these form open cells of a CW complex that is homeomorphic to $\vrmcircleq / \sim$ and thus homotopy equivalent to $\vrmcircleq$.
We thus have one cell in each dimension $0 \leq n \leq 2K+1$, and these cells are glued together as described in Section~\ref{section:introduction} to give a space homotopy equivalent to $S^{2K+1}$.
In the following section, we record some technical results that will allow us to describe a quotient map that is also a homotopy equivalence.

\vspace{.15cm}
\section{Homotopies, Quotients, and the Homotopy Extension Property}
\label{section:homotopies_quotients_HEP}

\subsection{General Facts}\label{subsection:HEP_general_facts}

This section covers some facts related to quotient maps and the homotopy extension property.
We recall the relevant definitions.
If $X$ is a topological space and $A \subseteq X$, the pair $(X,A)$ is said to have the \textit{homotopy extension property (HEP)} if given any homotopy $H \colon A \times I \to Z$ and any map $f \colon X \to Z$ such that $f(a) = H(a,0)$ for any $a \in A$, there exists a homotopy $G \colon X \times I \to Z$ such that $G |_{A \times I} = H$ and $G(\_\,,0) = f$.
A \textit{fiber} of a function is a preimage of a singleton.
A surjective continuous function $q \colon X \to Y$ is a quotient map if and only if it satisfies the following universal property: for any space $Z$ and any continuous $f \colon X \to Z$ that is constant on the fibers of $q$ (that is, $q(x_1) = q(x_2)$ implies $f(x_1) = f(x_2)$), there is a unique continuous function $g \colon Y \to Z$ such that $g \circ q = f$, as in the following diagram.
\[
\begin{tikzcd}
X \arrow[d, "q"'] \arrow[rd, "f"] &   \\
Y \arrow[r, "g"', dashed]         & Z
\end{tikzcd}
\]
Furthermore, if this property holds, then $Y$ is homeomorphic to the quotient space $X / \sim$ where $x_1 \sim x_2$ if and only if $q(x_1) = q(x_2)$ and a subset of $X / \sim$ is open if and only if its preimage under $q$ is open in $X$.

Proposition~\ref{prop:quotient_homotopy_equivalences} below shows that quotient maps meeting certain conditions are homotopy equivalences, and this is one of the main tools we will use.
Lemma~\ref{lemma:homotopies_on_quotients} will be used in the proof of Proposition~\ref{prop:quotient_homotopy_equivalences}, as well as in a later section.
Lemma~\ref{lemma:products_and_quotients_for_locally_compact_Hausdorff} will only be used for proofs in this section.

\newpage

\begin{lem}\label{lemma:products_and_quotients_for_locally_compact_Hausdorff}
Suppose $q \colon X \to Y$ is a quotient map and $Z$ is a locally compact Hausdorff space. 
Then the product map $q \times 1_Z \colon X \times Z \to Y \times Z$ is a quotient map.
\end{lem}
\begin{proof}
Lemma 4.72 of~\cite{Lee-Introduction_to_Topological_Manifolds}.
\end{proof}

\begin{lem}\label{lemma:homotopies_on_quotients}
Suppose $X$ is a topological space, $\sim$ is an equivalence relation on $X$, and $\sim'$ is an equivalence relation on $X \times I$ defined by $(x_1,t_1) \sim' (x_2,t_2)$ if and only if $x_1 \sim x_2$ and $t_1=t_2$.
Then we have a homeomorphism $(X / \sim) \times I \cong (X \times I) / \sim'$ defined by $([x],t) \mapsto [(x,t)]$.
\end{lem}
\begin{proof}
Let $q \colon X \to X / \sim$ and $q' \colon X \times I \to (X \times I) / \sim'$ be the quotient maps.
By Lemma~\ref{lemma:products_and_quotients_for_locally_compact_Hausdorff}, since $I$ is locally compact and Hausdorff, the map $q \times 1_I \colon X \times I \to (X / \sim) \times I$ is also a quotient map.
It can be checked that the function $f \colon (X / \sim) \times I \to (X \times I) / \sim'$ given by $([x],t) \mapsto [(x,t)]$ is well-defined and is a bijection, so we just must verify it is continuous and has a continuous inverse.
This follows from the universal property of quotients since the fibers of $q \times 1_I$ and $q'$ agree and we have both $f \circ (q \times 1_I) = q'$ and $f^{-1} \circ q' = q \times 1_I$.
\[
\begin{tikzcd}
X \times I \arrow[dd, "q\times1_I"'] \arrow[rdd, "q'"] &                    \\
                                                       &                    \\
(X / \sim) \times I \arrow[r, "f"']                    & (X \times I)/\sim'
\end{tikzcd}
\]
\end{proof}

We will use the following fact about pairs of spaces with the HEP, which establishes that certain quotient maps are homotopy equivalences.
This is a modest generalization of Proposition~0.17 from~\cite{Hatcher}, and we will mimic its proof.

\begin{prop}\label{prop:quotient_homotopy_equivalences}
Suppose $(X,A)$ has the HEP and suppose $H \colon A \times I \to A$ is a homotopy such that $H(\_\,,0) = 1_A$ and each $H(\_\,,t)$ sends each fiber of $H(\_\,,1)$ into a fiber of $H(\_\,,1)$.
Define an equivalence relation on $X$ by $x_1 \sim x_2$ if and only if either $x_1=x_2$ or $x_1, x_2 \in A$ and $H(x_1,1) = H(x_2,1)$.
Then the quotient map $q \colon X \to X/\sim$ is a homotopy equivalence.
\end{prop}

\begin{proof}
Apply the HEP to find a homotopy $G \colon X \times I \to X$ such that $G(\_\,,0) = 1_X$ and $G(a,t) = H(a,t)$ for all $(a,t) \in A \times I$.
Let $\sim'$ be an equivalence relation on $X \times I$ defined by $(x_1,t_1) \sim' (x_2,t_2)$ if and only if $x_1 \sim x_2$ and $t_1=t_2$.
Because each $H(\_\,,t)$ sends fibers of $H(\_\,,1)$ into fibers of $H(\_\,,1)$, each $G(\_\,,t)$ sends fibers of $q$ into fibers of $q$.
Thus, $q \circ G$ is constant on the fibers of the quotient map $X \times I \to (X \times I)/ \sim'$, so we get an induced map on the quotient.
By applying the homeomorphism of Lemma~\ref{lemma:homotopies_on_quotients}, we obtain a homotopy $\widetilde{G}$ such that the following diagram commutes for each $t$.

\[
\begin{tikzcd}
X \arrow[r, "{G(\_\,,t)}"] \arrow[dd, "q"']          & X \arrow[dd, "q"] \\
                                                     &                   \\
X/\sim \arrow[r, "{\widetilde{G}(\_\,,t)}"', dashed] & X/\sim           
\end{tikzcd}
\]
Since $G(\_\,,0) = 1_X$, we have $\widetilde{G}(\_\,,0) = 1_{X/\sim}$.
Furthermore, since $G(a,t) = H(a,t)$ for all $(a,t) \in A \times I$, we can see that $G(\_\,,1)$ is constant on the fibers of $q$, so we get a map $g \colon X/\sim \,\, \to X$ such that the following diagram commutes.
\[
\begin{tikzcd}
X \arrow[r, "{G(\_\,,1)}"] \arrow[dd, "q"']                           & X \arrow[dd, "q"] \\
                                                                      &                   \\
X/\sim \arrow[r, "{\widetilde{G}(\_\,,1)}"'] \arrow[ruu, "g", dashed] & X/\sim           
\end{tikzcd}
\]
Therefore $g \circ q \simeq 1_X$ via $G$ and $q \circ g \simeq 1_{X/\sim}$ via $\widetilde{G}$, so $q$ is a homotopy equivalence.
\end{proof}

The next proposition will allow for further use of the HEP in combination with quotient maps.

\begin{prop}\label{prop:HEP_for_quotients}
Let $(X,A)$ be a pair of spaces with the HEP and let $\sim$ be an equivalence relation on $X$ with quotient map $q \colon X \to X / \sim\,$. 
If for each $x \in X - A$ the equivalence class of $x$ is the singleton $\{ x \}$, then the pair $(X/\sim\,,\, q(A))$ has the HEP.
\end{prop}

\begin{proof}
A pair $(X,A)$ has the HEP if and only if there exists a retraction $r \colon X \times I \to X \times \{0\} \cup A \times I$ (see Proposition A.18 of~\cite{Hatcher}).
We will find a map $\widetilde{r}$ making the following diagram commute.
\[
\begin{tikzcd}
X \times I \arrow[rr, "r"] \arrow[ddd, "q \times 1_I"'] &  & X\times \{0\} \cup A \times I \arrow[ddd, "(q \times 1_I)\vert_{X\times \{0\} \cup A \times I}"] \\
                                                        &  &                                                                                                  \\
                                                        &  &                                                                                                  \\
(X/\sim) \times I \arrow[rr, "\widetilde{r}"', dashed]  &  & (X / \sim) \times \{0\} \cup (A / \sim) \times I                                                
\end{tikzcd}
\]

By Lemma~\ref{lemma:products_and_quotients_for_locally_compact_Hausdorff}, the map $q \times 1_I$ is a quotient map, so by the universal property of quotients, it is sufficient to show that $(q \times 1_I)\vert_{X\times \{0\} \cup A \times I} \circ r$ is constant on the fibers of $q \times 1_I$.
This follows from the fact that $r$ is constant on $A \times I$ and each $x \in X - A$ is the only element of its equivalence class.
Finally, $\widetilde{r}$ is a retraction because $r$ is, so the pair $(X/\sim\,,\, q(A))$ has the HEP.
\end{proof}

\vspace{.25cm}
\subsection{The Homotopy Extension Property for $\big( \vrmleq{S^1}{r}, W_{2k+1}(r) \big)$}\label{subsection:HEP_for_(VRm,W)}

In order to apply the ideas above in later sections, we will first demonstrate that certain pairs of spaces within $\vrmcircleq$ have the HEP.
We will use the fact that a pair $(X,A)$ has the HEP if and only if $X \times \{ 0 \} \cup A \times I$ is a retract of $X \times I$ (Proposition A.18 of~\cite{Hatcher}).
For any $n \geq 1$, let $\Delta^n \subset \mathbb{R}^n$ be a regular $n$-simplex centered at the origin.
We can first obtain retractions that demonstrate $(\Delta^n, \partial \Delta^n)$ has the HEP similar to the retractions used in Proposition 0.16 of~\cite{Hatcher}.
Let $\lambda_n \colon \Delta^n \times I \to \Delta^n \times \{0\} \cup \partial \Delta^n \times I$ be the map defined by projecting radially from the point $(\vec{0},2) \in \Delta^n \times \mathbb{R}$, where $\Delta^n \times \mathbb{R}$ is considered as a subspace of $\mathbb{R}^{n+1}$.
Then $\lambda_n$ is continuous, and if the vertices of $\Delta^n$ are $v_0, \dots, v_n$, then $\lambda_n$ has the form 
\[
\lambda_n \left( \sum_{i=0}^n a_i v_i , t\right) = 
\left( \sum_{i=0}^n \lambda_{n,i}(a_0, \dots, a_n, t) v_i , \sigma_n(a_0, \dots, a_n, t)\right),
\]
where $\sigma_n \colon \Delta^n \times I \to I$ is continuous and the barycentric coordinates $\lambda_{n,i} \colon \Delta^n \times I \to I$ are continuous.
Any point in the codomain $\Delta^n \times \{0\} \cup \partial \Delta^n \times I$ has at least one of the barycentric coordinates or the coordinate for $I$ equal to zero, so for any $\left( \sum_{i=0}^n a_i v_i , t\right) \in \Delta^n \times I$, either $\sigma_n(a_0, \dots, a_n, t) = 0$ or $\lambda_{n,i}(a_0, \dots, a_n, t) = 0$ for some $i$.
Furthermore, $\lambda_n$ respects the symmetry of $\Delta^n$, so any permutation of the vertices $v_i$ does not affect the definition.
Since $\lambda_n$ fixes points in $\Delta^n \times \{0\} \cup \partial \Delta^n \times I$, it is a retraction; specifically, $\lambda_n \left( \sum_{i=0}^n a_i v_i , t\right) = \left( \sum_{i=0}^n a_i v_i , t\right)$ if either $t=0$ or $a_i = 0$ for some $i$.

We extend the ideas above to subsets of $\vrmcircleq$.
Recall that we have defined $\partial V_{2k+1} = \overline{V_{2k+1}} - V_{2k+1}$ and that $\partial V_{2k+1}$ is the boundary of $V_{2k+1}$ in $W_{2k+1}$, although it is not necessarily the boundary in $\vrmcircleq$.
For each $k \geq 1$, we define a retraction $\rho_{2k+1} \colon \overline{V_{2k+1}} \times I \to  \overline{V_{2k+1}} \times \{0\} \cup \partial V_{2k+1} \times I$ based on $\lambda_{2k}$.
With measures written in $(2k+1)$-arc mass form, define
\[
\rho_{2k+1} \left( \sum_{i=0}^{2k} a_i \mu_i , t\right) = 
\left( \sum_{i=0}^{2k} \lambda_{2k, i}(a_0, \dots, a_{2k}, t) \mu_i ,\, \sigma_{2k}(a_0, \dots, a_{2k}, t)\right).
\]
Since for any $a_0, \dots, a_{2k},t$, either $\sigma_{2k}(a_0, \dots, a_{2k}, t)=0$ or $\lambda_{2k, i}(a_0, \dots, a_{2k}, t) = 0$ for some $i$, $\rho_{2k+1}$ does in fact send points into $\overline{V_{2k+1}} \times \{0\} \cup \partial V_{2k+1} \times I$.
Each measure $\mu \in \overline{V_{2k+1}}$ may be expressed in arc mass form $\mu = \sum_{i=0}^{2k} a_i \mu_i$ in multiple ways, either by permuting indices or by a choice of $\mu_i$ when $a_i = 0$; we must check that the definition of $\rho_{2k+1}$ does not depend on the choice of how $\mu$ is written. 
First, if $a_i = 0$ for some $i$, then as described above, $\lambda_{2k} \left( \sum_{i=0}^{2k} a_i v_i , t\right) = \left( \sum_{i=0}^{2k} a_i v_i , t\right)$, which implies $\rho_{2k+1}(\mu,t) = (\mu,t)$. 
Thus, if $a_i = 0$ for some $i$, then $\rho(\mu,t)$ is uniquely defined.  
If $a_i \neq 0$ for each $i$, then $\mu$ has $2k+1$ arcs, and thus two different ways of expressing $\mu$ in arc mass form must be the same up to a permutation of indices.
By the symmetry of $\lambda_{2k}$, permuting the set of indices does not affect the value of $\rho(\mu,t)$.
Therefore $\rho_{2k+1}$ is a well-defined function.

\begin{lem}
\label{lemma:HEP_for_V}
For each $k \geq 1$, the function $\rho_{2k+1} \colon \overline{V_{2k+1}} \times I \to  \overline{V_{2k+1}} \times \{0\} \cup \partial V_{2k+1} \times I$ is a retraction, and thus the pair $(\overline{V_{2k+1}(r)}, \partial V_{2k+1}(r))$ has the homotopy extension property.
\end{lem}

\begin{proof}
Since $\lambda_{2k}$ is a retraction, $\rho_{2k+1}$ fixes all points of $\overline{V_{2k+1}} \times \{0\} \cup \partial V_{2k+1} \times I$, as required.
We need to show it is continuous.
We will suppose $\{ (\nu_n,t_n) \}_{n \geq 0}$ is a sequence in $\overline{V_{2k+1}} \times I$ that converges to $(\mu, t) \in \overline{V_{2k+1}} \times I$ and check that $\rho_{2k+1}(\nu_n, t_n)$ converges to $\rho_{2k+1}(\mu, t)$, splitting into cases for when $\mu \in V_{2k+1}$ and when $\mu \in \partial V_{2k+1}$.

For the first case, suppose $\mu \in V_{2k+1}$ and write $\mu$ in $(2k+1)$-arc mass form as $\mu = \sum_{j=0}^{2k} a_j \mu_j$.
Then $a_j \neq 0$ for all $j$, since $\mu \in V_{2k+1}$.
Applying Lemma~\ref{lemma:close_arc_masses}(\ref{arc_mass_lemma_item_close_arc_masees}), we see that for all large enough $n$, we can write each $\nu_n$ in arc mass form as $\nu_n = \sum_{j=0}^{2k} a_{n,j} \nu_{n,j}$ such that $\lim_{n \to \infty} a_{n,j} = a_j$ for each $j$.
As in the lemma, the $\nu_{n,j}$ can be chosen so that expanding the $\nu_n$-arc containing $\supp(\nu_{n,j})$ by $\frac{\pi-r}{2}$ on either side produces an arc that contains $\supp(\mu_j)$.
Furthermore, we show $\nu_{n,j}$ converges weakly to $\mu_j$ for each $j$.
Let $A_j$ be the $\mu$-arc containing $\supp(\mu_j)$. 
Define $A'_j$ by expanding $A_j$ by $\frac{\pi-r}{4}$ on either side, and define $A''_j$ by expanding $A_j$ by $\frac{\pi-r}{2}$ on either side.
For all large enough $n$, $\supp(\nu_{n,j})$ is contained in $A'$, since by Lemma~\ref{lemma:close_measures_implies_close_masses}(\ref{lemma_item_support_points}), expanding all open arcs excluded by $\nu_{n}$ by $\frac{\pi-r}{4}$ covers all points excluded by $\mu$.
Any bounded continuous function $f \colon S^1 \to \mathbb{R}$ can be replaced with a bounded continuous function $\widetilde{f}$ equal to $f$ on $A'_j$ and with $\supp(\widetilde{f}) \subseteq A''_j$. 
Then $\int_{S^1} f \,d\nu_{n,j} = \int_{S^1} \widetilde{f} \,d\nu_n$ for all large enough $n$, so
\[
\lim_{n \to \infty} \int_{S^1} f \,d\nu_{n,j} = \lim_{n \to \infty} \int_{S^1} \widetilde{f} \,d\nu_n = \int_{S^1} \widetilde{f} \,d\mu = \int_{S^1} f \,d\mu_j,
\]
where the second equality follows from Lemma~\ref{lemma:weak_convergence_and_Wasserstein_convergence}.
Therefore $\nu_{n,j}$ converges weakly to $\mu_j$ for each $j$.

Since $a_{n,j}$ approaches $a_j$ for each $j$, continuity of each $\lambda_{2k,i}$ and $\sigma_{2k}$ show that $\lambda_{2k,i}(a_{n,0}, \dots, a_{n,2k}, t_n)$ approaches $\lambda_{2k,i}(a_{0}, \dots, a_{2k}, t)$ for each $i$ and $\sigma_{2k}(a_{n,0}, \dots, a_{n,2k}, t_n)$ approaches $\sigma_{2k}(a_{0}, \dots, a_{2k}, t)$ as $n$ approaches infinity.
Then since $\nu_{n,j}$ converges weakly to $\mu_j$ for each $j$, Lemma~\ref{lemma:weak_convergence_and_Wasserstein_convergence} shows the components $\sum_{i=0}^{2k} \lambda_{2k, i}(a_{n,0}, \dots, a_{n,2k}, t_n) \nu_{n,i}$ converge in the Wasserstein distance to $\sum_{i=0}^{2k} \lambda_{2k, i}(a_0, \dots, a_{2k}, t) \mu_i$, which is the first component of $\rho_{2k+1}(\mu,t)$.
We have thus shown both components of $\rho_{2k+1}(\nu_n, t_n)$ converge, so $\rho_{2k+1}(\nu_n, t_n)$ converges to $\rho_{2k+1}(\mu, t)$, as required.

We now consider the second case, where $\mu \in \partial V_{2k+1}$, and we have previously noted this means $\rho_{2k+1}(\mu,t) = (\mu, t)$.
First, we determine how $\lambda_{2k}$ behaves near $\partial \Delta^{2k} \times I$.
Since $\lambda_{2k}$ is continuous, we have a continuous function $\widetilde{\lambda}_{2k} \colon \Delta^{2k} \times I \to \mathbb{R}^{2k+1}$ given by $\widetilde{\lambda}_{2k}(x,t) = \lambda_{2k}(x,t)-(x,t)$.
For any open neighborhood of $\vec{0}$ in $\mathbb{R}^{2k+1}$, the preimage under $\widetilde{\lambda}_{2k}$ is an open set that contains the compact set $\partial \Delta^{2k} \times I$ and thus contains an open ball around this compact set\footnote{This is a general fact about compact subsets of metric spaces: see, for instance, Exercise 2 in Section 27 of~\cite{Munkres}.}.
Therefore, for any $\varepsilon > 0$, there is an $\eta > 0$ such that if $(a_0, \dots, a_{2k}, t) \in \Delta^{2k} \times I$ with $a_j < \eta$ for some $j$, then $|\lambda_{2k,i}(a_0, \dots, a_{2k},t) - a_i| < \varepsilon$ for all $i$ and $|\sigma_{2k}(a_0, \dots, a_{2k},t) - t| < \varepsilon$.

We apply this fact to describe the image of the sequence $\{(\nu_n, t_n)\}$ under $\rho_{2k+1}$.
Again, write each $\nu_n$ in arc mass form as $\nu_n = \sum_{j=0}^{2k} a_{n,j} \nu_{n,j}$.
Temporarily write the first component of $\rho_{2k+1}$ as a map $\omega_{2k+1} \colon \overline{V_{2k+1}} \times I \to \overline{V_{2k+1}}$, so that \[\omega_{2k+1}\left( \nu_n , t_n \right) = \sum_{i=0}^{2k} \lambda_{2k, i}(a_{n,0}, \dots, a_{n,2k}, t_n) \nu_{n,i}.\]
By Lemma~\ref{lemma:close_arc_masses}(\ref{arc_mass_lemma_item_boundary}) and the fact that $\nu_n$ converges to $\mu \in \partial V_{2k+1}$, given any $\eta > 0$, for all sufficiently large $n$, we have $a_{n,j} < \eta$ for some $j$.
Applying the fact above, this shows that given any $\varepsilon > 0$, for all sufficiently large $n$, we have $|\lambda_{2k,i}(a_{n,0}, \dots, a_{n,2k},t_n) - a_{n,i}| < \varepsilon$ for all $i$ and $|\sigma_{2k}(a_{n,0}, \dots, a_{n,2k},t_n) - t_n| < \varepsilon$.
Simple bounds on the Wasserstein distance show that this implies $\omega_{2k+1}(\nu_n,t_n)$ is arbitrarily close to $\nu_n$ for all sufficiently large $n$.
Combined with the fact that $(\nu_n,t_n)$ converges to $(\mu,t) = \rho_{2k+1}(\mu,t)$, this shows $\rho_{2k+1}(\nu_n, t_n)$ converges to $\rho_{2k+1}(\mu,t)$.
\end{proof}

We use Lemma~\ref{lemma:HEP_for_V} to prove the fact we will use in later sections, that $\big(\vrmcircleq, W_{2k+1} \big)$ has the homotopy extension property for each $k$.
Recall we have defined $K = K(r)$ to be the smallest integer such that $\vrmcircleq = W_{2K+1}$.
For each $k \geq 1$, we extend the retraction $\rho_{2k+1} 
$ to a retraction 
\[\widetilde{\rho}_{2k+1} \colon W_{2K+1} \times \{0\} \cup W_{2k+1} \times I \to W_{2K+1} \times \{0\} \cup W_{2k-1} \times I\] 
defined by
\[
\widetilde{\rho}_{2k+1}(\mu,t) =
\begin{cases}
\rho_{2k+1}(\mu,t) & \text{ if $(\mu,t) \in \overline{V_{2k+1}} \times I$}\\
(\mu,t) & \text{ if $(\mu,t) \in W_{2K+1} \times \{0\} \cup W_{2k-1} \times I$}.
\end{cases}
\]
We have defined $\widetilde{\rho}_{2k+1}$ on two closed subsets of $W_{2K+1} \times \{0\} \cup W_{2k+1} \times I$, since $W_{2k-1} \times I$ is closed for each $k$ by Lemma~\ref{lemma:V_1_to_V_2k+1_closed}.
The intersection is given by
\[
\overline{V_{2k+1}} \times I \cap \big( W_{2K+1} \times \{0\} \cup W_{2k-1} \times I \big)  
= \overline{V_{2k+1}} \times \{0\} \cup \partial V_{2k+1} \times I
\]
and $\rho_{2k+1}$ is constant on this intersection by Lemma~\ref{lemma:HEP_for_V}, which shows $\widetilde{\rho}_{2k+1}$ is well-defined.  
Again by Lemma~\ref{lemma:HEP_for_V}, the definitions on the two closed sets are continuous, so $\widetilde{\rho}_{2k+1}$ is continuous.
By definition, all points of $W_{2K+1} \times \{0\} \cup W_{2k-1} \times I$ are fixed, so $\widetilde{\rho}_{2k+1}$ is in fact a retraction for each $k$.
By applying these retractions in decreasing order starting with $\widetilde{\rho}_{2K+1}$, we obtain a retraction
\[\widetilde{\rho}_{2k+3}\circ\ldots\circ\widetilde{\rho}_{2K+1}\colon W_{2K+1} \times I \to W_{2K+1} \times \{0\} \cup W_{2k+1} \times I\]
for any $0 \leq k < K$.
Thus, Lemma~\ref{lemma:HEP_for_V} implies the following.

\begin{prop}
\label{prop:HEP_for_(vrm,W)}
For any $k \geq 0$, there exists a retraction
\[
\vrmleq{S^1}{r} \times I \longrightarrow  \vrmleq{S^1}{r} \times \{0\}  \cup  W_{2k+1}(r) \times I,
\]
and thus the pair $\big( \vrmleq{S^1}{r}, W_{2k+1}(r) \big)$ has the homotopy extension property. 
\end{prop}

\vspace{.15cm}
\section{Collapse to Regular Polygons}\label{section:collapse}

For each $k\geq 0$ and any $r \in [0,\pi)$ such that $V_{2k+1}(r)$ is nonempty, we define a homotopy that collapses $V_{2k+1}$ to the set of regular polygonal measures $P_{2k+1}$.
We first define a support homotopy (see Section~\ref{subsection: vrm and support homotopies}).
Let $\mu \in V_{2k+1}$.
We will choose a coordinate system $(x,\tau)$ with $x \colon S^1 - \{ [\theta_0] \} \to \mathbb{R}$.
If $k=0$, we can choose $[\theta_0]$ to be any point excluded by $\mu$ and let $A_0$ be the single $\mu$-arc.
Otherwise, let $A_0, A_1, \dots, A_{2k}$ be the $\mu$-arcs, in counterclockwise order around the circle and with $[\theta_0]$ chosen strictly between the two closest support points of $A_{2k}$ and $A_0$.
Let $v_{2k+1}^{x,\mu} \colon S^1 \to \mathbb{R}$ be a function such that $[\theta] \in A_{v_{2k+1}^{x,\mu}([\theta])}$ for any $[\theta]$ belonging to any $A_i$. 
Then $v_{2k+1}^{x,\mu}$ is constant on the arcs, and we can choose it to be continuous.
Define $m_{2k+1}^x \colon V_{2k+1} \to \mathbb{R}$ by
\[
m_{2k+1}^x(\mu) = \int_{S^1} \left( x - \frac{2 \pi}{2k+1} v_{2k+1}^{x,\mu} \right) d \mu = \sum_{i=0}^{2k} \int_{A_i} \bigg( x-\frac{2i \pi}{2k+1} \bigg) d \mu ,
\]
where we recall $x \colon S^1 - \{ [\theta_0] \} \to \mathbb{R}$ is a continuous function, defined everywhere except the set $\{ [\theta_0] \}$, which has measure $0$.
Furthermore, let $\cQ(S^1,V_{2k+1}) = \{ ([\theta],\mu) \in S^1 \times V_{2k+1} \mid [\theta] \in \text{supp}(\mu)\}$
and define, using any such coordinate system $(x, \tau)$, the function\footnote{We define $F_{2k+1}$ for any value of $r$ such that $V_{2k+1}(r)$ is nonempty.
However, the definition does not depend on $r$, so we can safely omit it from the notation.
We follow this convention for the homotopies $\widetilde{F}_{2k+1}$, $G_{2k+1}$, and $\widetilde{G}_{2k+1}$, defined later, as well.
In fact, we could even treat $r$ as fixed until the end of Section~\ref{section:cw_complex_and_homotopy_types}.} $F_{2k+1} \colon \cQ(S^1,V_{2k+1}) \times I \to S^1$ by
\[
F_{2k+1}([\theta],\mu,t) = \tau \bigg( (1-t) \, x([\theta]) + t \, \bigg(\frac{2 \pi }{2k+1}v_{2k+1}^{x,\mu}([\theta]) + m_{2k+1}^x(\mu) \bigg) \bigg).
\]

The intuition for these definitions is as follows.
We use the choice of $x$ to work with coordinates in $\mathbb{R}$ (we will soon show that the definition of $F_{2k+1}$ is independent of the choice of coordinate system).
The homotopy is constructed as a composition
\[
\begin{tikzcd}
{\cQ(S^1,V_{2k+1}) \times I} \arrow[r] & \mathbb{R} \arrow[r, "\tau"] & S^1.
\end{tikzcd}
\]
The integral $\int_{A_i} x \, d\mu$ acts as a weighted average (ignoring the total mass of $A_i$) of the images under $x$ of the support points in $A_i$.
Since the $\mu$-arcs $A_0, \dots, A_{2k}$ are in counterclockwise order around the circle, the integral $m_{2k+1}^x(\mu) = \sum_{i=0}^{2k} \int_{A_i} \left( x-\frac{2i \pi}{2k+1} \right) d \mu$ takes an average of where points in $\supp(\mu)$ ``expect" $A_0$ to be centered.
Then for each $[\theta] \in \supp(\mu)$, $\frac{2 \pi }{2k+1}v_{2k+1}^{x,\mu}([\theta]) + m_{2k+1}^x(\mu)$ is an angle of a point on a regular polygon associated to $\mu$, as $v_{2k+1}^{x,\mu} \in \{0, \dots,2k\}$.
The homotopy is then defined as a linear homotopy in $\mathbb{R}$, and we compose with the map $\tau$ to return to $S^1$.
This has the effect of moving all masses in a single $\mu$-arc to the same point and ending with masses located at $2k+1$ evenly spaced points; we can picture $F_{2k+1}$ as deforming the support of each measure in $V_{2k+1}$ into an average regular polygon.

\begin{lem}\label{lemma:F_is_a_support_homotopy}

For each $k \geq 0$, the function $F_{2k+1} \colon \cQ(S^1, V_{2k+1}(r)) \times I \to S^1$ is well-defined and is a support homotopy.
\end{lem}

\begin{proof}
We begin by showing $F_{2k+1}$ is well-defined, that is, that the choice of coordinate system does not affect the definition.
We compare the definition for two coordinate systems $x \colon S^1 - \{ [\theta_0] \} \to \mathbb{R}$ and $x' \colon S^1 - \{ [\theta'_0] \} \to \mathbb{R}$, where $\{ [\theta_0] \}$ and $\{ [\theta'_0] \}$ are points excluded by $\mu$.
As above, if $k=0$, let $A_0$ be the single $\mu$-arc, and otherwise, let $A_0, A_1, \dots, A_{2k}$ be the $\mu$-arcs, in counterclockwise order around the circle, and with $[\theta_0]$ between $A_{2k}$ and $A_0$.
If $k \geq 1$, then $[\theta'_0]$ is excluded by $\mu$, and it lies between two $\mu$-arcs; let $l$ be such that $[\theta'_0]$ lies between $A_{l-1}$ and $A_l$, or let $l=2k+1$ if $[\theta'_0]$ lies between $A_{2k}$ and $A_{0}$.
If $k = 0$, let $l = 1$.
By converting between coordinates $x$ and $x'$ as in Section~\ref{subsection: coordinates on the circle}, we find that there is an $s \in \mathbb{R}$ such that on arcs, 
\[
x'([\theta]) = \begin{cases}
x([\theta]) - s + 2\pi & \text{ if $[\theta] \in A_0 \cup \dots \cup A_{l-1}$}\\
x([\theta]) - s & \text{ if $[\theta] \in A_l \cup \dots \cup A_{2k}$}.
\end{cases}
\]
Furthermore, there exist periodic functions $\tau, \tau' \colon \mathbb{R} \to S^1$ such that $\tau \circ x = 1_{S^1}$ and $\tau' \circ x' = 1_{S^1}$, and these must satisfy $\tau(z) = \tau'(z-s)$.

We just need to convert each term in the definition of $F_{2k+1}$ between the two coordinate systems.
First, for $[\theta] \in A_0 \cup \dots \cup A_{2k}$, we have
\[
v_{2k+1}^{x',\mu}([\theta]) = \begin{cases}
v_{2k+1}^{x,\mu}([\theta]) + (2k+1) - l & \text{ if $[\theta] \in A_0 \cup \dots \cup A_{l-1}$} \\
v_{2k+1}^{x,\mu}([\theta]) - l & \text{ if $[\theta] \in A_l \cup \dots \cup A_{2k}$}. \\
\end{cases}
\]
Next, keeping in mind that $\supp(\mu) \subset A_0 \cup \dots \cup A_{2k}$, we compute
\begin{align*}
m_{2k+1}^{x'}(\mu) &= \int_{S^1} \left( x' - \frac{2 \pi}{2k+1} v_{2k+1}^{x',\mu} \right) d \mu \\
& = \int_{A_0 \cup \dots \cup A_{l-1}} \left( x - s + 2 \pi - \frac{2 \pi}{2k+1}\left( v_{2k+1}^{x,\mu} + (2k+1) - l \right)\right) d \mu \\
& \hspace{1cm} + \int_{A_l \cup \dots \cup A_{2k}} \left( x - s - \frac{2 \pi}{2k+1} \left( v_{2k+1}^{x,\mu} -l \right) \right) d \mu \\
& = \frac{2 \pi}{2k+1}l - s + \int_{S^1} \left( x - \frac{2 \pi}{2k+1} v_{2k+1}^{x,\mu} \right) d \mu\\
& = \frac{2 \pi}{2k+1}l - s + m_{2k+1}^{x}(\mu).
\end{align*}
From these, we can convert the following term in the definition of $F_{2k+1}$:
\[
\frac{2 \pi }{2k+1}v_{2k+1}^{x',\mu}([\theta]) + m_{2k+1}^{x'}(\mu) = \begin{cases}
\frac{2 \pi}{2k+1} v_{2k+1}^{x,\mu} + m_{2k+1}^{x}(\mu) - s + 2\pi & \text{ if $[\theta] \in A_0 \cup \dots \cup A_{l-1}$}\\
\frac{2 \pi}{2k+1} v_{2k+1}^{x,\mu} + m_{2k+1}^{x}(\mu) - s & \text{ if $[\theta] \in A_l \cup \dots \cup A_{2k}$}.
\end{cases}
\]
Along with the conversion for $x'([\theta])$, this gives
\begin{align*}
&\tau' \bigg( (1-t) \, x'([\theta]) + t \, \bigg(\frac{2 \pi }{2k+1}v_{2k+1}^{x',\mu}([\theta]) + m_{2k+1}^{x'}(\mu) \bigg) \bigg) \\
& \hspace{1cm} = \tau \bigg( (1-t) \, x([\theta]) + t \, \bigg(\frac{2 \pi }{2k+1}v_{2k+1}^{x,\mu}([\theta]) + m_{2k+1}^x(\mu) \bigg) \bigg),   
\end{align*}
where we have used the fact that $\tau(z) = \tau'(z-s)$ for all $z \in \mathbb{R}$ and $\tau'$ is $2 \pi$-periodic.
This shows the definition of $F_{2k+1}$ does not depend on the choice of coordinate system.

We next show $F_{2k+1}$ is continuous at an arbitrary point $([\theta'],\mu',t')$.
First, choose a $[\theta_0]$ such that $[\theta_0+\pi] \in \supp(\mu')$ (thus, $[\theta_0]$ is excluded by $\mu'$), and work with a coordinate system $x \colon S^1 - \{ [\theta_0] \} \to \mathbb{R}$ and $\tau$ such that $\tau \circ x = 1_{S^1 - \{ [\theta_0] \}}$.
By Lemma~\ref{lemma:close_measures_implies_close_masses}(\ref{lemma_item_support_points}), for any measure $\mu$ sufficiently close to $\mu'$, there is a point in $\supp(\mu)$ at distance less than $\pi-r$ from $[\theta_0+\pi]$ and thus excludes $[\theta_0]$ as well.
Therefore we may use the same coordinate system $(x, \tau)$ in some neighborhood of $\mu'$.
Since $x$ and $\tau$ are continuous and the argument for $\tau$ in the definition of $F_{2k+1}$ is defined by a linear homotopy in $\mathbb{R}$, it is sufficient to check that the function given by $([\theta],\mu) \mapsto \frac{2 \pi}{2k+1}v_{2k+1}^{x,\mu}([\theta]) + m_{2k+1}^x(\mu)$, defined on a neighborhood of $([\theta'],\mu')$, is continuous.
By Lemma~\ref{lemma:close_arc_masses}(\ref{arc_mass_lemma_item_close_arc_masees}), for any $\mu$ sufficiently close to $\mu'$, if $A_0, \dots, A_{2k}$ are defined by extending the $\mu$-arcs by $\frac{\pi-r}{2}$ on both sides, all points of $\supp(\mu')$ contained in a single $\mu'$-arc are contained in the same $A_i$, with distinct $\mu'$-arcs corresponding to distinct $A_i$.
This implies that if $[\theta''] \in \supp(\mu)$ and $d_{S^1}(\theta',\theta'') < \pi-r$, then $ v_{2k+1}^{x,\mu'}([\theta']) = v_{2k+1}^{x,\mu}([\theta''])$, and thus we now only need to show the function $\mu \mapsto m_{2k+1}^x(\mu)$, defined on some neighborhood of $\mu'$, is continuous.
With $\mu$ and $A_0, \dots, A_{2k}$ as above, we have
\begin{align*}
m_{2k+1}^x(\mu') - m_{2k+1}^x(\mu) & = \int_{S^1} \left( x - \frac{2 \pi}{2k+1} v_{2k+1}^{x,\mu'} \right) d \mu'  -  \int_{S^1} \left( x - \frac{2 \pi}{2k+1} v_{2k+1}^{x,\mu} \right) d \mu  \\
& = \int_{S^1} x \,d\mu' - \sum_{i=0}^{2k} \frac{2i \pi}{2k+1}\mu'(A_i) - \int_{S^1} x \,d\mu + \sum_{i=0}^{2k} \frac{2i \pi}{2k+1}\mu(A_i) \\
&= \int_{S^1} x \,d\mu' - \int_{S^1} x \,d\mu + \sum_{i=0}^{2k} \frac{2i \pi}{2k+1}\big(\mu(A_i) - \mu'(A_i)\big).
\end{align*}
For the integrals, $\int_{S^1} x \,d\mu$ approaches $\int_{S^1} x \,d\mu'$ as $\mu$ approaches $\mu'$: this follows from Lemma~\ref{lemma:weak_convergence_and_Wasserstein_convergence} after replacing $x$ by an appropriate bounded continuous function without changing the values of the integrals.
For the sum, by Lemma~\ref{lemma:close_arc_masses}(\ref{arc_mass_lemma_item_close_arc_masees}), each $| \mu(A_i) - \mu'(A_i) |$ can be made arbitrarily small by choosing a sufficiently small neighborhood of $\mu'$.
Therefore the function $\mu \mapsto m_{2k+1}^x(\mu)$ is continuous, so we conclude that $F_{2k+1}$ is continuous.

Finally, to see that $F_{2k+1}$ satisfies the definition of a support homotopy (Section~\ref{subsection: vrm and support homotopies}), we must check that for any $\mu = \sum_{i=1}^n a_i \delta_{[\theta_i]} \in V_{2k+1}$ with $a_i>0$ for all $i$, and for any $t\in I$, we have $\sum_{i=1}^n a_i \delta_{F_{2k+1}([\theta_i],\mu,t)} \in V_{2k+1}$.
This amounts to checking that $\sum_{i=1}^n a_i \delta_{F_{2k+1}([\theta_i],\mu,t)}$ has diameter at most $r$ and has exactly $2k+1$ arcs.
First, supposing $V_{2k+1}$ is nonempty, we must have $r \geq \frac{2k \pi}{2k+1}$ by Proposition~\ref{prop:facts_about_arcs_and_V_2k+1}. 
For the diameter bound, consider any two points in $\supp(\mu)$, without loss of generality writing them as $[\theta_{1}]$ and $[\theta_{2}]$.
Choose a coordinate system $(x,\tau)$ with a corresponding ordered set of $\mu$-arcs $A_0, \dots, A_{2k}$ so that, without loss of generality, $[\theta_{1}] \in A_0$ and $[\theta_{2}] \in A_j$ with $0\leq j \leq k$.
Then $| x([\theta_{1}]) - x([\theta_{2}]) | \leq r$.
For any $t\in I$, we apply the fact that $\tau$ is $1$-Lipschitz (Section~\ref{subsection: coordinates on the circle}), giving the following bound:

\begin{align*}
&\, d_{S^1}\big(F_{2k+1}([\theta_{1}],\mu,t), F_{2k+1}([\theta_{2}],\mu,t)\big) \\
= &\, d_{S^1}\left(\tau\left((1-t)x([\theta_{1}]) + t\, m_{2k+1}^x(\mu)\right), \tau \left((1-t)x([\theta_{2}]) + t\,\left( \frac{2j \pi}{2k+1} + m_{2k+1}^x(\mu) \right) \right)\right) \\
\leq &\, \biggr| (1-t)x([\theta_{1}]) + t\, m_{2k+1}^x(\mu) - \left( 
(1-t)x([\theta_{2}]) + t\,\left( \frac{2j \pi}{2k+1} + m_{2k+1}^x(\mu) \right) \right) \biggr| \\
\leq &\, (1-t) \biggr| x([\theta_{1}]) - x([\theta_{2}]) \biggr| + t \biggr| m_{2k+1}^x(\mu)  - \left( \frac{2j \pi}{2k+1} + m_{2k+1}^x(\mu) \right) \biggr| \\
\leq &\, (1-t) r + t \frac{2j \pi}{2k+1} \\
\leq &\, (1-t) r + t \frac{2k \pi}{2k+1} \\
\leq &\, (1-t) r + t\, r \\
= &\, r. \\
\end{align*}
Therefore, $\sum_{i=1}^n a_i \delta_{F_{2k+1}([\theta_i],\mu,t)}$ has diameter at most $r$.

To see that each $\sum_{i=1}^n a_i \delta_{F_{2k+1}([\theta_i],\mu,t)}$ has $2k+1$ arcs, we first associate to any nonempty finite subset $\Theta \subset S^1$ of diameter at most $r$ a continuous map $f_\Theta \colon S^1 \to S^1$.
Color the points of $\Theta$ blue and the points opposite them red.
Let $f_{\Theta}$ send each blue point to $[0]$ and each red point to $[\pi]$.
On any arc between consecutive colored points that are the same color, let $f_{\Theta}$ remain constant at the value of the endpoints.
On an arc between consecutive colored points with opposite colored endpoints, let the angle of $f_{\Theta}([\theta])$ increase at a constant rate as $\theta$ increases, such that it increases by $\pi$ across the length of the arc.
Since each blue point is at a distance of at least $\pi-r$ from each red point, $f_{\Theta}$ is $\frac{\pi}{\pi-r}$-Lipschitz.
We can see that $\arcs_r(\Theta)$ is equal to the degree of $f_\Theta$.
Letting $\Theta_t = \{ F_{2k+1}([\theta_i],\mu,t) \mid 1 \leq i \leq n \}$, we get a function $f_{\Theta_t} \colon S^1 \to S^1$ for each $t$.
The continuity of $F_{2k+1}$ can be used to check that we get a continuous map $S^1 \times I \to S^1$ defined by $([\theta],t) \mapsto f_{\Theta_t}([\theta])$.
Thus, any $f_{\Theta_t}$ is homotopic to $f_{\Theta_0}$, so for each $t$, 
\[\arcs_r(\Theta_t) = \deg(f_{\Theta_t}) = \deg(f_{\Theta_0}) = \arcs_r(\Theta_0) = \arcs_r(\supp(\mu)) = 2k+1.
\]
This shows each $\sum_{i=1}^n a_i \delta_{F_{2k+1}([\theta_i],\mu,t)}$ has $2k+1$ arcs and completes the proof that $F_{2k+1}$ is a support homotopy.
\end{proof}

Applying Lemma~\ref{lemma_support_homotopy} to the support homotopy $F_{2k+1}$, we obtain a homotopy $\widetilde{F}_{2k+1} \colon V_{2k+1} \times I \to V_{2k+1}$, defined for $\mu = \sum_{i=1}^n a_i \delta_{[\theta_i]}$ with $a_i > 0$ for each $i$ by $\widetilde{F}_{2k+1}(\mu,t) = \sum_{i=1}^n a_i \delta_{F_{2k+1}([\theta_i],\mu,t)}$.
For each $\mu \in V_{2k+1}$, $\widetilde{F}_{2k+1}(\mu,1)$ is a measure supported on $2k+1$ evenly spaced points on the circle, and all masses in $\mu$ in a single $\mu$-arc are moved to a single one of these evenly spaced points.
Explicitly, following the notation above, for each $\mu \in V_{2k+1}$, we have
\[
\widetilde{F}_{2k+1}(\mu,0) = \mu
\]
\begin{equation}\label{equation:expression_for_F(_,1)}
\widetilde{F}_{2k+1}(\mu,1) = \sum_{i=0}^{2k} \mu(A_i) \delta_{\tau\left(\frac{2i \pi}{2k+1} + m_{2k+1}^{x}(\mu)\right)}.    
\end{equation}
Thus, $\widetilde{F}_{2k+1}(\_\,,1)$ sends $V_{2k+1}$ into $P_{2k+1}$.

\begin{figure}[h]
    \centering
    \includegraphics[width = .8\textwidth]{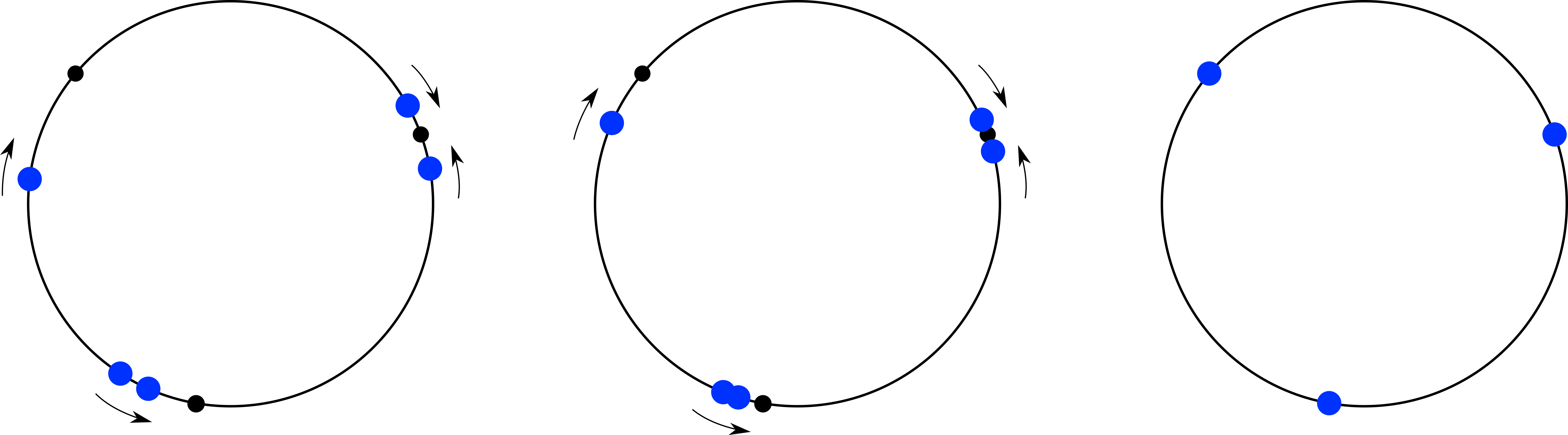}
    \caption{$\widetilde{F}_{2k+1}$ can be visualized for an individual measure by sliding the support points along the circle.
    This example shows $\widetilde{F}_{2k+1}(\mu,0) = \mu$, $\widetilde{F}_{2k+1}(\mu,\frac{1}{2})$, and $\widetilde{F}_{2k+1}(\mu,1)$ for a specific measure $\mu$.
    The blue points are the support points, which move until they reach the smaller black points.}
    \label{fig:homotopy}
\end{figure}

For each $k$, it can be checked that the homotopy $\widetilde{F}_{2k+1}$ is a deformation retraction, which is enough to show that $V_{2k+1} \simeq P_{2k+1}$.
However, this is not enough for our purposes, as we would like to collapse all the $V_{2k+1}$ while preserving the homotopy type of the entire space $\vrmcircleq$.
We will describe in Section~\ref{section:sequence_of_quotients} how this can be accomplished using Proposition~\ref{prop:quotient_homotopy_equivalences}, by defining equivalence relations that relate measures in $V_{2k+1}$ if they are sent to the same measure by $\widetilde{F}_{2k+1}(\_\,, 1)$.
To prepare for this use of Proposition~\ref{prop:quotient_homotopy_equivalences}, we prove the following lemma, which implies that each $\widetilde{F}_{2k+1}(\_\,, t)$ sends each fiber of $\widetilde{F}_{2k+1}(\_\,, 1)$ into the same fiber.

\begin{lem}\label{lemma:partial_homotopy_keeps_same_destination}

For any $r \in [0,\pi)$ such that $V_{2k+1}(r)$ is nonempty, any $k \geq 0$, any $t \in I$, and any $\mu \in V_{2k+1}(r)$, we have 
\[
\widetilde{F}_{2k+1}(\widetilde{F}_{2k+1}(\mu,t),1) = \widetilde{F}_{2k+1}(\mu,1).
\]
\end{lem}

\begin{proof}
Let $\mu = \sum_{i=1}^n a_i \delta_{[\theta_i]}$ with $a_i > 0$ for each $i$.
We will show the claimed equation holds at a fixed $t_0 \in I$.
We first show we can find a coordinate system that can be used for each computation of $\widetilde{F}_{2k+1}$.
Temporarily, we define a \textit{reduced $\mu$-arc} to be the smallest closed arc containing all the points of $\supp(\mu)$ contained in a given $\mu$ arc; that is, its endpoints are the outermost support points of the $\mu$-arc.
For any $\mu \in V_{2k+1}$, we know that $\widetilde{F}_{2k+1}$ collapses the masses of each reduced $\mu$-arc to a single point.
If some reduced $\mu$-arc contains the point it is collapsed to, let $[\theta_0]$ be the point opposite it (note that this is the only possible case when $k=0$).
Otherwise, suppose each reduced $\mu$-arc is collapsed to a point outside it and hence, within each reduced $\mu$-arc, all support points are moved in the same direction by $F_{2k+1}$.
Since $m_{2k+1}^x$ is defined by a weighted average, we can show not all support points are moved clockwise and not all are moved counterclockwise.
This can be seen using any valid coordinate system $(x, \tau)$: by the definition of $F_{2k+1}$, a point $[\theta_i] \in \supp(\mu)$ is moved in $x$ values from $x([\theta_i])$ to $\frac{2 \pi }{2k+1}v_{2k+1}^{x,\mu}([\theta_i]) + m_{2k+1}^x(\mu)$.
Since $\sum_{i=1}^n a_i \left( x([\theta_i]) - \left(\frac{2 \pi }{2k+1}v_{2k+1}^{x,\mu}([\theta_i]) + m_{2k+1}^x(\mu) \right)\right) = 0$ by the definition of $m_{2k+1}$, not all support points move in the same direction.
Thus, beginning with an arc that is moved clockwise and reading counterclockwise around the circle until we reach the first arc that is moved counterclockwise, we can find two reduced $\mu$-arcs $A$ and $A'$ such that $A'$ is the $\mu$-arc immediately counterclockwise from $A$, $A$ is moved clockwise, and $A'$ is moved counterclockwise.
Because these are contained in distinct $\mu$ arcs, there must be a point excluded by $\mu$ between the two, immediately counterclockwise of $A$ and clockwise of $A'$; let $[\theta_0]$ be any such point.
In either case, we can check that no mass is moved through $[\theta_0]$ by $\widetilde{F}_{2k+1}(\mu, \_\,)$, so the coordinate system $(x,\tau)$ with $x \colon S^1 -\{[\theta_0]\} \to \mathbb{R}$ is a valid choice of coordinate system for both $\mu$ and $\widetilde{F}_{2k+1}(\mu,t)$ for the computation of $\widetilde{F}_{2k+1}$.
In the notation of Section~\ref{subsection: coordinates on the circle}, we can choose $y_0$ to be $0$, so that the image of $x$ is $(0,2\pi)$ .
By the choice of $[\theta_0]$, the expression $(1-t) \, x([\theta]) + t \, \left(\frac{2 \pi }{2k+1}v_{2k+1}^{x,\mu}([\theta]) + m_{2k+1}^x(\mu) \right)$ used in the definition of $\widetilde{F}_{2k+1}$ produces values in $(0,2\pi)$ for all $t$ and all $[\theta] \in \supp(\mu)$.
This means we will be able to use the fact that $x \circ \tau$ restricted to the interval $(0, 2\pi)$ is the identity.

Equation~(\ref{equation:expression_for_F(_,1)}) above shows 
\[
\widetilde{F}_{2k+1}(\mu,1) = \sum_{i=0}^{2k} \mu(A_i) \delta_{\tau\left(\frac{2i \pi}{2k+1} + m_{2k+1}^{x}(\mu)\right)}
\]
and
\[
\widetilde{F}_{2k+1}(\widetilde{F}_{2k+1}(\mu,t_0),1) = \sum_{i=0}^{2k} \widetilde{F}_{2k+1}(\mu,t_0)(A'_i) \delta_{\tau\left(\frac{2i \pi}{2k+1} + m_{2k+1}^{x}(\widetilde{F}_{2k+1}(\mu,t_0))\right)},
\]
where $A_0, \dots, A_{2k}$ are the arcs of $\mu$ and $A'_0, \dots, A'_{2k}$ are the arcs of $\widetilde{F}_{2k+1}(\mu,t_0),1)$, both ordered counterclockwise starting at $[\theta_0]$.
Since the arcs of $\mu$ and $\widetilde{F}_{2k+1}(\mu,t_0)$ remain in the same order and have the same amounts of mass, $\widetilde{F}_{2k+1}(\mu,t_0)(A'_i) = \mu(A_i)$ for each $i$.
Thus, it is sufficient to show that $m_{2k+1}^x(\mu) = m_{2k+1}^x(\widetilde{F}_{2k+1}(\mu,t_0))$.
By definition,
\[
m_{2k+1}^x(\mu) = \int_{S^1} \left(x - \frac{2\pi}{2k+1}v_{2k+1}^{x,\mu}\right) \,d \mu,
\]
and
\[
m_{2k+1}^x(\widetilde{F}_{2k+1}(\mu,t_0)) = \int_{S^1} \left(x - \frac{2\pi}{2k+1}v_{2k+1}^{x,\widetilde{F}_{2k+1}(\mu,t_0)}\right) \,d \widetilde{F}_{2k+1}(\mu,t_0).
\]
Again, the arcs of $\mu$ and $\widetilde{F}_{2k+1}(\mu,t_0)$ remain in the same order and have the same amounts of mass, so the terms $v_{2k+1}^{x,\mu}$ and $v_{2k+1}^{x,\widetilde{F}_{2k+1}(\mu,t_0)}$ integrate to the same value.
We thus need to show that
\[
\int_{S^1} x \,d \mu
=
\int_{S^1} x \,d \widetilde{F}_{2k+1}(\mu,t_0).
\]
By definition, if $\mu = \sum_{i=1}^n a_i \delta_{[\theta_i]}$ with $a_i > 0$ for each $i$, then $\widetilde{F}_{2k+1}(\mu,t_0) = \sum_{i=1}^n a_i \delta_{F_{2k+1}([\theta_i],\mu,t_0)}$.
We compute, applying the fact that $x \circ \tau$ restricted to the interval $(0, 2\pi)$ is the identity:
\begin{align*}
\int_{S^1} x \,d \widetilde{F}_{2k+1}(\mu,t_0) &= \sum_{i=1}^n a_i x(F_{2k+1}([\theta_i]),\mu,t_0)\\
&= \sum_{i=1}^n a_i x \circ \tau \left( (1-t_0) \, x([\theta_i]) + t_0 \, \left(\frac{2 \pi }{2k+1}v_{2k+1}^{x,\mu}([\theta_i]) + m_{2k+1}^x(\mu) \right) \right)\\
&= \sum_{i=1}^n a_i \left( (1-t_0) \, x([\theta_i]) + t_0 \, \left(\frac{2 \pi }{2k+1}v_{2k+1}^{x,\mu}([\theta_i]) + m_{2k+1}^x(\mu) \right) \right)\\
&= (1-t_0)\int_{S^1} x \,d\mu + t_0 \left(\int_{S^1} \frac{2 \pi }{2k+1}v_{2k+1}^{x,\mu} \,d\mu \,+\,m_{2k+1}^x(\mu) \right)\\
&=\int_{S^1} x \,d\mu,
\end{align*}
where the last step uses the definition of $m_{2k+1}^x(\mu)$.
\end{proof}

\vspace{.15cm}
\section{A Sequence of Quotients}\label{section:sequence_of_quotients}

Having defined homotopies $\widetilde{F}_{2k+1}\colon V_{2k+1} \times I \to V_{2k+1}$ that collapse the $V_{2k+1}$ to measures supported on regularly spaced points, we now show how to collapse all $V_{2k+1}$ at once in a way that preserves the homotopy type.
There is not necessarily a natural way to extend a given $\widetilde{F}_{2k+1}$ continuously to all of $\vrmcircleq$.
However, it turns out that proceeding one $k$ at a time, we can identify points with equal images under $\widetilde{F}_{2k+1}$ while preserving the homotopy type, which produces a much simpler space.
We introduce a sequence of quotient maps as follows.

\[
\begin{tikzcd}
\vrmcircleq \arrow[dd, "1_{\vrmcircleq}"'] \arrow[rdd, "q_1"] \arrow[rrdd, "q_3"] \arrow[rrrrdd, "q_{2K+1}"] &                                                  &                             &                                          &                         \\
                                                                                             &                                                  &                             &                                          &                         \\
\vrmcircleq \arrow[r, "\widetilde{q}_1"']                                                            & \frac{\vrmcircleq}{\sim_1} \arrow[r, "\widetilde{q}_3"'] & \frac{\vrmcircleq}{\sim_3} \arrow[r, "\widetilde{q}_5"'] & \dots \arrow[r, "\widetilde{q}_{2K+1}"'] & \frac{\vrmcircleq}{\sim_{2K+1}}
\end{tikzcd}
\]
\vspace{.2cm}

For each $k \geq 0$, let the equivalence relation $\sim_{2k+1}$ on $\vrmcircleq$ be defined by $\mu_1 \sim_{2k+1} \mu_2$ if and only if $\mu_1 = \mu_2$ or for some $l \leq k$, $\mu_1$ and $\mu_2$ are in $V_{2l+1}$ and $\widetilde{F}_{2l+1}(\mu_1,1) = \widetilde{F}_{2l+1}(\mu_2,1)$.
Let $q_1, \dots, q_{2K+1}$ be the associated quotient maps.
For convenience, we will also let $\sim_{-1}$ be equality and let $q_{-1}$ be the identity map on $\vrmcircleq$.
Because each equivalence relation respects the previous ones, we also get quotient maps $\widetilde{q}_1, \dots, \widetilde{q}_{2K+1}$.
We also note that for all $k$ and $l$, $W_{2l+1}$ is a closed, $q_{2k+1}$-saturated\footnote{Given a function $f \colon X \to Y$, a subset $U \subseteq X$ is called \textit{$f$-saturated}, or simply \textit{saturated}, if $U = f^{-1}(f(U))$.
A continuous, surjective function between topological spaces is a quotient map if and only if the image of each saturated open (closed) set is open (closed)~\cite{Munkres}.} subspace of $\vrmcircleq$, which implies the restriction $q_{2k+1}\vert_{W_{2l+1}} \colon W_{2l+1} \to q_{2k+1}(W_{2l+1})$ is a quotient map (Theorem 22.1 of~\cite{Munkres}).
Our aim is to show that each quotient $\widetilde{q}_{2k+1}$ is a homotopy equivalence.

We extend the composition $q_{2k-1} \circ \widetilde{F}_{2k+1} \colon V_{2k+1} \times I \to q_{2k-1}(V_{2k+1})$ to the following map so that we will be able to apply Proposition~\ref{prop:quotient_homotopy_equivalences}.
For each $k\geq 0$, define $G_{2k+1} \colon W_{2k+1} \times I \to q_{2k-1}(W_{2k+1})$ by
\[
G_{2k+1}(\mu,t) = \begin{cases}
q_{2k-1} \circ \widetilde{F}_{2k+1} (\mu,t) & \text{ if $\mu \in V_{2k+1}$} \\
q_{2k-1}(\mu) & \text{ if $\mu \in W_{2k-1}$}.
\end{cases}
\]
Thus, we have $\mu_1 \sim_{2k+1} \mu_2$ if and only if $G_{2k+1}(\mu_1,1) = G_{2k+1}(\mu_2,1)$.

Checking that each $G_{2k+1}$ is continuous will be tedious, so we place the proof of continuity in Appendix~\ref{appendix:continuity_of_G}.
The intuition for the continuity of $G_{2k+1}$ is as follows.
We can reduce to checking continuity at each point in $\partial V_{2k+1} \times I$.
Since $\widetilde{F}_{2k+1}(\_\,, 1)$ performs an averaging operation on measures of $V_{2k+1}$ and $q_{2k-1}$ identifies measures with the same averages under the various $\widetilde{F}_{2l+1}(\_\,, 1)$ with $l < k$, we need to check that these averages are compatible with each other (where for $\widetilde{F}_{2k+1}$ we actually need to consider a limit as we approach $\partial V_{2k+1}$).
This compatibility is analogous to the fact that to take a weighted average in $\mathbb{R}$, we can perform the sum in any order, and in particular, averaging certain subsets of points first does not change the final average.
Since the averaging operation performed by each $\widetilde{F}_{2l+1}$ depends on taking weighted averages of coordinates in $\mathbb{R}$, it is reasonable to expect that the various averages are in fact compatible.

We proceed with our goal of showing each $\widetilde{q}_{2k+1}$ is a homotopy equivalence.
Letting $r \in [0, \pi)$ and $0 \leq k \leq K(r)$, we will check that we can apply Proposition~\ref{prop:quotient_homotopy_equivalences} to the pair $\left( \frac{\vrmcircleq}{\sim_{2k-1}},\, q_{2k-1}(W_{2k+1}) \right)$ and the homotopy $\widetilde{G}_{2k+1}$ constructed below.
Proposition~\ref{prop:HEP_for_(vrm,W)} states that each pair $(\vrmcircleq, W_{2k+1})$ has the HEP.
By Proposition~\ref{prop:HEP_for_quotients}, since each $\mu \in \vrmcircleq - W_{2k+1}$ is only equivalent to itself under the equivalence relation $\sim_{2k-1}$, we find that each pair $\left( \frac{\vrmcircleq}{\sim_{2k-1}},\, q_{2k-1}(W_{2k+1}) \right)$ has the HEP.
Each $G_{2k+1}(\_\,, t) \colon W_{2k+1} \to q_{2k-1}(W_{2k+1})$ is constant on the equivalence classes of $\sim_{2k-1}$, so applying Lemma~\ref{lemma:homotopies_on_quotients} and the universal property of quotients, we get a homotopy $\widetilde{G}_{2k+1} \colon q_{2k-1}(W_{2k+1}) \times I \to q_{2k-1}(W_{2k+1})$ defined by $\widetilde{G}_{2k+1}(q_{2k-1}(\mu),t) = G_{2k+1}(\mu,t)$.
Specifically,
\[
\widetilde{G}_{2k+1}(q_{2k-1}(\mu),t) = \begin{cases}
q_{2k-1} \circ \widetilde{F}_{2k+1} (\mu,t) & \text{ if $\mu \in V_{2k+1}$} \\
q_{2k-1}(\mu) & \text{ if $\mu \in W_{2k-1}$}.
\end{cases}
\]
Thus, $\widetilde{G}_{2k+1}(q_{2k-1}(V_{2k+1}) \times I) \subseteq q_{2k-1}(V_{2k+1})$ and $\widetilde{G}_{2k+1}(q_{2k-1}(W_{2k-1}) \times I) \subseteq q_{2k-1}(W_{2k-1})$, where we can note that $q_{2k-1}(V_{2k+1})$ and $q_{2k-1}(W_{2k-1})$ are disjoint.
Furthermore, equivalence classes of $V_{2k+1}$ with respect to $q_{2k-1}$ are singletons, so for $\mu_1, \mu_2 \in V_{2k+1}$, we have $\widetilde{G}_{2k+1}(q_{2k-1}(\mu_1),1) = \widetilde{G}_{2k+1}(q_{2k-1}(\mu_2),1)$ if and only if $\widetilde{F}_{2k+1} (\mu_1,1) = \widetilde{F}_{2k+1} (\mu_2,1)$.
Therefore, the quotient map $\widetilde{q}_{2k+1} \colon \frac{\vrmcircleq}{\sim_{2k-1}} \to \frac{\vrmcircleq}{\sim_{2k+1}}$ described above identifies $q_{2k-1}(\mu_1)$ and $q_{2k-1}(\mu_2)$ if and only if $\widetilde{G}_{2k+1}(q_{2k-1}(\mu_1),1) = \widetilde{G}_{2k+1}(q_{2k-1}(\mu_2),1)$.
Finally, by Lemma~\ref{lemma:partial_homotopy_keeps_same_destination}, for any $t \in I$, we have $\widetilde{G}_{2k+1}(\widetilde{G}_{2k+1}(q_{2k-1}(\mu),t),1) = \widetilde{G}_{2k+1}(q_{2k-1}(\mu),1)$, so each $\widetilde{G}_{2k+1}(\_\,,t)$ sends each fiber of $\widetilde{G}_{2k+1}(\_\,,1)$ back into the same fiber.
Therefore, all conditions of Proposition~\ref{prop:quotient_homotopy_equivalences} apply to the pair of spaces $\left( \frac{\vrmcircleq}{\sim_{2k-1}},\, q_{2k-1}(W_{2k+1}) \right)$ and the homotopy $\widetilde{G}_{2k+1}$, so we conclude that $\widetilde{q}_{2k+1}$ is a homotopy equivalence.
By forming the composition $\widetilde{q}_{2K+1} \circ \dots \circ \widetilde{q}_{3} \circ \widetilde{q}_{1}$ of homotopy equivalences, we have thus proved the following theorem.
We now simplify notation, writing the final equivalence relation $\sim_{2K+1}$ above as $\sim$ and writing $q \colon \vrmcircleq \to \vrmcircleq / \sim$ for the quotient map.

\begin{thm}\label{theorem:vrm_homotopy_equivalent_to_quotient}
Define an equivalence relation $\sim$ on $\vrmleq{S^1}{r}$ by setting $\mu_1 \sim \mu_2$ if and only if for some $k \geq 0$, $\mu_1$ and $\mu_2$ are in $V_{2k+1}(r)$ and $\widetilde{F}_{2k+1}(\mu_1,1) = \widetilde{F}_{2k+1}(\mu_2,1)$.
Then $\vrmleq{S^1}{r} \simeq \vrmleq{S^1}{r} / \sim$.

\end{thm}

The quotient $\vrmleq{S^1}{r} / \sim$ is a much simpler space than $\vrmleq{S^1}{r}$.
Each measure is deformed to a regular polygonal measure by some $\widetilde{F}_{2k+1}$ and is equivalent to this measure under the equivalence relation.
This means every class in $\vrmleq{S^1}{r} / \sim$ can be represented by a regular polygonal measure.

\vspace{.15cm}
\section{The CW Complex and Homotopy Types}\label{section:cw_complex_and_homotopy_types}

We now show that each quotient $\vrmleq{S^1}{r} / \sim$ described in Theorem~\ref{theorem:vrm_homotopy_equivalent_to_quotient} has the topology of a CW complex, which will allow us to determine the homotopy types.
We will use the description of CW complexes from Proposition A.2 of~\cite{Hatcher}, which first requires that $\vrmleq{S^1}{r} / \sim$ be Hausdorff; this is not generally true of a quotient of a metric space, so the proof will depend on the construction of this particular quotient.

\begin{lem}\label{lemma:Hausdorff}

For each $0 \leq k \leq K(r)$, $\vrmleq{S^1}{r} / \sim_{2k+1}$ is Hausdorff.
\end{lem}

\begin{proof}
We will use induction on $k$.
Recall we defined $\sim_{-1}$ as equality, so that $\vrmcircleq / \sim_{-1} \cong \vrmcircleq$ is Hausdorff.
We use this as the base case.
For the inductive step, let $k \geq 0$ and suppose that $\vrmcircleq / \sim_{2k-1}$ is Hausdorff.
Supposing that $q_{2k+1}(\mu_1) \neq q_{2k+1}(\mu_2)$, we must find disjoint open neighborhoods of these points in $\vrmcircleq / \sim_{2k+1}$.
This is equivalent to finding $q_{2k+1}$-saturated, disjoint, open neighborhoods of $\mu_1$ and $\mu_2$ in $\vrmcircleq$.

We split into three cases.
If $\mu_1$ and $\mu_2$ are in $\vrmcircleq - W_{2k+1}$, then let $U_1 = B_{\vrmcircleq}(\mu_1,\varepsilon)$ and $U_2 = B_{\vrmcircleq}(\mu_2,\varepsilon)$, with $\varepsilon > 0$ small enough so that $U_1$ and $U_2$ are disjoint.
Then since $W_{2k+1}$ is closed in $\vrmcircleq$, $U_1 - W_{2k+1}$ and $U_2 - W_{2k+1}$ are open, disjoint neighborhoods of $\mu_1$ and $\mu_2$.  
They are $q_{2k+1}$-saturated since each element in $\vrmcircleq - W_{2k+1}$ is the only element in its equivalence class.

Next, suppose $\mu_1 \in W_{2k+1}$ and $\mu_2 \in \vrmcircleq - W_{2k+1}$.
Let ${U'_1 = \bigcup_{\mu \in W_{2k+1}} B_{\vrmcircleq}(\mu, \varepsilon)}$ and let $U'_2 = B_{\vrmcircleq}(\mu_2, \varepsilon)$, where $\varepsilon>0$ is chosen by Lemma~\ref{lemma:close_measures_implies_close_masses}(\ref{lemma_item_number_of_arcs}) so that all measures of $B_{\vrmcircleq}(\mu_2, 2 \varepsilon)$ have at least as many arcs as $\mu_2$.
Suppose for a contradiction that there is a $\nu \in U'_1 \cap U'_2$.
Then for some $\mu \in W_{2k+1}$, we have $\nu \in B_{\vrmcircleq}(\mu, \varepsilon)$, so $d_{W}(\mu,\mu_2) \leq d_{W}(\mu,\nu) + d_{W}(\nu, \mu_2) < 2\varepsilon$.
But this contradicts the choice of $\varepsilon$, since $\mu$ has at most $2k+1$ arcs and $\mu_2$ has greater than $2k+1$ arcs.
Therefore $U'_1$ and $U'_2$ are disjoint open neighborhoods of $\mu_1$ and $\mu_2$.
Furthermore, $W_{2k+1} \subseteq U'_1$ and $U'_2 \cap W_{2k+1} = \varnothing$ because all measures in $U'_2$ have at least as many arcs as $\mu_2$.
Therefore $U'_1$ and $U'_2$ are $q_{2k+1}$-saturated, again because each element in $\vrmcircleq - W_{2k+1}$ is the only element in its equivalence class.

Finally, we consider the case where $\mu_1$ and $\mu_2$ are both in $W_{2k+1}$.
Recall we have shown that $G_{2k+1} \colon W_{2k+1} \times I \to q_{2k-1}(W_{2k+1})$ is continuous and that $G_{2k+1}(\nu_1,1) = G_{2k+1}(\nu_2,1)$ if and only if $q_{2k+1}(\nu_1) = q_{2k+1}(\nu_2)$, for $\nu_1,\nu_2 \in W_{2k+1}$.
Since we have supposed $q_{2k+1}(\mu_1) \neq q_{2k+1}(\mu_2)$, we must have $G_{2k+1}(\mu_1,1) \neq G_{2k+1}(\mu_2,1)$.
By the inductive hypothesis, we can find disjoint open neighborhoods of $G_{2k+1}(\mu_1,1)$ and $G_{2k+1}(\mu_2,1)$ in $q_{2k-1}(W_{2k+1}) \subseteq \vrmcircleq / \sim_{2k-1}$; let $U''_1$ and $U''_2$ be their preimages under $G_{2k+1}(\_\,,1)$.
Then $U''_1$ and $U''_2$ are $q_{2k+1}$-saturated, disjoint, open subsets of $W_{2k+1}$ that contain $\mu_1$ and $\mu_2$ respectively.
We must extend these to open subsets of $\vrmcircleq$, so we will thicken around every point, as follows.
For each $\nu_1 \in U''_1$ and each $\nu_2 \in U''_2$, define
\begin{align*}
\varepsilon_1(\nu_1) &= \sup \{ \varepsilon \mid B_{W_{2k+1}}(\nu_1, \varepsilon) \subseteq U''_1 \} \\   
\varepsilon_2(\nu_2) &= \sup \{ \varepsilon \mid B_{W_{2k+1}}(\nu_2, \varepsilon) \subseteq U''_2 \}.
\end{align*}
These are always positive since $U''_1$ and $U''_2$ are open in $W_{2k+1}$, so we can obtain the following open sets of $\vrmcircleq$:
\begin{align*}
U'''_1 &= \bigcup_{\nu_1 \in U''_1} B_{\vrmcircleq}\left(\nu_1, \tfrac{1}{2} \varepsilon_1(\nu_1)\right) \\
U'''_2 &= \bigcup_{\nu_2 \in U''_2} B_{\vrmcircleq}\left(\nu_2, \tfrac{1}{2} \varepsilon_2(\nu_2)\right).
\end{align*}
If $\nu \in U'''_1 \cap W_{2k+1}$, then 
$\nu \in U''_1$ by choice of $\varepsilon_1(\nu_1)$, so we have $U'''_1 \cap W_{2k+1} = U''_1$.
Therefore $U'''_1$ is $q_{2k+1}$-saturated, since $U''_1$ is $q_{2k+1}$-saturated and each point not in $W_{2k+1}$ is the only element in its equivalence class.
Similarly, we see $U'''_2$ is $q_{2k+1}$-saturated.
To show $U'''_1$ and $U'''_2$ are disjoint, suppose $\nu \in U'''_1 \cap U'''_2$, so that $\nu \in B_{\vrmcircleq}(\nu_1, \frac{1}{2} \varepsilon_1(\nu_1)) \cap B_{\vrmcircleq}(\nu_2, \frac{1}{2} \varepsilon_2(\nu_2))$ for some $\nu_1 \in U''_1$ and $\nu_2 \in U''_2$.
Without loss of generality, suppose $\varepsilon_1(\nu_1) \geq \varepsilon_2(\nu_2)$, so that $d_{W}(\nu_1,\nu_2) < \frac{1}{2}(\varepsilon_1(\nu_1)+\varepsilon_2(\nu_2)) \leq \varepsilon_1(\nu_1)$.
Then by definition of $\varepsilon_1(\nu_1)$, we have $\nu_2 \in U''_1 \cap U''_2$, contradicting the fact that $U''_1$ and $U''_2$ are disjoint.
Therefore $U'''_1$ and $U'''_2$ are $q_{2k+1}$-saturated, disjoint, open neighborhoods of $\mu_1$ and $\mu_2$ in $\vrm{S^1}{r}$, as required.
\end{proof}

For the following lemma, recall that we have defined $R_{2k} \subseteq \vrmcircleq$ to be the set of measures with support equal to the regular $(2k+1)$-gon $\{ [0], [\frac{1 \cdot 2 \pi}{2k+1}], \dots, [\frac{2k \cdot 2\pi}{2k+1} ] \}$.

\begin{lem}\label{lemma:quotients_of_individual_polygons}
\sloppy
For each $k \geq 1$, and any $r \in [0,\pi)$ such that $R_{2k} \subseteq \vrmleq{S^1}{r}$, restricting $q$ gives a surjective map $q|_{\partial R_{2k}} \colon \partial R_{2k} \to q(W_{2k-1}(r))$.
If we further restrict the domain to $\partial R_{2k} \cap V_{2k-1}(r)$, then $q|_{\partial R_{2k} \cap V_{2k-1}(r)}$ is a bijection onto $q(V_{2k-1}(r))$.
\fussy
\end{lem}

\begin{proof}
To simplify notation, we will write $(z_0, z_1, \dots, z_{2k})$ for the measure $\sum_{i = 0}^{2k} z_i \delta_{[\frac{i \cdot 2\pi}{2k+1}]}$ and will refer to the masses being in positions $0$ through $2k$.
When we describe points between consecutive positions, we will mean points on the shorter arc, of length $\frac{2\pi}{2k+1}$, immediately between them (as opposed to the longer arc of length $\frac{2k \cdot 2\pi}{2k+1}$ on the other side of the circle).
Any equivalence class in $q(W_{2k-1})$ can be represented by a regular polygonal measure with at most $2k-1$ vertices, so we begin with an arbitrary set of masses $a_0, \dots, a_{2k-2}$ with $a_i \geq 0$ for each $i$ and $\sum_{i=0}^{2k-2} a_i = 1$.
We will write indices of the masses $a_i$ modulo $2k-1$ and the positions modulo $2k+1$.
To determine the arcs of a measure in $\partial R_{2k}$, we can use the fact that a position $i$ has nonzero mass in a measure $\mu$ if and only if the open arc between positions $i+k$ and $i+k+1$ contains a point excluded by $\mu$.

We begin with the measure $(a_0,0,a_1, \dots, a_{k-1}, 0, a_k, \dots, a_{2k-2})$ and gradually pass masses between the support points.
Define $\gamma_0 \colon [0,a_0] \to \partial R_{2k}$ by 
\[
\gamma_0(t) = (a_0 - t, t, a_1, \dots, a_{k-1}, 0, a_k, \dots, a_{2k-2}).
\]
Since there is zero mass at position $k+1$ throughout, this is in fact a map into $\partial R_{2k}$, and furthermore, there is no excluded point between positions $0$ and $1$.  
This shows positions $0$ and $1$ belong to the same arc of $\gamma_0(t)$ for each $t$.
Thus, if $k'$ is such that $\gamma_0(0) \in V_{2k'+1}$, then $\gamma_0(t) \in V_{2k'+1}$ for all $t \in [0, a_0]$, and if we consider $(2k'+1)$-arc mass forms, the arc masses of $\gamma_0(t)$ are the same for all values of $t$.
The measure $\gamma_0(a_0)$ has zero mass at positions $0$ and $k+1$, so positions $k$ and $k+1$ belong to the same $\gamma_0(a_0)$-arc.
We will next move mass between these positions, then repeat this process.  
In general, we obtain paths $\gamma_{l} \colon [0,a_{l(k-1)}] \to \partial R_{2k}$, defining $\gamma_{l}(t)$ by letting the masses as positions $lk$ through $lk+2k$ be, in order,
\[
a_{l(k-1)}-t, t, a_{l(k-1)+1}, \dots, a_{l(k-1)+k-1}, 0, a_{l(k-1)+k}, \dots, a_{l(k-1)+2k-2}.
\]
Note that the domain of $\gamma_l$ is the singleton $\{0\}$ if $a_{l(k-1)} = 0$.
Again, a mass of zero at position $lk+k+1$ implies that positions $lk$ and $lk+1$ are in the same arc, so the arc masses remain constant in each path.

\begin{figure}[h]
    \centering
    \fontsize{9pt}{11pt}
    \def\svgwidth{.9\textwidth}
    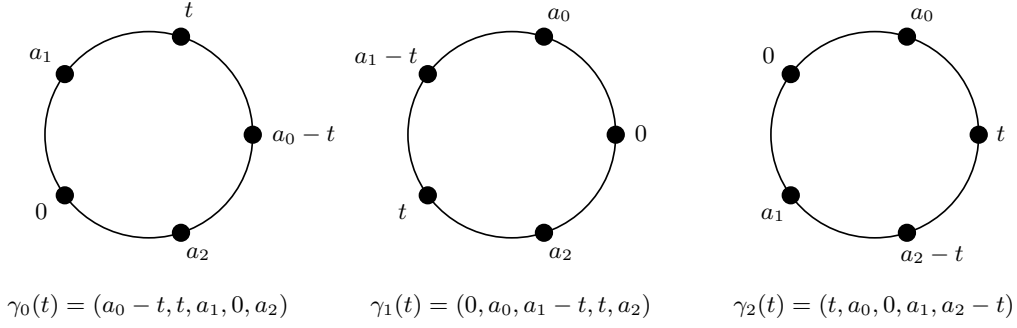
    \caption{Paths in the proof of Lemma~\ref{lemma:quotients_of_individual_polygons} with $k=2$.
    We have $\gamma_0(a_0) = \gamma_1(0)$ and $\gamma_1(a_1) = \gamma_2(0)$, so the paths may be concatenated.
    Compare $\gamma_0(0)$ to $\gamma_2(a_2)$: the masses are shifted by one position.}
    \label{fig:lemma_14}
\end{figure}

It can be checked that $\gamma_{l}(a_{l(k-1)}) = \gamma_{l+1}(0)$ for each $l$, so we may concatenate these paths; write the resulting path as $\gamma_l \cdot \gamma_{l+1}$.
Then the path $\gamma_0 \cdot \gamma_1 \cdot \cdot \cdot \gamma_{2k-2}$ preserves arc masses throughout and has starting point $(a_0,0,a_1, \dots, a_{k-1}, 0, a_k, \dots, a_{2k-2})$ and ending point $(a_{2k-2}, a_0,0,a_1, \dots, a_{k-1}, 0, a_k, \dots, a_{2k-3})$.
That is, the overall effect has been to move each mass over one position.
Repeating $2k+1$ times, we define $ \gamma = \gamma_0 \cdot \gamma_1 \cdot \cdot \cdot \gamma_{(2k-1)(2k+1)-1}$, which rotates each mass once around the circle; thus, $\gamma$ is a loop.
By scaling, we can assume the domain of $\gamma$ is $[0,1]$.
More generally, for any $l$ and $l'$, we have $\gamma_{l} = \gamma_{l'}$ if ${l \equiv l' \mod (2k-1)(2k+1)}$.

To see that $q|_{\partial R_{2k}}$ is surjective onto $q(W_{2k-1})$, take any equivalence class in $q(W_{2k-1})$ and choose a representative $\mu \in W_{2k-1}$ with support a regular $(2k'+1)$-gon, with $k' < k$.
We can choose $a_0, \dots a_{2k-2}$ and set $\nu = (a_0,0,a_1, \dots, a_{k-1}, 0, a_k, \dots, a_{2k-2})$ such that in $(2k'+1)$-arc mass form, the ordered arc masses of $\nu$ match those of $\mu$: for instance, if we choose $\nu$ to have nonzero masses at exactly positions $0, k, \dots, 2k'k$, then these positions lie in separate $\nu$-arcs, and we can choose the masses at these positions to match the arc masses of $\mu$.
Define each $\gamma_l$ and $\gamma$ as above with this choice of $a_0, \dots a_{2k-2}$.
Then $\gamma(0) = \nu$, and for any $t \in [0,1]$, since $\gamma(t)$ and $\gamma(0)$ have the same ordered arc masses, $\widetilde{F}_{2k'+1}(\gamma(t),1)$ is a measure with support a regular $(2k'+1)$-gon and ordered arc masses matching those of $\mu$.
Working in any coordinate system $x$ valid near some $\gamma(t)$, we can see that $m_{2k'+1}^x(\gamma(t))$ is strictly increasing in $t$, since $\gamma$ moves mass counterclockwise.
This implies that locally, the masses of $\widetilde{F}_{2k'+1}(\gamma(t),1)$ move strictly counterclockwise as $t$ increases.
Furthermore, $\gamma$ is a loop and each mass of $\widetilde{F}_{2k'+1}(\gamma(t),1)$ traverses the entire circle exactly once as $t$ ranges from $0$ to $1$, so there is some $t_0$ such that $\widetilde{F}_{2k'+1}(\gamma(t_0),1) = \mu$.
This shows that $q(\gamma(t_0)) = q(\mu)$, and since $\gamma(t_0) \in \partial R_{2k}$, we can conclude that $q|_{\partial R_{2k}}$ is surjective. 

To show that $q|_{\partial R_{2k} \cap V_{2k-1}}$ is injective a bijection onto $q(V_{2k-1})$, consider any equivalence class in $q(V_{2k-1})$ and let $\mu'$ be a representative measure with support a regular $(2k-1)$-gon.
We show the equivalence class of $\mu'$ intersects $\partial R_{2k} \cap V_{2k-1}$ in a single point.
Let $a'_0, \dots, a'_{2k-2}$ be the masses at the support points of $\mu'$, ordered counterclockwise around the circle, and note that $a'_i > 0$ for all $i$, since $\mu' \in V_{2k-1}$.
Define each $\gamma'_l$ like $\gamma_l$ above, with $a'_i$ in place of $a_i$ in each case, and define $\gamma'$ similarly to $\gamma$.
If $\nu' \in \partial R_{2k} \cap V_{2k-1}$ satisfies $q(\nu') = q(\mu')$, then $\nu'$ must have $2k-1$ arcs, so one of the positions will have zero mass.  
Then the opposite two positions are in the same $\nu'$-arc, and the remaining positions must each be in separate $\nu'$-arcs.
It follows that the masses of $\nu'$ are, in order counterclockwise and beginning with the two positions opposite a position with mass zero,
\[
a'_j-t, t, a'_{j+1}, \dots, a'_{j+k-1}, 0, a'_{j+k}, \dots, a'_{j+2k-2}
\]
for some $j$ and some $t \in [0, a'_j]$.
In fact, if $t = a'_j$, we could instead write the list above starting with $a'_{j+k-1}$ (or starting with the only nonzero mass if $k=1$), so the masses of $\nu'$ can actually be written as above with $t \in [0, a'_j)$.
Thus, we have $\nu' = \gamma'_l(t)$ for some $l$ and $t \in [0,a'_{l(k-1)})$: in particular, we can choose $0 \leq l < (2k-1)(2k+1)$ by the Chinese remainder theorem.
Therefore, every $\nu' \in \partial R_{2k} \cap V_{2k-1}$ satisfying $q(\nu') = q(\mu')$ is of the form $\nu' = \gamma'(t)$ for some $t \in [0,1)$, so it is sufficient to show that if $q(\gamma'(t_1)) = q(\gamma'(t_2))$, then $\gamma'(t_1) = \gamma'(t_2)$.

The simplest case occurs when the masses of $\mu'$ have no rotational symmetry: that is, there is no nontrivial cyclic permutation of its masses that leaves it unchanged.
In this case, since the masses of $\widetilde{F}_{2k-1}(\gamma'(t), 1)$ move strictly counterclockwise as $t$ increases and traverse the circle exactly once, there is a unique $t \in [0,1)$ such that $q(\gamma'(t)) = q(\mu')$.
Now consider the case where the masses of $\mu'$ have some nontrivial symmetry: let $j$ be the least positive integer dividing $2k-1$ such that $a'_{i+j} = a'_i$ for all $i$.
Then once again, since each mass of $\widetilde{F}_{2k-1}(\gamma'(t),1)$ traverses the circle exactly once as $t$ ranges from $0$ to $1$, there must be exactly $\frac{2k-1}{j}$ values of $t \in [0,1)$ such that $q(\gamma'(t)) = q(\mu')$.
We show that each of these values of $t$ yields the same value of $\gamma'(t)$.
It can be checked that for any $l$, $\gamma'_{l+2k+1}(t)$ is defined by the formula for $\gamma'_{l}(t)$ with each $a'_i$ replaced by $a'_{i-1}$.
Applying the assumed symmetry, $\gamma'_{l+(2k+1)j} = \gamma'_{l}$ for each $l$, so $\gamma'_l = \gamma'_{l'}$ if $l \equiv l' \mod{(2k+1)j}$.
Therefore, there must be an index $0 \leq l_0 < (2k+1)j$ and a $t_0 \in [0,a'_{l_0(k-1)})$ such that $q(\gamma'_{l_0}(t_0)) = q(\mu')$.
Furthermore, we have defined $\gamma' = \gamma'_0 \cdot \gamma'_1 \cdot \cdot \cdot \gamma'_{(2k-1)(2k+1)-1}$ and we have $\gamma'_{l_0+n(2k+1)j}(t_0) = \gamma'_{l_0}(t_0)$ for each $0 \leq n < \frac{2k-1}{j}$. 
Therefore, the $\frac{2k-1}{j}$ values of $t \in [0,1)$ such that $q(\gamma'(t)) = q(\mu')$ all satisfy $\gamma'(t) = \gamma'_{l_0}(t_0)$, so there is exactly one $\nu' \in \partial R_{2k} \cap V_{2k-1}$ such that $q(\nu') = q(\mu')$. 
\end{proof}

We can now describe $\vrmcircleq / \sim$ as a simple CW complex, with one cell in each dimension from $0$ to $2K+1$.
See Figure~\ref{figure:cell_structure} for an illustration of the case of $K=1$.
We partition $\vrmcircleq / \sim$ into cells $C_0, \dots, C_{2K+1}$: for $0 \leq k \leq K$, define
\begin{align*}
C_{2k} &= q(R_{2k})\\   
C_{2k+1} &= q(V_{2k+1}) - q(R_{2k}).
\end{align*}
Since $q$ only identifies a measure with measures that have the same number of arcs, the collection of subspaces $q(V_{2k+1})$ for all $k \geq 0$ partitions $\vrmcircleq / \sim$, and thus the cells $C_0, \dots, C_{2K+1}$ partition $\vrmcircleq / \sim$ as well.
Since $\vrmcircleq / \sim$ is Hausdorff by Lemma~\ref{lemma:Hausdorff}, to give $\vrmcircleq / \sim$ the structure of a CW complex, it is sufficient to construct for each $n \geq 1$ a map from a closed $n$-disk $D^n$ into $\vrmcircleq / \sim$ such that the interior is mapped homeomorphically onto $C_n$ and the boundary is mapped into the union of the lower dimensional cells; see Proposition A.2 of~\cite{Hatcher}.
We will write each $n$-skeleton as $X_n = C_0 \cup \dots \cup C_n$, so for each $0 \leq k \leq K$, we have
\begin{align*}
X_{2k} &= q(W_{2k-1}) \cup q(R_{2k})\\
X_{2k+1} &= q(W_{2k+1})
\end{align*}

We consider the even dimensions first.
For $k=0$, the single $0$ cell is $q(R_{0}) = \{ q(\delta_{[0]}) \}$.
For each $k \geq 1$, choosing a homeomorphism $D^{2k} \to \overline{R_{2k}}$ that maps $S^{2k-1}$ homeomorphically onto $\partial R_{2k}$, we define the characteristic map $\Phi_{2k}$ by the following composition:
\[
\begin{tikzcd}
D^{2k} \arrow[r, "\cong"] & \overline{R_{2k}} \arrow[r, hook] & \vrmcircleq \arrow[r, "q"] & \vrmcircleq/\sim.
\end{tikzcd}  
\]
Combining Lemma~\ref{lemma:Hausdorff} with the closed map lemma, we find that $\Phi_{2k}$ is a closed map.
Since $q$ is injective on $R_{2k}$, $\Phi_{2k}$ maps the interior of $D^{2k}$ bijectively onto $C_{2k}$.
It can be checked\footnote{In general, if $f \colon X \to Y$ is a closed map and $A \subseteq X$ is an $f$-saturated set, then $f|_A \colon A \to f(A)$ is a closed map.
We will use this fact once more below.} that since $\Phi_{2k}$ is a closed map and the interior of $D^{2k}$ is $\Phi_{2k}$-saturated, it is in fact mapped homeomorphically onto $C_{2k}$.
Because $\partial R_{2k}$ consists of measures with less than $2k+1$ arcs, the boundary of $D_{2k}$ is sent into $q(W_{2k-1}) = X_{2k-1}$, as required.

For the odd dimensions, for any $k \geq 1$, we consider $D^{2k} \times I$ as a $(2k+1)$-cell and construct a map into $\vrmcircleq$.
We can choose a continuous, surjective map $D^{2k} \times I \to \overline{P_{2k+1}}$
that, for each $t \in I$, maps $D^{2k} \times \{t\}$ homeomorphically onto the set of measures with support contained in $\{ [\frac{t}{2k+1}2 \pi], [\frac{1+t}{2k+1}2 \pi], \dots, [\frac{2k+t}{2k+1} 2\pi] \}$ and maps $S^{2k-1} \times \{t\}$ to the set of such measures with zero mass at at least one of these points.
Thus, $I$ parameterizes the regular $(2k+1)$-gons and $D^{2k} \times \{0\}$ and $D^{2k} \times \{1\}$ are both mapped into $\overline{R_{2k}}$.
Define $\Phi_{2k+1}$ by the following composition:
\[
\begin{tikzcd}
D^{2k}\times I \arrow[r] & \overline{P_{2k+1}} \arrow[r, hook] & \vrmcircleq \arrow[r, "q"] & \vrmcircleq/\sim
\end{tikzcd}
\]
Each element of $q(V_{2k+1})$ can be represented by a unique measure in $P_{2k+1}$, so by an argument similar to the above, $\Phi_{2k+1}$ maps the interior of $D^{2k} \times I$ homeomorphically onto $C_{2k+1}$.
Furthermore, points in the boundary of $D^{2k} \times I$ are mapped into either $q(\partial P_{2k+1}) \subseteq q(W_{2k-1})$ or $q(R_{2k})$, so the boundary is mapped into $X_{2k}$.
We have thus shown $\vrmcircleq / \sim$ has the CW-complex structure described above.

We now find the homotopy types of the skeletons: we show for each $k \geq 0$ that $X_{2k} \simeq D^{2k} \simeq \{ \ast \}$ and $X_{2k+1} \simeq S^{2k+1}$.
We use induction on $k$ to construct, for each $k \geq 0$, a homotopy equivalence $\varphi_{2k+1} \colon X_{2k+1} \to S^{2k+1}$ that maps $X_{2k}$ to a point $z \in S^{2k+1}$ and maps the cell $C_{2k+1}$ homeomorphically onto $S^{2k+1} - \{z\}$.
For the base case, $q(R_{0}) = \{ q(\delta_{[0]}) \}$ is the single $0$-cell, and since $X_1 = q(W_1)$ is formed by gluing a $1$-cell to by its two boundary points to the zero cell, we in fact have a homeomorphism $X_1 \cong S^1$ that maps $C_1$ homeomorphically onto $S^1 - \{ [0] \}$.

\begin{figure}
    \centering
    \[
\begin{tikzcd}
S^{2k-1} \arrow[rrr, "\Phi_{2k} |_{S^{2k-1}}"] \arrow[rrdd, "1"] \arrow[dddd, hook'] &  &                                                            & X_{2k-1} \arrow[dddd, hook'] \arrow[rdd, "\varphi_{2k-1}"] &                            \\
                                                                                     &  &                                                            &                                                            &                            \\
                                                                                     &  & S^{2k-1} \arrow[dddd, hook'] \arrow[rr, "\hspace{.5cm} f"] &                                                            & S^{2k-1} \arrow[dddd]      \\
                                                                                     &  &                                                            &                                                            &                            \\
D^{2k} \arrow[rrr, "\Phi_{2k}"] \arrow[rrdd, "1"]                                    &  &                                                            & X_{2k} \arrow[rdd, "\psi", dashed]                         &                            \\
                                                                                     &  &                                                            &                                                            &                            \\
                                                                                     &  & D^{2k} \arrow[rr]                                          &                                                            & S^{2k-1} \sqcup_{f} D^{2k}
\end{tikzcd}
\]
    \caption{Diagram for determining the homotopy type of $X_{2k}$.  The front and back squares are pushouts and the map $\psi$ is a homotopy equivalence.}
    \label{fig:cube_diagram}
\end{figure}
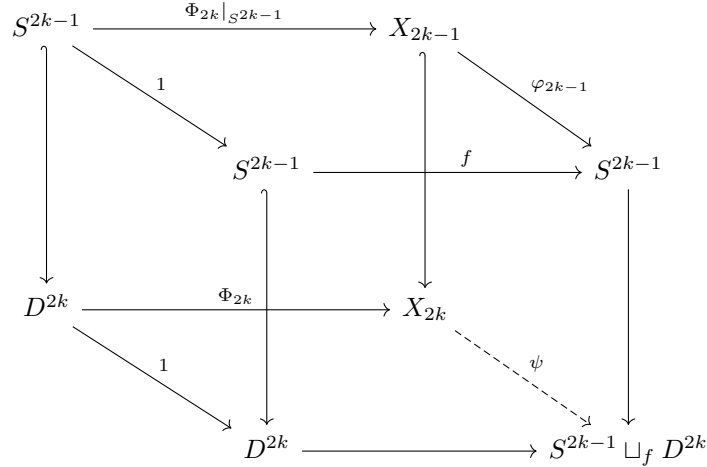

For the inductive step, let $k \geq 1$ and suppose $\varphi_{2k-1} \colon X_{2k-1} \to S^{2k-1}$ is a homotopy equivalence that maps $X_{2k-2}$ to a point $z \in S^{2k-1}$ and maps the cell $C_{2k-1}$ homeomorphically onto $S^{2k-1} - \{z\}$.
In the diagram in Figure~\ref{fig:cube_diagram}, $\Phi_{2k} \colon D^{2k} \to X_{2k}$ is the characteristic map defined above, with the codomain restricted. 
By Lemma~\ref{lemma:quotients_of_individual_polygons}, $\Phi_{2k}|_{S^{2k-1}}$ is surjective, and this implies $\Phi_{2k}$ is also surjective.
Since $\Phi_{2k}$ is a closed map, this implies it is a quotient map, and these facts can be used to check that the square in the diagram containing $\Phi_{2k}$ and $\Phi_{2k}|_{S^{2k-1}}$ is a pushout.
Letting $f = \varphi_{2k-1} \circ \Phi_{2k}|_{S^{2k-1}}$, the diagram commutes and both the front and back squares are pushouts.
By the gluing theorem for adjunction spaces (7.5.7 of ~\cite{Brown}), the resulting map $\psi \colon X_{2k} \to S^{2k-1} \sqcup_{f} D^{2k}$ defined by the universal property of pushouts is a homotopy equivalence.
Since $(D^{2k}, S^{2k-1})$ has the HEP, the homotopy type of $S^{2k-1} \sqcup_{f} D^{2k}$ depends only on the homotopy equivalence class of the map $f$ (Proposition 0.18 of~\cite{Hatcher}).
This, in turn, depends only on the degree of the map $f$ (see, for instance, Corollary 4.25 of ~\cite{Hatcher}), which we will find by considering the local degree at a suitable point.

Since $(q|_{\partial R_{2k}})^{-1} (C_{2k-1}) \subseteq (q|_{\partial R_{2k}})^{-1} (q(V_{2k-1})) \subseteq V^{2k-1} \cap \partial R_{2k}$,
Lemma~\ref{lemma:quotients_of_individual_polygons} shows $q$ restricts to a bijection from $(q|_{\partial R_{2k}})^{-1} (C_{2k-1})$ onto $C_{2k-1}$.
Furthermore, $\Phi_{2k}|_{S^{2k-1}}$ factors as 
\[
\begin{tikzcd}
S^{2k-1} \arrow[r, "\cong"] & \partial R_{2k} \arrow[r, "q|_{\partial R_{2k}}"] & q(W_{2k-1}),
\end{tikzcd}
\]
so $\Phi_{2k}|_{S^{2k-1}}$ restricts to a bijection from $(\Phi_{2k}|_{S^{2k-1}})^{-1}(C_{2k-1})$ onto $C_{2k-1}$.
Since $\Phi_{2k}$ is a closed map and $S^{2k-1}$ is $\Phi_{2k}$-saturated, $\Phi_{2k}|_{S^{2k-1}} \colon S^{2k-1} \to X_{2k-1}$ is also a closed map.
Similarly, since $(\Phi_{2k}|_{S^{2k-1}})^{-1}(C_{2k-1})$ is $\Phi_{2k}|_{S^{2k-1}}$-saturated, it can be checked that the restriction of $\Phi_{2k}|_{S^{2k-1}}$ to $(\Phi_{2k}|_{S^{2k-1}})^{-1}(C_{2k-1})$ is in fact a homeomorphism onto $C_{2k-1}$.
By the inductive hypothesis, we have $\varphi_{2k-1}^{-1}( S^{2k-1} - \{z\}) = C_{2k-1}$, and this cell is mapped homeomorphically onto $S^{2k-1} - \{z\}$ by $\varphi_{2k-1}$, so we can conclude that $f$ restricts to a homeomorphism from $f^{-1}(S^{2k-1} - \{z\})$ onto $S^{2k-1} - \{z\}$.
Therefore, the local degree of $f$ at any point in $f^{-1}(S^{2k-1} - \{z\})$ is $\pm 1$, which shows that the degree of $f$ is $\pm 1$ (see Proposition 2.30 of~\cite{Hatcher}).
This shows $S^{2k-1} \sqcup_{f} D^{2k}$ is homotopy equivalent to the space formed by gluing the boundary of $D^{2k}$ to $S^{2k-1}$ by the identity map, which is homeomorphic to $D^{2k}$.
Thus, we find $X_{2k} \simeq S^{2k-1} \sqcup_{f} D^{2k} \simeq D^{2k} \simeq \{ \ast \}$.

Finally, since CW pairs have the HEP and we have shown $X_{2k}$ is contractible, the quotient map $X_{2k+1} \to X_{2k+1} / X_{2k}$ is a homotopy equivalence by Proposition 0.17 of~\cite{Hatcher} (or by our Proposition~\ref{prop:quotient_homotopy_equivalences}).
In our case, $X_{2k+1}$ is obtained by gluing single a $(2k+1)$-cell by its boundary to $X_{2k}$, and thus we have the homotopy equivalence $\varphi_{2k+1}$ defined by the composition
\[
\begin{tikzcd}
X_{2k+1} \arrow[r] & X_{2k+1}/X_{2k} \arrow[r, "\cong"] & D^{2k+1}/S^{2k} \arrow[r, "\cong"] & S^{2k+1}.
\end{tikzcd}
\]
Furthermore, $\varphi_{2k+1}$ sends $X_{2k}$ to a single point of $S^{2k+1}$ and sends the cell $C_{2k+1}$ homeomorphically onto the remainder of $S^{2k+1}$, completing the inductive step.

We have thus found the homotopy types of the skeletons.
Furthermore, for $r \geq \pi$, it can be checked that $\vrmcircleq$ is contractible\footnote{One simple option to show $\vrmcircleq$ is contractible is to use a linear homotopy in the metric thickening.
Specifically, fix some $\mu_0 \in \vrmcircleq$ and define $L \colon \vrmcircleq \times I \to \vrmcircleq$ by $L(\mu,t) = (1-t)\mu + t\mu_0$.
Linear homotopies have played a significant role in previous work: see Lemma~3.9 of~\cite{AAF}, Lemma~4 of~\cite{Moy_Masters_Thesis}, and Proposition~2.4 of~\cite{vrp_arxiv}.}.
Recalling the $r$ values for which $V_{2k+1}(r)$ is nonempty, described in Proposition~\ref{prop:facts_about_arcs_and_V_2k+1}, we have proved the following theorem.

\begin{thm}\label{thm:vrm_homotopy_equivalent_to_spheres}
For each $k \geq 0$, if $V_{2k+1}(r)$ is nonempty, then $q(W_{2k+1}(r)) \simeq S^{2k+1}$.
This implies
\[
\vrmleq{S^1}{r} \simeq \begin{cases}
S^{2k+1} & \text{ if $r \in \big[\frac{2k \pi}{2k+1}, \frac{(2k+2)\pi}{2k+3}\big)$} \\
\{\ast\} & \text{ if $r \geq \pi$}.
\end{cases}
\]
\end{thm}

\vspace{.15cm}
\section{Persistent Homology}\label{section:persistent_homology}

Finally, we will address the inclusion maps between the metric thickenings as the parameter $r$ varies and find the associated persistent homology barcodes.
As mentioned above, the barcodes are already implied by previous work, as long as we disregard whether endpoints of bars are open or closed.
Thus, this section simply demonstrates that the techniques we have used are sufficient to find the barcodes directly; this will hopefully be of use in future research.

Here we must be careful to distinguish between the quotients we have constructed at different values of the parameter $r$.
Let $k \geq 0$ and let $r, r' \in \big[\frac{2k \pi}{2k+1}, \frac{(2k+2)\pi}{2k+3}\big)$ with $r \leq r'$.
Let the equivalence relations on $\vrmleq{S^1}{r}$ and $\vrmleq{S^1}{r'}$ described in Theorem~\ref{theorem:vrm_homotopy_equivalent_to_quotient} be denoted $\sim$ and $\sim'$ respectively, and let the corresponding quotient maps be $q$ and $q'$ respectively.
In both quotients $\vrmleq{S^1}{r} / \sim$ and $\vrmleq{S^1}{r'} / \sim'$, any equivalence class can be represented by a regular polygonal measure with at most $2k+1$ support points, and all such regular polygonal measures represent distinct equivalence classes.
Since the definition of each $\widetilde{F}_{2l+1}$ does not depend on the parameter $r$, if $\mu \in \vrmleq{S^1}{r} \subseteq \vrmleq{S^1}{r'}$, then $q(\mu)$ and $q'(\mu)$ are in fact represented by the same polygonal measure.
We thus have a homeomorphism $\vrmleq{S^1}{r} / \sim \,\, \to \vrmleq{S^1}{r'} / \sim'$ that sends the equivalence class of a regular polygonal measure in $\vrmleq{S^1}{r}$ to the equivalence class of the same regular polygonal measure in $\vrmleq{S^1}{r'}$ (note that this homeomorphism can be viewed as the natural homeomorphism of the CW complexes constructed in Section~\ref{section:cw_complex_and_homotopy_types}).
This homeomorphism makes the following diagram commute:

\[
\begin{tikzcd}
\vrmleq{S^1}{r} \arrow[rr, hook] \arrow[dd, "q"'] &  & \vrmleq{S^1}{r'} \arrow[dd, "q'"] \\
                                          &  &                          \\
\frac{\vrmleq{S^1}{r}}{\sim} \arrow[rr]           &  & \frac{\vrmleq{S^1}{r'}}{\sim'}.    
\end{tikzcd}
\]

The vertical maps are homotopy equivalences by Theorem~\ref{theorem:vrm_homotopy_equivalent_to_quotient} and the bottom map is the homeomorphism described above.
Therefore, after applying a singular homology functor $H_n$ in any dimension $n \geq 0$ and with coefficients in any fixed field, we obtain a commutative square in which each map is an isomorphism:

\[
\begin{tikzcd}
H_n(\vrmleq{S^1}{r}) \arrow[rr] \arrow[dd]   &  & H_n(\vrmleq{S^1}{r'}) \arrow[dd]   \\
                                 &  &                       \\
H_n\left(\frac{\vrmleq{S^1}{r}}{\sim}\right) \arrow[rr] &  & H_n\left(\frac{\vrmleq{S^1}{r'}}{\sim'}\right).
\end{tikzcd}
\]
By Theorem~\ref{thm:vrm_homotopy_equivalent_to_spheres}, both $\vrmleq{S^1}{r} / \sim$ and $\vrmleq{S^1}{r'} / \sim'$ are homotopy equivalent to $S^{2k+1}$.
From the homology of spheres, for any $r \in \big[\frac{2k \pi}{2k+1}, \frac{(2k+2)\pi}{2k+3}\big)$, we find that $H_{0}(\vrmleq{S^1}{r})$ and $H_{2k+1}(\vrmleq{S^1}{r})$ are 1-dimensional and the homology in all other dimensions is zero.
Applying these facts across all scale parameters $r$, this shows that the quotient maps induce an isomorphism of persistence modules between $H_n(\vrmleq{S^1}{\_\,})$ and $H_n\left(\frac{\vrmleq{S^1}{\_\,}}{\sim}\right)$.
In particular, for zero-dimensional homology, we note that the class of a fixed delta measure is a generator for all $r \geq 0$.
By considering all $k \geq 0$, we can find the persistent homology in all dimensions.
The barcodes are given in the following theorem.

\begin{thm}
\label{thm:vrm_barcodes}
The filtration $\vrmleq{S^1}{\_\,}$ of Vietoris--Rips metric thickenings of the circle has one persistent homology bar $[0, \infty)$ in dimension 0, one bar $\big[\frac{2k \pi}{2k+1}, \frac{(2k+2)\pi}{2k+3}\big)$ in dimension $2k+1$ for each $k \geq 0$, and no bars in the remaining dimensions.
\end{thm}

\begin{figure}[h]
    \centering
    \fontsize{9pt}{11pt}
    \def\svgwidth{.85\textwidth}
    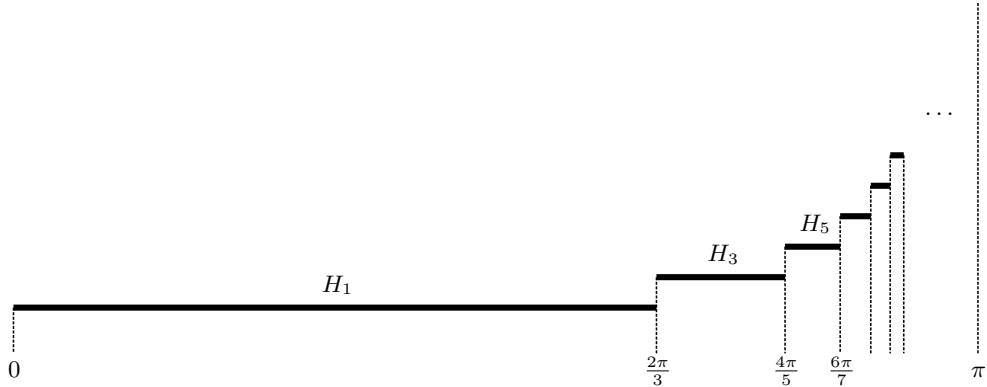
    \caption{Visualization of the reduced persistent homology bars of $\vrmleq{S^1}{\_\,}$.
    There is one bar in each odd dimension, corresponding to the homotopy types of odd-dimensional spheres.}
    \label{fig:barcodes}  
\end{figure}

\vspace{.1cm}
\section{Conclusion}
While certain techniques used here depend on specific properties of the circle, such as characterizing a measure $\mu$ by the number of $\mu$-arcs, some ideas will hopefully generalize to other settings.
The notion of a support homotopy already applies to more general Vietoris--Rips metric thickenings (and other simplicial metric thickenings), as we proved Lemma~\ref{lemma_support_homotopy} with only the requirement of a bounded metric space.
This suggests that the idea of collapsing a Vietoris--Rips metric thickening to a subset of representative measures, which we did via a support homotopy, could work in other spaces as well.
Here we used $\mu$-arcs to describe how to identify $\mu$ with a simple representative in a quotient.
In another space, a similar description of what portions of the space are excluded by a support point may lead to a similar, albeit possibly more complicated, way to choose a representative.

For those interested in pursuing other specific spaces, natural choices include spheres of higher dimensions.
Previous work indicates the homotopy types are likely to be more complicated than those of the circle: see Theorem~5.4 of~\cite{AAF}.
Another approach to future work would be to examine what techniques presented here can be generalized to reasonable classes of metric spaces, for instance, those that have the topology of a compact manifold.
Ideally, we would eventually be able to show that Vietoris--Rips metric thickenings (and possibly other simplicial metric thickenings) of metric spaces meeting reasonable conditions are homotopy equivalent to CW complexes, thus providing an analog of Morse theory.
Even in cases where the homotopy types are too difficult to describe fully, such a theory could provide bounds on the dimensions of homology modules and the numbers of persistent homology bars, analogous to the Morse inequalities.
Results in this direction will improve our understanding of not only Vietoris--Rips metric thickenings but also of Vietoris--Rips simplicial complexes, since in terms of persistent homology, these constructions are essentially interchangeable.

\newpage
\section{Acknowledgements}


I would like to thank Henry Adams for discussing this project many times and for reading the paper as it was in progress.

\vspace{1cm}
\bibliographystyle{plain}
\bibliography{main}

\begin{thebibliography}{10}

\bibitem{Adamaszek_clique_complexes}
Micha{\l} Adamaszek.
\newblock Clique complexes and graph powers.
\newblock {\em Israel Journal of Mathematics}, 196(1):295--319, August 2013.

\bibitem{AA_circle}
Michal Adamaszek and Henry Adams.
\newblock The {V}ietoris-{R}ips complexes of a circle.
\newblock {\em Pacific Journal of Mathematics}, 290, March 2015.

\bibitem{AAF}
Micha{\l} Adamaszek, Henry Adams, and Florian Frick.
\newblock Metric reconstruction via optimal transport.
\newblock {\em SIAM Journal on Applied Algebra and Geometry}, 2(4):597--619,
  2018.

\bibitem{ABF}
Henry Adams, Johnathan Bush, and Florian Frick.
\newblock Metric thickenings, {B}orsuk--{U}lam theorems, and orbitopes.
\newblock {\em Mathematika}, 66:79--102, 2020.

\bibitem{vrp_arxiv}
Henry {Adams}, Facundo {M{\'e}moli}, Michael {Moy}, and Qingsong {Wang}.
\newblock {The Persistent Topology of Optimal Transport Based Metric
  Thickenings}.
\newblock {\em arXiv e-prints arXiv:2109.15061}, September 2021.

\bibitem{AM}
Henry Adams and Joshua Mirth.
\newblock Metric thickenings of {E}uclidean submanifolds.
\newblock {\em Topology and its Applications}, 254:69--84, 2019.

\bibitem{Brown}
Ronald Brown.
\newblock {\em Topology and Groupoids}.
\newblock http://www.groupoids.org.uk/, Deganwy, United Kingdom, 2006.

\bibitem{chazalPersistenceStabilityGeometric2014}
Frédéric Chazal, Vin de~Silva, and Steve Oudot.
\newblock Persistence stability for geometric complexes.
\newblock {\em Geometriae Dedicata}, 173(1):193--214, 2014.

\bibitem{Hatcher}
Allen Hatcher.
\newblock {\em Algebraic Topology}.
\newblock Cambridge {U}niversity {P}ress, Cambridge, 2002.

\bibitem{Hausmann1995}
Jean-Claude Hausmann.
\newblock On the {V}ietoris--{R}ips complexes and a cohomology theory for
  metric spaces.
\newblock {\em Annals of Mathematics Studies}, 138:175--188, 1995.

\bibitem{katz1991neighborhoods}
Mikhail Katz.
\newblock On neighborhoods of the {K}uratowski imbedding beyond the first
  extremum of the diameter functional.
\newblock {\em Fundamenta Mathematicae}, 137(3):161--175, 1991.

\bibitem{Lee-Introduction_to_Topological_Manifolds}
John~M. Lee.
\newblock {\em Introduction to Topological Manifolds}.
\newblock Graduate Texts in Mathematics. Springer, 2011.

\bibitem{lim_memoli_okutan}
Sunhyuk {Lim}, Facundo {Memoli}, and Osman {Berat Okutan}.
\newblock {Vietoris-Rips Persistent Homology, Injective Metric Spaces, and The
  Filling Radius}.
\newblock {\em arXiv e-prints arXiv:2001.07588}, January 2020.

\bibitem{Moy_Masters_Thesis}
Michael Moy.
\newblock Persistence stability for metric thickenings.
\newblock Master's thesis, Colorado State University, March 2021.

\bibitem{Munkres}
James~R. Munkres.
\newblock {\em Topology}.
\newblock Prentice Hall, Inc., Upper Saddle River, NJ, second edition, 2000.

\bibitem{villani2003topics}
C{\'e}dric Villani.
\newblock {\em Topics in optimal transportation}.
\newblock Number~58 in Graduate Studies in Mathematics. American Mathematical
  Society, Providence, 2003.

\bibitem{Virk_footprints}
{\v{Z}}iga {Virk}.
\newblock {Footprints of geodesics in persistent homology}.
\newblock {\em arXiv e-prints}, page arXiv:2103.07158, March 2021.

\bibitem{zaremsky}
Matthew C.~B. {Zaremsky}.
\newblock {Bestvina-Brady discrete Morse theory and Vietoris-Rips complexes}.
\newblock {\em arXiv e-prints arXiv:1812.10976}, December 2018.

\end{thebibliography}


\newpage
\appendix

\vspace{.15cm}
\section{Proof of Lemma~\ref{lemma:closure_of_V_2k+1}}

Here we prove Lemma~\ref{lemma:closure_of_V_2k+1}, which states that $\mu \in \overline{V_{2k+1}(r)}$ if and only if $\supp(\mu)$ is contained in a finite set $T \subset S^1$ such that $\diam(T) \leq r$ and $\arcs_r(T) = 2k+1$.

\begin{proof}[Proof of Lemma~\ref{lemma:closure_of_V_2k+1}]
Let $C$ be the set of measures $\mu \in \vrmcircleq$ with $\supp(\mu)$ contained in some finite set $T \subset S^1$ such that $\diam(T) \leq r$ and $\arcs_r(T) = 2k+1$.
We must show $C = \overline{V_{2k+1}}$, and we start by noting that $V_{2k+1} \subseteq C$.  
We first show $C \subseteq \overline{V_{2k+1}}$.
We can write any $\alpha \in C$ in the form $\alpha = \sum_{i=0}^{n} a_i \delta_{[\theta_i]}$, with $a_i \geq 0$ for each $i$ and such that $\diam(\{ [\theta_0], \dots, [\theta_n] \}) \leq r$ and $\arcs_r(\{ [\theta_0], \dots, [\theta_n] \}) = 2k+1$ (note that some $a_i$ may be $0$, allowing for the case when the support of $\mu$ is strictly contained in $\{ [\theta_0], \dots, [\theta_n] \}$).
For each positive integer $j$, let $\alpha_j = \sum_{i=0}^n \left((1-\frac{1}{j})a_i + \frac{1}{j}\frac{1}{(n+1)}\right) \delta_{[\theta_i]}$. 
Then $\alpha_j \in V_{2k+1}$ for each $j$ and the sequence $\{ \alpha_j \}$ converges to $\alpha$, so $\alpha \in \overline{V_{2k+1}}$.

The remainder of the proof will handle the converse: we suppose $\mu \in \vrmcircleq - C$ and show $\mu \in \vrmcircleq - \overline{V_{2k+1}}$.
If $k=0$, then this is true since $V_1$ is closed (by Lemma~\ref{lemma:V_1_to_V_2k+1_closed}) and $V_1 \subseteq C$; thus, we may assume for the remainder of the proof that $k\geq 1$.
If $\arcs_r(\mu) > 2k+1$, then $\mu \in \vrmcircleq - W_{2k+1}$, so $\mu \in \vrmcircleq - \overline{V_{2k+1}}$ because Lemma~\ref{lemma:V_1_to_V_2k+1_closed} implies $\overline{V_{2k+1}} \subseteq W_{2k+1}$.
Thus, we consider the case where $\arcs_r(\mu) \leq 2k+1$, and in this case, we must in fact have $\arcs_r(\mu) < 2k+1$ because $V_{2k+1} \subseteq C$.  
Then for any finite set $T \subset S^1$ with $\diam(T) \leq r$ such that $\supp(\mu) \subseteq T$, we must have $\arcs_r(T) < 2k+1$, since $\mu \notin C$.

We examine the ways that points can be added to $\supp(\mu)$ to produce such a set $T$.
Begin by coloring the points of $\supp(\mu)$ blue and the points opposite them red.
From here on, whenever we color a point red or blue, we assume the point opposite it is colored the opposite color, and thus it is sufficient to describe colored points on half the circle.
Fix some blue point $[\theta] \in \supp(\mu)$, and let $A_1, \dots, A_N$ be all arcs between consecutive colored points on a fixed half of the circle between the blue point $[\theta]$ and the red point $[\theta + \pi]$. 
Then $\diam(A_i)$ is the length of the arc $A_i$ for each $i$.
In general, if a finite set of points on the circle are colored blue and the points opposite them are colored red, the set of blue points has diameter at most $r$ if and only if the distance between any blue point and any red point is at least $\pi-r$.
Following this restriction on distances, we search for a way to color additional points of an arc $A_i$ that produces the greatest increase in the number of arcs of the set of blue points.
Adding a pair of antipodal points, with one red and one blue, increases the number of arcs of the set of blue points by two if and only if the blue point is placed between consecutive colored points that are both red, which happens if and only if the red point is placed between consecutive colored points that are both blue.
If the endpoints of $A_i$ are both the same color, without loss of generality we let them be blue and note that after adding additional points, the increase in the number of arcs is equal to two times the number of new red points in $A_i$ immediately counterclockwise of a blue point.
There can be at most $\lfloor \frac{\diam(A_i)}{2(\pi-r)} \rfloor$
such red points because of the required distance between red and blue points, and this number of new red points can be achieved by placing points of alternating colors at distance $\pi-r$ from each other, beginning at one endpoint $A$ and continuing until no new red points can be placed.
Therefore $2\lfloor \frac{\diam(A_i)}{2(\pi-r)} \rfloor$ is the maximal increase in the number of arcs of the set of blue points that can be produced by coloring additional points of $A_i$, and this maximal increase can be achieved.
By similar reasoning, if one endpoint of $A_i$ is red and the other is blue, we find that the maximal increase is $2 \lfloor \frac{\diam(A_i) - (\pi-r)}{2(\pi-r)} \rfloor$.

If $T \subset S^1$ is any finite subset with $\diam(T) \leq r$ and such that $\supp(\mu) \subseteq T$, then $T$ can be obtained as a set of blue points meeting the description above.
Using the bounds on the maximal increases described above, we have
\[
\arcs_r(T) \leq \arcs_r(\mu) + \sum_{i \in I} 2 \Big\lfloor \frac{\diam(A_i)}{2(\pi-r)} \Big\rfloor + \sum_{i \in J} 2 \Big\lfloor \frac{\diam(A_i) - (\pi-r)}{2(\pi-r)} \Big\rfloor
\]
where $I$ is the set of all $i$ such that $A_i$ has endpoints of the same color and $J$ is the set of all $i$ such that $A_i$ has endpoints of opposite colors.
Furthermore, this bound is tight, since the maximal increase can be achieved for each $A_i$, so since $\mu \notin C$ implies $\arcs_r(T) < 2k+1$ for a $T$ producing the maximal increase in arcs, we have
\[
\arcs_r(\mu) + \sum_{i \in I} 2 \Big\lfloor \frac{\diam(A_i)}{2(\pi-r)} \Big\rfloor + \sum_{i \in J} 2 \Big\lfloor \frac{\diam(A_i) - (\pi-r)}{2(\pi-r)} \Big\rfloor
<
2k+1.
\]

We can now choose an $\varepsilon>0$ such that increasing any $\diam(A_i)$ by $2\varepsilon$ does not increase the value of any floor function above.
Specifically, choose $\varepsilon > 0$ so that
\[
\varepsilon < 
(\pi-r) \min_{i \in I} \Big(\Big\lfloor \frac{\diam(A_i)}{2(\pi-r)} \Big\rfloor + 1 - \frac{\diam(A_i)}{2(\pi-r)} \Big)
\]
and
\[
\varepsilon < 
(\pi-r) \min_{i \in J} \Big(\Big\lfloor \frac{\diam(A_i) - (\pi-r)}{2(\pi-r)} \Big\rfloor + 1 - \frac{\diam(A_i) - (\pi-r)}{2(\pi-r)} \Big)
,
\]
noting that each minimum is taken over a finite set of positive values.
By Lemma~\ref{lemma:close_measures_implies_close_masses}(\ref{lemma_item_support_points}), there exists a $\delta > 0$ such that if $\nu \in \vrmcircleq$ and $d_W(\mu,\nu) < \delta$, then each point of $\supp(\mu)$ has a point of $\supp(\nu)$ that is at distance less than $\varepsilon$.
For any such $\nu$, choose one such point in $\supp(\nu)$ for each point of $\supp(\mu)$ to define a set $U \subseteq \supp(\nu)$, and color the points of $U$ green and the points opposite them orange.
Shrinking $\varepsilon$ if necessary, we can assume each green point is within $\varepsilon$ of a unique blue point, and the ordering of the green and orange points matches the ordering of the corresponding blue and red points.
This implies that $\arcs_r(U) = \arcs_r(\mu)$; that the arcs $A_1, \dots, A_N$ above have corresponding arcs $A'_1, \dots, A'_N$ defined analogously for corresponding the green and orange points; and that for each $i$ the endpoints of $A'_i$ differ from the corresponding endpoints of $A_i$ by less than $\varepsilon$.
Since $U \subseteq \supp(\nu)$, $\arcs_r(\nu)$ can be bounded by the same method we used to bound $\arcs_r(T)$ above, replacing $A_i$ with $A'_i$ for each $i$.
For each $i$, $\diam(A'_i) < \diam(A_i) + 2 \varepsilon$, so by the choice of $\varepsilon$, we have 
\begin{align*}
\arcs_r(\nu) & \leq \arcs_r(U) + \sum_{i \in I} 2 \Big\lfloor \frac{\diam(A'_i)}{2(\pi-r)} \Big\rfloor + \sum_{i \in J} 2 \Big\lfloor \frac{\diam(A'_i) - (\pi-r)}{2(\pi-r)} \Big\rfloor\\
&=\arcs_r(\mu) + \sum_{i \in I} 2 \Big\lfloor \frac{\diam(A_i)}{2(\pi-r)} \Big\rfloor + \sum_{i \in J} 2 \Big\lfloor \frac{\diam(A_i) - (\pi-r)}{2(\pi-r)} \Big\rfloor \\
&<2k+1.
\end{align*}
This shows $\nu \notin V_{2k+1}$, so $\mu$ has an open neighborhood that does not intersect $V_{2k+1}$, and we can conclude $\mu \in \vrmcircleq - \overline{V_{2k+1}}$.
\end{proof}

\vspace{.15cm}
\section{Continuity of $G_{2k+1}$}\label{appendix:continuity_of_G}

We now return to check that each $G_{2k+1}$ is continuous.
The intuition is described in Section~\ref{section:sequence_of_quotients}.
The main challenge is that there is not a unique natural way to extend the definition of $\widetilde{F}_{2k+1}$ to $\partial V_{2k+1} \times I$, which makes it difficult to consider a limit as $\mu$ approaches $\partial V_{2k+1}$.
To handle this, we consider all sensible ways one could attempt to extend $\widetilde{F}_{2k+1}$ to a given point in $\partial V_{2k+1} \times I$ and find that there are finitely many.
This allows us to use a compactness argument to consider a limit as $\mu$ approaches $\partial V_{2k+1}$.

In the proof, it will be convenient to bound the 1-Wasserstein distance between measures by specifying how only part of the mass is transported.
Formally, this will be described by a \textit{partial matching} between measures $\mu = \sum_{i=1}^n a_i \delta_{[\theta_i]}$ and $\mu'= \sum_{j=1}^{n'} a'_j \delta_{[\theta'_j]}$, which is defined to be an indexed set ${\kappa = \{ \kappa_{i,j} \mid 1 \leq i \leq n, 1 \leq j \leq n' \}}$ of nonnegative real numbers such that $\sum_{i=1}^n \kappa_{i,j} \leq a'_j$ for all $j$ and $\sum_{j=1}^{n'} \kappa_{i,j} \leq a_i$ for all $i$.
A partial matching gives an incomplete description of how mass is transported from $\mu$ to $\mu'$, and the cost of a partial matching is defined in the same way as the cost of a matching.
Any partial matching from $\mu$ to $\mu'$ can be completed to a matching from $\mu$ to $\mu'$: that is, given a partial matching $\kappa$, there exists a matching $\kappa'$ such that $\kappa_{i,j} \leq \kappa'_{i,j}$ for all $i,j$.
The cost of transporting the remaining mass not accounted for by the partial matching $\kappa$ can be bounded using the diameter of $S^1$ (as a metric space): the maximum distance between two points of $S^1$ is $\pi$.

\begin{proof}[Proof of continuity of $G_{2k+1}$]

Note that $G_1 = \widetilde{F}_{1}$ is continuous, so we let $k \geq 1$.
It is sufficient to check sequential continuity for each point in $\partial V_{2k+1}$, since the continuity of $q_{2k-1}$ and $\widetilde{F}_{2k+1}$ imply that $G_{2k+1}$ is continuous on $V_{2k+1}$ and $W_{2k+1}-\overline{V_{2k+1}}$, which are open in $W_{2k+1}$.
Suppose $\{(\mu_n, t_n)\}_n$ is a sequence in $W_{2k+1} \times I$ that converges to $(\mu,t) \in \partial V_{2k+1} \times I$. 
We need to show $\{ G_{2k+1}(\mu_n,t_n) \}_n$ converges to $G_{2k+1}(\mu,t) = q_{2k-1}(\mu)$. 
For the subsequence consisting of those $(\mu_n, t_n)$ in $W_{2k-1} \times I$, we have $G_{2k+1}(\mu_n,t_n) = q_{2k-1}(\mu_n)$, and we can apply continuity of $q_{2k-1}$ to show this subsequence converges.
Thus, we can reduce to the case where $(\mu_n,t_n) \in V_{2k+1} \times I$ for all $n$.

Let $l < k$ be such that $\mu \in V_{2l+1}$.
By Lemma~\ref{lemma:close_measures_implies_close_masses}(\ref{lemma_item_close_arc_masses}), for any $\mu' \in V_{2k+1}$ sufficiently close to $\mu$, if we extend the arcs of $\mu'$ on either side by $\frac{\pi-r}{2}$, we obtain disjoint arcs $A_0, \dots, A_{2k}$ that collectively contain the support of $\mu$.
Let $(x,\tau)$ be a coordinate system that excludes a point not in these arcs and assume the arcs are in counterclockwise order starting from the excluded point.
As before, we let $v_{2k+1}^{x, \mu'} \colon S^1 \to \mathbb{R}$ be a function such that for any $[\theta]$ in some $A_i$, we have $[\theta] \in A_{v_{2k+1}^{x, \mu'}([\theta])}$.
With $\mu$ fixed as above, define
\[
m^{x,\mu'} = \int_{S^1} \left( x - \frac{2 \pi}{2k+1}v_{2k+1}^{x, \mu'} \right) d \mu
\]
and writing $\mu$ as $\mu = \sum_{i=1}^N a_i \delta_{[\theta_i]}$, define $J^{x, \mu'} \colon I \to V_{2l+1}$ by
\[
J^{x, \mu'}(t) = \sum_{i=1}^N a_i \delta_{\tau((1-t) x([\theta_i]) + t(\frac{2 \pi}{2k+1} v_{2k+1}^{x, \mu'}([\theta_i])+ m^{x,\mu'}) )}.
\]
This mimics the definition of $\widetilde{F}_{2k+1}$, but applies it to $\mu$, which is not in $V_{2k+1}(r)$.
By an argument similar to that in the proof of Lemma~\ref{lemma:F_is_a_support_homotopy}, each $J^{x, \mu'}(t)$ is in fact in $V_{2l+1}$.
Since $J^{x, \mu'}$ is continuous, $J^{x, \mu'}(I)$ is compact.
Note that the only reason $J^{x, \mu'}$ depends on $\mu'$ is because of the use of $v_{2k+1}^{x, \mu'}$ in these definitions.
Since there are only finitely many points in $\supp(\mu)$ and finitely many indices of arcs they are assigned to by $v_{2k+1}^{x, \mu'}$, there are only finitely many sets $J^{x, \mu'}(I)$ that can be obtained from all possible $\mu'$.  
Taking the union of these finitely many  $J^{x, \mu'}(I)$ for all possible $\mu'$, we obtain a compact set $S \subseteq W_{2k+1}$.

Any open set of $q_{2k-1}(W_{2k+1})$ containing $G_{2k+1}(\mu,t) = q_{2k-1}(\mu)$ has a preimage equal to a $(q_{2k-1})$-saturated open subset $U \subseteq W_{2k+1}$ containing $\mu$.
For any such $U$, we show $S \subseteq U$ by showing $\widetilde{F}_{2l+1}(J^{x, \mu'}(t), 1) = \widetilde{F}_{2l+1}(\mu, 1)$ for each $t \in I$ and each $\mu'$ meeting the description above.
We mimic the proof of Lemma~\ref{lemma:partial_homotopy_keeps_same_destination}, omitting details.
As in the proof of Lemma~\ref{lemma:partial_homotopy_keeps_same_destination}, we can choose $(x,\tau)$ to be a valid coordinate system for both $\mu$ and $J^{x, \mu'}(t)$ and such that all points of $\supp(\mu)$ and $\supp(J^{x, \mu'}(t))$ are sent into $(0, 2\pi)$ by $x$.
Since the masses of the corresponding arcs of $\mu$ and $J^{x, \mu'}(t)$ agree, following Equation~\ref{equation:expression_for_F(_,1)} before Lemma~\ref{lemma:partial_homotopy_keeps_same_destination}, it is sufficient to check that $m_{2l+1}^x(\mu) = m_{2l+1}^x(J^{x, \mu'}(t))$.
If $A'_0, \dots, A'_{2l}$ are the arcs of $\mu$, we have $m_{2l+1}^x(\mu) = \int_{S^1} x \, d \mu - \sum_{i=0}^{2l} \frac{2i \pi}{2l+1} \mu(A'_i)$, and analogously for $m_{2l+1}^x(J^{x, \mu'}(t))$.
Again, since the arc masses of $\mu$ and $J^{x, \mu'}(t)$ agree, we only must check that $\int_{S^1} x\, d\mu = \int_{S^1} x\, d (J^{x, \mu'}(t))$.
Using the notation above for $\mu$ and $J^{x, \mu'}(t)$, we have
\begin{align*}
\int_{S^1} x \, d(J^{x, \mu'}(t)) &= \sum_{i=1}^N a_i x \circ \tau \left( (1-t)x([\theta_i])+t\left( \frac{2 \pi}{2k+1}v_{2k+1}^{x, \mu'}([\theta_i]) + m^{x,\mu'} \right) \right)  \\
&= \sum_{i=1}^N a_i \left( (1-t)x([\theta_i])+t\left( \frac{2 \pi}{2k+1}v_{2k+1}^{x, \mu'}([\theta_i]) + m^{x,\mu'} \right) \right)  \\
&=  (1-t)\int_{S^1} x \, d\mu + t \int_{S^1} \frac{2\pi}{2k+1} v_{2k+1}^{x, \mu'} d\mu + t\,m^{x,\mu'}\\
&= \int_{S^1} x \, d\mu,
\end{align*}
where the last equality follows from the definition of $m^{x, \mu'}$.

Therefore we have $S \subseteq U$, and since $S$ is compact and $U$ is open in $W_{2k+1}$, there exists\footnote{This is a general fact about compact subsets of metric spaces, which was also used in the proof of Lemma~\ref{lemma:HEP_for_V}.  See, for instance, Exercise 2 in Section 27 of~\cite{Munkres}.} an $\varepsilon > 0$ such that any point within $\varepsilon$ of $S$ is contained in $U$.
We will show that even though $\{ \widetilde{F}_{2k+1}(\mu_n, t_n) \}_n$ does not necessarily converge to a specific point in $S$, the points of the sequence become arbitrarily close to $S$ as $n$ approaches infinity and are thus contained in $U$ for all sufficiently large $n$.
Since $\mu_n$ approaches $\mu$, we can set $\mu' = \mu_n$ for all sufficiently large $n$. 
We can also make a choice of a coordinate system $(x,\tau)$ that meets the requirements above simultaneously for all $\mu_n$ with $n$ sufficiently large: for instance, let $x$ exclude a point opposite a point of $\supp(\mu)$.
Then we have
\[
m^{x,\mu_n} = \int_{S^1} \left( x - \frac{2 \pi}{2k+1}v_{2k+1}^{x, \mu_n} \right) d \mu
\]
\[
m_{2k+1}^{x}(\mu_n) = \int_{S^1} \left( x - \frac{2 \pi}{2k+1} v_{2k+1}^{x, \mu_n} \right) d \mu_n,
\]
where $m_{2k+1}^{x}$ is as defined in Section~\ref{section:collapse}.
Thus,
\[
J^{x,\mu_n}(t_n) = \sum_{i=1}^N a_i \delta_{\tau((1-t_n) x([\theta_i]) + t_n(\frac{2 \pi}{2k+1} v_{2k+1}^{x, \mu_n}([\theta_i])+ m^{x,\mu_n}) )},
\]
and if $\mu_n = \sum_{j=1}^{N_n} a_{n,j} \delta_{[\theta_{n,j}]}$, then
\[
\widetilde{F}_{2k+1}(\mu_n,t_n) = \sum_{j=1}^{N_n} a_{n,j} \delta_{\tau((1-t_n) x([\theta_{n,j}]) + t_n(\frac{2 \pi}{2k+1} v_{2k+1}^{x, \mu_n}([\theta_{n,j}])+ m_{2k+1}^{x}(\mu_n)) )}.
\]

We show that $\widetilde{F}_{2k+1}(\mu_n,t_n)$ becomes close to $S$ by showing the distance between $\widetilde{F}_{2k+1}(\mu_n,t_n)$ and $J^{x,\mu_n}(t_n)$ approaches zero as $n$ approaches infinity.
For all sufficiently large $n$, we will have a bound $| m^{x,\mu_n} - m_{2k+1}^{x}(\mu_n) | < \frac{\varepsilon}{2}$ by Lemma~\ref{lemma:close_arc_masses}(\ref{arc_mass_lemma_item_close_arc_masees}) and the fact that $\int_{S_1} x \, d\mu_n$ approaches $\int_{S_1} x \, d\mu$ as $n$ approaches infinity (by Lemma~\ref{lemma:weak_convergence_and_Wasserstein_convergence}, replacing $x$ with a suitable bounded continuous function that does not change the value of the integrals).
As long as $n$ is sufficiently large, we can define the arcs $A_0, \dots, A_{2k}$ as above with $\mu' = \mu_n$ and these arcs collectively contain $\supp(\mu)$ and $\supp(\mu')$.
We now fix $n$ and let $\{\kappa_{i,j}\}$ be an optimal matching between $\mu$ and $\mu_n$. 
Distinct arcs are separated by a distance of at least $\pi-r$, so a mass of no more than $\frac{d_W(\mu,\mu_n)}{\pi-r}$ may be transported between distinct arcs by $\{\kappa_{i,j}\}$.
Thus, letting $B = \{ (i,j) \mid v_{2k+1}^{x, \mu_n}([\theta_i]) = v_{2k+1}^{x, \mu_n}([\theta_{n,j}])\}$, we have
\[
\sum_{(i,j) \in B} \kappa_{i,j} \geq 1 - \frac{d_W(\mu,\mu_n)}{\pi-r}.
\]
We define a partial matching for the measures $J^{x,\mu_n}(t_n)$ and $\widetilde{F}_{2k+1}(\mu_n,t_n)$ by using the same values $\kappa_{i,j}$ for $(i,j) \in B$.
We will use the fact that for $(i,j) \in B$, the distance $| x([\theta_{i}]) - x([\theta_{n,j}]) |$ is the arc length between $[\theta_{i}]$ and $[\theta_{n,j}]$ in the arc $A_{v_{2k+1}^{x,\mu_n}([\theta_i])}$ containing them, so $| x([\theta_{i}]) - x([\theta_{n,j}]) | = d_{S^1}([\theta_i], [\theta_{n,j}])$.
Thus, the cost of this partial matching is bounded by
\begin{align*}
&\sum_{(i,j) \in B} \kappa_{i,j} 
d_{S^1} \Big( 
\tau((1-t_n) x([\theta_i]) + t_n(\tfrac{2 \pi}{2k+1} v_{2k+1}^{x, \mu_n}([\theta_i])+ m^{x,\mu_n}) ) ,\\
&\hspace{2.5cm} \tau((1-t_n) x([\theta_{n,j}]) + t_n(\tfrac{2 \pi}{2k+1} v_{2k+1}^{x, \mu_n}([\theta_{n,j}])+ m_{2k+1}^{x}(\mu_n)) )
\Big)\\
\leq &\sum_{(i,j) \in B} \kappa_{i,j} 
d_{\mathbb{R}} \Big( 
(1-t_n) x([\theta_i]) + t_n(\tfrac{2 \pi}{2k+1} v_{2k+1}^{x, \mu_n}([\theta_i])+ m^{x,\mu_n})  ,\\
&\hspace{2.5cm} (1-t_n) x([\theta_{n,j}]) + t_n(\tfrac{2 \pi}{2k+1} v_{2k+1}^{x, \mu_n}([\theta_{n,j}])+ m_{2k+1}^{x}(\mu_n)) 
\Big)\\
\leq & (1-t_n)\sum_{(i,j) \in B} \kappa_{i,j} | x([\theta_{i}]) - x([\theta_{n,j}]) | + t_n \sum_{(i,j) \in B} \kappa_{i,j} | m^{x,\mu_n} - m_{2k+1}^{x}(\mu_n) |
\\
= & (1-t_n)\sum_{(i,j) \in B} \kappa_{i,j} d_{S^1}([\theta_i], [\theta_{n,j}]) + t_n \sum_{(i,j) \in B} \kappa_{i,j} | m^{x,\mu_n} - m_{2k+1}^{x}(\mu_n) |
\\
< & (1-t_n)d_W(\mu,\mu_n) + t_n \frac{\varepsilon}{2}\\
\leq & d_W(\mu,\mu_n) + \frac{\varepsilon}{2}
\end{align*}
There is mass at most $\frac{d_W(\mu,\mu_n)}{\pi-r}$ remaining, and this mass can be transported arbitrarily at a cost of at most $\frac{\pi}{\pi-r} d_W(\mu,\mu_n)$.
This shows there exists a matching between $J^{x,\mu_n}(t_n)$ and $\widetilde{F}_{2k+1}(\mu_n,t_n)$ with cost at most $ (1+\frac{\pi}{\pi-r}) d_W(\mu,\mu_n) + \frac{\varepsilon}{2}$.
Thus, for all sufficiently large $n$, we have $d_W(J^{x,\mu_n}(t_n), \widetilde{F}_{2k+1}(\mu_n,t_n)) < \varepsilon$.

Therefore, for all sufficiently large $n$, $\widetilde{F}_{2k+1}(\mu_n,t_n)$ is within $\varepsilon$ of $J^{x,\mu_n}(t_n)$ and is thus within $\varepsilon$ of $S$, so it is in $U$.
So for any open neighborhood of $q_{2k-1}(\mu)$ in $q_{2k-1}(W_{2k+1})$, we have shown $q_{2k-1} \circ \widetilde{F}_{2k+1}(\mu_n,t_n)$ is in this neighborhood for all sufficiently large $n$, so $\{G_{2k+1}(\mu_n,t_n)\}_n$ converges to $G_{2k+1}(\mu, t)$. 
This completes the proof that $G_{2k+1}$ is continuous.
\end{proof}

\end{document}